\documentclass[11pt,twoside]{article}

\usepackage[utf8]{inputenc}
\usepackage[T1]{fontenc}

\usepackage{amsmath}
\usepackage{eqnarray}
\usepackage{amsfonts}
\usepackage{amssymb}
\usepackage{amsthm}
\usepackage{bbm}
\usepackage{mathtools}
\usepackage{upgreek}
\usepackage{mathrsfs}
\usepackage{nicefrac}

\usepackage{microtype}
\usepackage{xcolor}
\usepackage{textcomp}

\usepackage{epsf}
\usepackage{epsfig}
\usepackage{graphicx}
\usepackage{psfrag}
\usepackage[skip=0pt]{caption}
\usepackage{subcaption}
\usepackage{wrapfig}

\usepackage{fancyhdr}
\usepackage{fullpage}
\usepackage{pdfpages}
\usepackage{multicol}
\usepackage{changepage}
\usepackage{ragged2e}

\usepackage{booktabs}
\usepackage{multirow}
\usepackage{makecell}
\usepackage{colortbl}
\usepackage{siunitx}

\usepackage{algorithm}
\usepackage{algorithmic}

\usepackage{url}

\definecolor{customyellow}{HTML}{fedf8a}

\newtheorem{assumption}{Assumption}
\newtheorem{lemma}{Lemma}
\newtheorem{example}{Example}
\newtheorem{theorem}{Theorem}
\newtheorem{proposition}{Proposition}
\newtheorem{definition}{Definition}
\newtheorem{observation}{Observation}
\newtheorem{corollary}{Corollary}

\def\ouppi{{\overline{\uppi}}}
\def\oupgamma{{\overline{\upgamma}}}

\begin{document}

\begin{center}

{\bf{\Large{Characterizing Heterogeneous Rates in Finite Mixture\vspace{0.2em} Estimation via Partial Optimal Transport}}}
  
\vspace*{.2in}
{\large{
\begin{tabular}{cccccc}
Dung Le$^{\star}$ & Huy Nguyen$^{\star}$ & Trang Pham & Alessandro Rinaldo & Nhat Ho 
\end{tabular}
}}

\vspace*{.2in}

\begin{tabular}{cc}
Department of Statistics and Data Science\\
The University of Texas at Austin
\end{tabular}

\vspace*{.2in}
\today


\begin{abstract}
    Parameter estimation in finite mixture models can exhibit highly heterogeneous convergence behavior: locally isolated components may be estimated substantially faster than groups of competing components. Existing analyses based on Wasserstein distances typically characterize only the worst-case rate and therefore do not fully capture this local heterogeneity. In this paper, we introduce a Voronoi-based partial optimal transport (VPOT) framework for obtaining refined local and global convergence guarantees for the maximum likelihood estimator of the mixing measure. The key geometric idea is to localize the comparison of two mixing measures to extended Voronoi neighborhoods and use partial optimal transport to accommodate the unequal masses of their local restrictions. Within each neighborhood, the first-order POT discrepancy is raised to a power determined by the number of locally competing atoms, allowing the resulting loss to adapt to the local degree of singularity. Under suitable regularity and strong identifiability conditions, we establish uniform local and global upper bounds for a maximum likelihood estimator under the VPOT loss. These bounds reveal a configuration-dependent form of parameter estimation: less singular local configurations admit faster convergence, whereas the most singular configuration recovers the classical worst-case behavior characterized by Wasserstein-based analyses. We further establish a minimax lower bound showing that the convergence rate for estimating the mixing measure under the VPOT loss is optimal. Our results hold in arbitrary fixed dimension without requiring mixing proportions to be uniformly bounded away from zero or prior knowledge of the true number of mixture components. Overall, VPOT provides a configuration-adaptive framework for capturing heterogeneous parameter-estimation behavior in finite mixture models. 
\end{abstract}
\end{center}
\let\thefootnote\relax\footnotetext{$^\star$Equal contribution.}

\section{Introduction}
\label{sec:introduction}
Finite mixture models \cite{Mclachlan-1988,Lindsay-1995, mclachlan2019finite} constitute a fundamental class of probabilistic models for representing heterogeneous data arising from multiple latent subpopulations. In their most basic form, a finite mixture model assumes that the observed data are generated from a convex combination of component distributions, each corresponding to a distinct latent group. Mixture models are widely used across statistics and machine learning due to their flexibility in capturing multimodality, skewness, and other complex distributional features that cannot be adequately described by a single parametric family. They arise naturally in numerous applications, including clustering \cite{banfield1993model, celeux1995gaussian, fraley2002model, malsiner2016model, scrucca2016mclust}, density estimation \cite{roeder1997practical, li1999mixture, scrucca2019transformation}, machine learning \cite{ hastie1996discriminant, hinton1997modeling, viroli2019deep, figueiredo2002unsupervised, jiang2016variational,han2024fusemoe}, economics \cite{kasahara2009nonparametric, bonhomme2016non, compiani2016using, higgins2023identification}, biology and genomics \cite{yeung2001model, mclachlan2002mixture, pan2002model, silva2023finite}, among others \cite{reynolds1995robust, reynolds2000speaker, kim2023finite,le2024mixture}. 
In these applications, the principal objective is to conduct statistical inference for the mixture parameters, which gives rise to the classical problem of characterizing the optimal convergence rates for parameter estimation in finite mixture models. In particular, let $\mathcal{F}=\{f(x\mid\gamma):x\in\mathcal{X}, \ \gamma\in\Gamma\}$ be a known parametric family of probability density functions with respect to a dominating $\sigma$-finite measure, where the parameter space $\Gamma\subseteq\mathbb{R}^d$, for some $d\geq 1$, is a compact set with non-empty interior, and $\mathcal{X}\subseteq\mathbb{R}^{\bar{d}}$, for some $\bar{d}\geq 1$. Next, let $X_1,X_2,\ldots,X_n$ be an i.i.d. sample drawn from a finite mixture model with $k_*\geq 1$ components, whose probability density function is given by 
\begin{align}
    p_{G_*}(x) =\int f(x\mid\gamma)dG_*(\gamma)= \sum_{\ell=1}^{k_*}\pi^*_{\ell}\, f(x\mid\gamma^*_{\ell}), \quad x\in\mathcal{X},
\end{align}
where $G_* = \sum_{\ell=1}^{k_*} \pi^*_{\ell}\delta_{\gamma^*_{\ell}}$ is a probability mixing measure with $k_*$ atoms $\gamma^*_{\ell}\in\Gamma$. Meanwhile, the mixing proportions $(\pi^*_{\ell})_{\ell}$ are non-negative and sum up to one, that is,  $\sum_{i=1}^{k_*}\pi^*_{\ell}=1$. Then, the goal here is to characterize the convergence rates of estimating mixture parameters $(\gamma^*_{\ell})_{\ell}$.\\



\noindent
\textbf{Related work.} There is a long line of work on the convergence behavior of parameter estimation in finite mixture models. First, Chen et al. \cite{Chen1995} introduced a strong identifiability condition on univariate mixtures under which they established a local minimax rate of order $n^{-1/4}$ for estimating the mixture parameters, where $n$ denotes the sample size. Next, Nguyen \cite{nguyen2016latentmixing} proposed the Wasserstein distance as a natural metric to capture the convergence rates of individual parameters through the associated mixing measure. 
More specifically, for any two equal-mass mixing measures $G=\sum_{i=1}^{k}\pi_i\delta_{\gamma_i}$ and $G'=\sum_{j=1}^{k'}\pi'_j\delta_{\gamma'_j}$, the $r$-Wasserstein distance with the Euclidean norm between $G$ and $G'$ is defined as
\begin{align*}
    W_r(G,G'):=\left(\inf~\sum_{i,j}q_{ij}\|\gamma_i-\gamma'_j\|_2^r\right)^{1/r},
\end{align*}
where the infimum is taken over all couplings $(q_{ij})_{ij}\in[0,1]^{k\times k'}$ such that $\sum_{i=1}^{k}q_{ij}=\pi'_j$, for any $1\leq j\leq k'$, and $\sum_{j=1}^{k'}q_{ij}=\pi_i$, for any $1\leq i\leq k$.
The parameter estimation rates in \cite{nguyen2016latentmixing} were achieved by relating the Wasserstein distances on the
space of mixing measures to the Hellinger distances on the space of mixture distributions.
This framework was later adopted by Ho and Nguyen \cite{Ho-Nguyen-EJS-16, Ho-Nguyen-Ann-16} to establish the convergence rates of the maximum likelihood estimators of the mixture parameters along with corresponding minimax lower bounds under the settings of strong identifiability and weak identifiability of finite mixture models, respectively. It should be noted that the rates derived in these two works were pointwise.
Subsequently, Heinrich and Kahn \cite{heinrich2018} aimed to determine uniform rates of estimating parameters in finite mixtures. They demonstrated that, under some regularity and strong identifiability conditions, around a given mixing distribution $G_0$ with $k_0$ components, the optimal local minimax rate for parameter estimation decreased exponentially when the degree of over-specification, or the number of excess mixture components, $d_0: = k-k_0$ increased, through the Wasserstein bound
\begin{align}
    \label{eq:old_rate}
    \sup_{\substack{G\in\mathcal{G}_{\leq k}(\Gamma),\\ W_{2d_0+1}(G,G_0)<\varepsilon}}\mathbb{E}_{G}[W_{2d_0+1}(\widetilde{G}_n,G)]\lesssim n^{-1/(4d_0+2)},
\end{align}
for some $\epsilon > 0$.
Above, $\mathcal{G}_{\leq k}(\Gamma)$ denotes the set of probability mixing measures on $\Gamma$ with at most $k$ atoms. Additionally, $\widetilde{G}_n$ stands for the minimum distance estimator defined as $\|F(\cdot,\widetilde{G}_n)-F_n\|_{\infty}=\inf_{G\in\mathcal{G}_{\leq k}(\Gamma)}\|F(\cdot,G)-F_n\|_{\infty}$, where $F$ and $F_n$ are the population and empirical distributions, respectively. The inequality~\eqref{eq:old_rate} indicates that the optimal rates for estimating ground-truth parameters $(\gamma^*_{\ell})_{\ell}$ admit the same order of $n^{-1/(4d_0+2)}$. 
This is a limitation of the Wasserstein distances since they can only characterize the worst-case parameter estimation rates, whereas the rates for estimating most individual parameters, particularly those fitted by a single component, should be substantially faster. \\

\noindent
To overcome this issue, Manole and Ho \cite{manole22refined} advocated using a class of loss functions built upon a set of Voronoi cells generated by the support points of the given mixing measure $G_0=\sum_{i=1}^{k_0}\pi_{0i}\delta_{\gamma_{0i}}$, that is, $\mathcal{A}_j(G):=\{1\leq i\leq k:\|\gamma_i-\gamma_{0j}\|_2\leq\|\gamma_i-\gamma_{0\ell}\|_2,\forall \ell\neq j\}$, for all $1\leq j\leq k_0$, for a mixing measure $G=\sum_{i=1}^{k}\pi_{i}\delta_{\gamma_i}$. 
In particular, for a true mixing measure $G_{*}$ with exactly $k_*$ atoms in a small Wasserstein-neighborhood of $G_0$, they captured the heterogeneity of local convergence rates of estimating true parameters, that is, atoms of $G_*$, using the maximum likelihood method. The result of Manole and Ho \cite{manole22refined} can be interpreted heuristically in terms of the local configurations of the fitted and true components. 
Ignoring polylogarithmic factors, smaller cardinalities of $\mathcal{A}_j(\widehat{G}_n)$ and $\mathcal{A}_j(G_*)$ correspond to faster parameter estimation rates, whereas larger cardinalities indicate slower estimation. For instance, suppose that for some $1\leq j\leq k_0$, both cardinalities attain their maximal values,
\begin{align*}
    |\mathcal{A}_j(\widehat{G}_n)|=k-k_0+1, \qquad |\mathcal{A}_j(G_*)|=k_*-k_0+1.
\end{align*}
This corresponds to all redundant atoms of $\widehat{G}_n$ and $G_*$ concentrating around the $j$-th component of $G_0$. In this most crowded local configuration, the convergence rates of fitted parameters near that component are of order $n^{-1/(k+k_*-2k_0+1)}$. At the same time, the parameters associated with the remaining $k_0-1$ components exhibit the parametric rate $n^{-1/2}$. Moreover, this is the only configuration in which the worst-case rate $n^{-1/(k+k_*-2k_0+1)}$ can arise: whenever at least one of the above cardinalities is smaller than its maximum, the corresponding interpretation yields strictly faster rates for all fitted atoms. Note that because $\mathcal{A}_j(\widehat{G}_n)$ is random, these component-wise rates should be understood as an interpretation of their expected loss bound rather than as deterministic convergence rates.
Their results, however, were obtained under three main restrictions: (i) the parameter space was assumed to be one-dimensional to facilitate their derivations, (ii) both the ground-truth mixing measure $G_{*}$ and the reference mixing measure $G_0$ had mixing proportions uniformly bounded away from zero, and (iii) the number of ground-truth atoms $k_*$ was assumed to be known.\\

\noindent
The main goal of this paper is to characterize the heterogeneity of parameter estimation rates in finite mixtures in more general and practical settings. In particular, we consider the parameter space of an arbitrary yet fixed dimension, and do not require mixing proportions to be bounded away from zero nor assume prior knowledge of $k_*$. Additionally, we analyze the widely used maximum likelihood method for parameter estimation. Since the true mixture order $k_*$ is typically unknown in practice, we study the maximum likelihood estimator (MLE) of $G_*$ with order at most $k\geq k_*$, which is given by 
\begin{align}
    \label{eq:MLE}
    \widehat{G}_n\equiv \widehat{G}_n(k)
    :=
    \sum_{j=1}^{\widehat{k}_n}
    \widehat{\pi}_{n,j}\delta_{\widehat{\gamma}_{n,j}}
    \in
    \operatorname*{arg\,max}_{G'\in\mathcal{G}_{\leq k}(\Gamma)}
    \frac{1}{n}\sum_{i=1}^n\log f_{G'}(X_i).
\end{align}
\textbf{Why partial optimal transport?} Towards these goals, we develop a novel approach of Voronoi-based partial optimal transport (VPOT), which allows for a more refined convergence analysis in which two mixing measures are compared locally over extended neighborhoods of the Voronoi cells generated by the reference mixing measure $G_0$. A direct application of the Wasserstein distance, however, becomes problematic after such a localization. Indeed, although two mixing measures $\widehat{G}_n$ and $G_{*}$ have the same total mass globally, their restrictions to a given Voronoi cell generally do not: the two measures may allocate different amounts of probability mass to the same cell. In particular, the standard Wasserstein distance requires its arguments to have equal total masses, and therefore cannot be applied. Instead, we deploy  partial optimal transport (POT) \cite{Figalli2010pot}, which transports only the common mass between the two restricted measures while explicitly penalizing their unmatched mass. This construction allows us to retain the local geometric information captured by Voronoi localization while accommodating discrepancies in mixing weights across cells. More importantly, it remains well-defined even in regimes in which the mixing proportions may vanish or merge -- thus changing the effective order of the mixing measure -- and therefore provides a natural framework for studying uniform parameter estimation rates in finite mixture models.\\  

\noindent
\textbf{Contributions.} Our contributions are twofold and can be summarized as follows. 

\emph{1. A novel Voronoi-based POT (VPOT) definition between mixing measures.} First, we introduce VPOT as a principled framework for capturing the optimal minimax convergence rate of parameter estimation in finite mixture models. Specifically, we partition the parameter space into the Voronoi cells generated by the components of a limiting reference mixing measure and compare the restrictions of two mixing measures within each cell using POT. This construction is particularly well suited to finite mixtures because these restricted measures generally have unequal total masses, in which case the standard Wasserstein distance is not directly applicable. By transporting the common mass between the two restricted mixing measures and explicitly penalizing the difference in their total masses, POT captures both discrepancies in component locations and discrepancies in the aggregate mixing weights within each local neighborhood. Moreover, by choosing the transport order according to the number of components involved in each Voronoi cell, the resulting discrepancy adapts to the local degree of singularity of finite mixtures and, therefore, captures the heterogeneous convergence behavior of different groups of mixture components.

\emph{2. Characterizing heterogeneous parameter estimation rates.}
Our second contribution is to use this VPOT framework to capture the heterogeneity of parameter estimation rates in finite mixtures. We establish upper and lower bounds showing that the estimation rates of individual parameters are governed by the number of locally competing mixture atoms, rather than by the degree of over-specification $d_0$ captured by the Wasserstein distance in equation~\eqref{eq:old_rate}. More specifically, within a neighborhood of a reference mixing measure $G_0$, our results reveal that the estimation rates for true parameters $\gamma^*_j$ around a reference atom $\gamma_{0i}$ depend on the numbers of fitted and true atoms in its Voronoi cell: the smaller numbers yield faster estimation rates, while the larger ones lead to slower rates. 

\begin{itemize}
    \item In the most favorable scenario in which a Voronoi cell generated by an atom of $G_0$ consistently contains one fitted atom and one true atom, the local convergence rate of this fitted parameter admits a parametric order $n^{-1/2}$. 
    \item In the worst case of a singular configuration, where all excess components of $\widehat{G}_n$ and $G_*$ consistently concentrate around the same component of $G_0$, the local convergence rates of these fitted parameters match the classical worst-case rate of order $n^{-1/(4d_0+2)}$ in equation~\eqref{eq:old_rate}. Thus, the classical Wasserstein rate~\eqref{eq:old_rate} arises as an extreme case of a broader spectrum of configuration-dependent rates, while less singular configurations enjoy strictly faster parameter recovery. 
\end{itemize}
Consequently, different groups of mixture components will exhibit different convergence rates, ranging from $n^{-1/2}$ to $n^{-1/(4d_0+2)}$, depending on their local configurations. Importantly, our results hold for parameter spaces of arbitrary finite dimension, thereby extending the refined minimax theory beyond the one-dimensional setting considered in previous works.\\

\noindent
\textbf{Paper organization.} 
The remainder of the paper is organized as follows. Section~\ref{sec:preliminaries} introduces the notation and preliminary results used throughout the paper, including the proposed VPOT loss and the convergence rate of maximum likelihood density estimator. Section~\ref{sec:uniform_bound} presents our main theoretical results on finite mixture estimation: we first establish local and global uniform upper bounds under the VPOT loss and then derive the corresponding minimax lower bound, thereby obtaining the optimal rates of parameter estimation. Next, we streamline the proof of the upper-bound results in Section~\ref{sec:proof_upper_bound} before concluding the paper in Section~\ref{sec:conclusion}. Lastly, additional results and other proofs are deferred to the appendices.



\section{Preliminaries}


\label{sec:preliminaries}
In this section, we first present necessary notation for our analysis and the standard POT framework. We then introduce extended Voronoi cells used to construct the novel VPOT loss for capturing the convergence behavior of parameter estimation. Finally, we state an uniform bound for the $L_1$ density-estimation bound, which is later combined with the VPOT bounds in Section~\ref{sec:uniform_bound} to obtain rates for estimating parameters in finite mixtures.\\

\noindent
\textbf{Notation.} 
For two natural numbers $m,n\in\mathbb{N}$ such that $m< n$, we denote $[n]:=\{1,2,\ldots,n\}$ and $[m,n] := \{m,m+1,\ldots,n\}$. Next, let $|A|$ denote the cardinality of a finite set $A$. For $a,b\in\mathbb{R}$, we define $a\vee b:=\max\{a,b\}$ and $a\wedge b:=\min\{a,b\}$. For any two vectors $a,b \in \mathbb{R}^d$, we denote $(a,b) = (a_1,b_1)\times \ldots\times (a_d,b_d)$ and $[a,b]= [a_1,b_1]\times \ldots\times [a_d,b_d]$. The Euclidean norm on $\mathbb{R}^d$ is denoted by $\|\cdot\|_2$. For nonnegative quantities $a$ and $b$, we write $a\lesssim b$ if $a\leq Cb$, for some constant $C>0$ independent of these quantities. We write $a\gtrsim b$ when $b\lesssim a$ and $a\asymp b$ when both relations hold.
Given a multi-index $\boldsymbol{\beta}=(\beta_1,\ldots,\beta_d)\in\mathbb N^d$, set $|\boldsymbol{\beta}|:=\sum_{j=1}^d \beta_j$.
The corresponding derivative of order $|\boldsymbol{\beta}|$ of a function $f$ is denoted by $ D^{\boldsymbol{\beta}} f(x):= \frac{\partial^{|\boldsymbol{\beta}|}f} {\partial x_1^{\beta_1}\cdots\partial x_{\bar{d}}^{\beta_{\bar{d}}}}$. If a function $f$ is also parameterized by $\gamma$, we also denote $ D_{\gamma}^{\boldsymbol{\beta}} f(x\mid \gamma):= \frac{\partial^{|\boldsymbol{\beta}|}f(x\mid \gamma)} {\partial \gamma_1^{\beta_1}\cdots\partial \gamma_{{d}}^{\beta_{{d}}}}$. For a Lebesgue-integrable measurable function $g:\mathbb{R}^{\bar{d}}\to\mathbb{R}$, let $\|g\|_1:=\int_{\mathbb{R}^{\bar{d}}}|g(x)|\,dx$. Additionally, if $g$ is bounded, we denote $\|g\|_\infty:=\sup_{x\in\mathbb{R}^{\bar{d}}}|g(x)|$. For any two probability density functions $f_1$ and $f_2$ with respect to the Lebesgue measure, we define the Hellinger distance between them as $h(f_1,f_2):=\left(\frac{1}{2}\int_{\mathbb{R}^{\bar{d}}}\left(\sqrt{f_1(x)}-\sqrt{f_2(x)}\right)^2dx\right)^{1/2}$.\\

\noindent
\emph{Mixing measures.} A finite mixing measure is written as $G=\sum_{i=1}^{\ell}\pi_i\delta_{\gamma_i}$, where $\delta$ stands for the Dirac measure, with total mass $m(G):=\sum_{i=1}^{\ell}\pi_i$ and support $\mathcal{S}(G):=\{\gamma_1,\gamma_2,\ldots,\gamma_\ell\}$. We denote $\mathcal{G}_k(\Gamma)$ and $\mathcal{G}_{\leq k}(\Gamma)$ as the classes of probability mixing measures on $\Gamma$ satisfying
$|\mathcal{S}(G)|=k$ and $1\leq |\mathcal{S}(G)|\leq k$, respectively. The mixture density function and the cumulative distribution function induced by a mixing measure $G$ are given by
\[
p_G(x):=\sum_{i=1}^{\ell}\pi_i f(x\mid\gamma_i),
\qquad
F(x \mid G):=\sum_{i=1}^{\ell}\pi_i F(x\mid\gamma_i),
\]
where $F(x\mid\gamma):=\int_{(-\infty,x]}f(t\mid\gamma)\,dt$. We also use the notations $F_G(x)$ and $F(x\mid G)$ interchangeably. The restriction of a mixing measure $G$ to a Borel set $A\subseteq\Gamma$ is defined by $G|_A$, where $G|_A(B):=G(A\cap B)$, for every Borel set $B\subseteq\Gamma$, or equivalently, $G|_A=\sum_{i=1}^{\ell}\pi_i\delta_{\gamma_i}
\mathbf{1}_{\{\gamma_i\in A\}}$. For finite nonnegative measures $\mu$ and $\nu$, we write $\mu\preceq\nu$ if $\mu(B)\leq\nu(B)$, for every Borel set $B\subseteq\Gamma$.


\subsection{(Voronoi-based) Partial Optimal Transport}
\textbf{POT discrepancy.}
For any two mixing measures $G_1 = \sum_{\ell=1}^{k_1}\pi_{1,\ell}\delta_{\gamma_{1,\ell}}$ and $G_2 = \sum_{\ell=1}^{k_2}\pi_{2,\ell}\delta_{\gamma_{2,\ell}}$ that may have different total masses, we quantify the discrepancy between them using the $r$-POT discrepancy \cite{Figalli2010pot, caffarelli2010free}, where $r\in\mathbb{N}$, defined as 
\begin{equation}
\label{eqn:POT_definition}
    \mathsf{POT}_r(G_1,G_2) = \left(\inf_{\substack{\mu_1\preceq G_1,\ \mu_2\preceq G_2\\ m(\mu_1) = m(\mu_2) = \min\{m(G_1), m(G_2)\}}} W_{r}^{r}(\mu_1,\mu_2) + |m(G_1) - m(G_2)|\right)^{1/r}. 
\end{equation}
The optimization in POT selects submeasures of $G_1$ and $G_2$ having the largest possible common mass, namely $\min\{m(G_1),m(G_2)\}$. Hence, all the mass of the smaller measure is transported to a submeasure of the larger one. The term $W_r^r(\mu_1,\mu_2)$ quantifies the transportation cost of this matched mass, whereas $|m(G_1)-m(G_2)|$ accounts for the mass that cannot be matched. In particular, when $m(G_1)=m(G_2)$, no mass is discarded and $\mathsf{POT}_r(G_1,G_2)$ reduces to the usual $r$-Wasserstein distance. In addition, POT can also be represented through a Wasserstein distance between suitably augmented measures \cite{chapel2020pot,le2022mpot}.\\

\noindent
It should be noted that POT is particularly convenient for comparing local restrictions of probability mixing measures. Although two probability measures $G$ and $G'$ have the same total mass globally, their restrictions $G|_A$ and $G'|_A$ to a subset $A\subseteq\Gamma$ generally have different masses. Applying the ordinary Wasserstein distance would therefore require an additional normalization, which would discard the discrepancy in the amount of mass assigned to $A$. In contrast, POT compares the locations of the common mass while simultaneously retaining the difference in local masses. Thus, it captures both discrepancies in component locations and discrepancies in the aggregate mixing weights within each local region. The following proposition formalizes how such local POT comparisons over a cover of $\Gamma$ characterize the underlying mixing measures.



\begin{proposition}
    \label{proposition:local_equal_implies_global_equal}
    Let $G, G' \in \mathcal{G}_{\leq k}(\Gamma)$ be two probability mixing measures and $\{A_i\}^m_{i=1}$ be a Borel cover of the support space $\Gamma$. If for each $i\in[m]$, there exists $q_i\in\mathbb{N}$ such that $\mathsf{POT}_{q_i}(G|_{A_i}, G'|_{A_i}) = 0$, then we have $G=G'$. 
\end{proposition}
\noindent
Proposition~\ref{proposition:local_equal_implies_global_equal} shows that equality of two mixing measures can be characterized through local POT comparisons. Indeed, for each $i\in[m]$, the condition $\mathsf{POT}_{q_i}(G|_{A_i}, G'|_{A_i}) = 0$
implies both equality of the total masses of the two restrictions and zero transportation cost between them; hence, $G|_{A_i}= G'|_{A_i}$.
Since $\{A_i\}^m_{i=1}$ covers the parameter space $\Gamma$, these local equalities determine the measures on the entire support space and therefore imply $G=G'$. Thus, rather than comparing $G$ and $G'$ globally, it suffices to compare their restrictions on a collection of local regions covering $\Gamma$. In the sequel, we construct such regions around the support points of the reference mixing measure $G_0$ using its associated Voronoi cells.\\

\noindent
\textbf{Voronoi cells.}
Let $G_0 = \sum_{i=1}^{k_0}\pi_{0i}\delta_{\gamma_{0i}}$ be a known mixing measure. Then, a Voronoi cell generated by an atom $\gamma_{0i}$ of $G_0$ is defined as 
\begin{align*}
    {\mathcal{V}}^{i}_{G_0} = \{\gamma\in\Gamma:\|\gamma - \gamma_{0i}\|_2 \leq \|\gamma-\gamma_{0j}\|_2,\ \forall j\neq i\}.
\end{align*}
Geometrically, $\mathcal{V}_{G_0}^i$
 consists of all parameter values in $\Gamma$ that are at least as close to $\gamma_{0i}$
 as to any other support point of $G_0$. Thus, the collection $\{\mathcal{V}_{G_0}^i\}_{i=1}^{k_0}$
 provides a natural localization of the parameter space around the atoms of $G_0$
 and covers $\Gamma$, up to overlaps along cell boundaries. However, direct restriction to a fixed cell is unstable near its boundary, as an arbitrarily small perturbation can move an atom across the boundary and change whether it is retained in the restricted measure. For example, an atom of $G$ and a nearby atom of $G'$ may lie on opposite sides of the boundary. In that case, the cell contains one atom but not the other, so the two restricted measures can have different masses even when the atoms are arbitrarily close. To avoid this issue, we introduce a finite family of nested extensions of each Voronoi cell below.\\

\noindent
\textbf{Extension of Voronoi cells. }
For $j\in[2k]$, we define an \textit{$(\delta,j)$-extension of a Voronoi cell $\mathcal{V}_{G_0}^i$}, for $i\in[k_0]$, as follows:
\begin{equation*}
    \mathcal{V}^{i,j}_{G_0,\delta} = \left\{\gamma \in \Gamma: \inf_{\vartheta\in \mathcal{V}^{i}_{G_0}}\|\vartheta-\gamma\|_2 < \dfrac{j\delta}{2k}\right\}. 
\end{equation*} 
These $2k$ nested extensions provide enough candidate boundaries to ensure that at least one of them contains no limiting cluster location arising in the sequence arguments below. Indeed, consider two sequences $(G_n)$ and $(G_n')$ in $\mathcal{G}_{\leq k}(\Gamma)$ that converge in $W_1$ to the same mixing measure. By compactness of $\Gamma$, after extracting a common subsequence and relabeling the atoms, we may suppose that all atom locations converge. For every $n$, the combined support $\mathcal{S}(G_n)\cup\mathcal{S}(G_n')$ contains at most $2k$ points. Since the common limiting measure has nonempty support, at least one limiting location is shared by the two sequences. Therefore, their atoms have at most $2k-1$ distinct limiting cluster locations. For each fixed $i\in[k_0]$, the boundaries $\{\partial\mathcal{V}_{G_0,\delta}^{i,j}:j\in[2k]\}$ are pairwise disjoint because they correspond to the distinct extension radii $j\delta/(2k)$. Thus, each limiting cluster location can lie on at most one boundary, and at least one of the $2k$ boundaries contains no such location. Let $j_i\in[2k]$ denote an index corresponding to such a boundary. Since no limiting cluster location lies on $\partial\mathcal{V}_{G_0,\delta}^{i,j_i}$, each location has a neighborhood contained either in $\mathcal{V}_{G_0,\delta}^{i,j_i}$ or in its complement. Hence, for all sufficiently large $n$, every corresponding cluster lies entirely inside or outside the selected extended cell. Additionally, as $\delta\to0$, each extended Voronoi cell $\mathcal{V}_{G_0,\delta}^{i,j}$ converges to the original cell $\mathcal{V}_{G_0}^{i}$, thus serving as a natural and well-justified representative of the original cell. The Voronoi cells and their extensions are illustrated in Figure \ref{fig:voronoi-extension}.
\begin{figure}[t]
\centering
\includegraphics[width=0.95\linewidth]{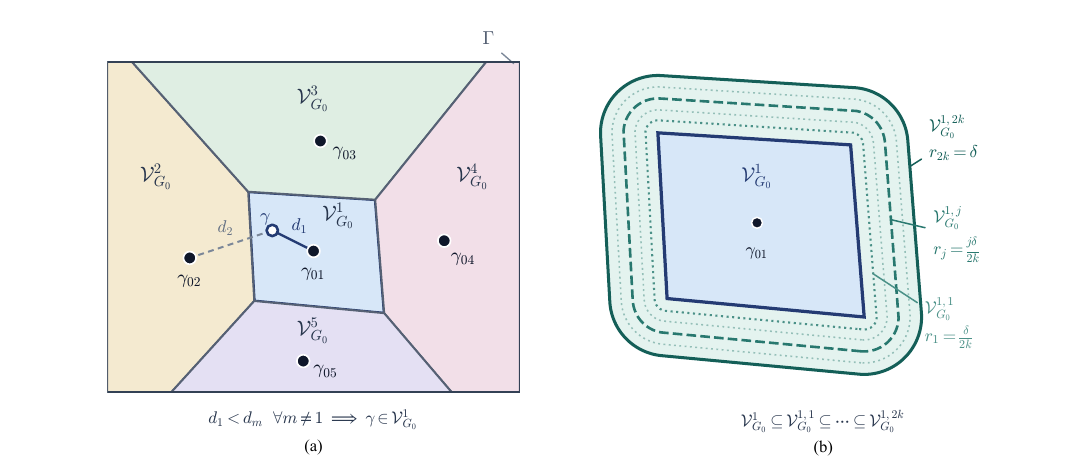}
\caption{Illustration of Voronoi cells and their $(\delta,j)$-extensions.
    (a) For the illustrated point $\gamma$, let
    $d_m=\|\gamma-\gamma_{0m}\|_2$. The shortest segment is the one to
    $\gamma_{01}$, so $d_1<d_m$ for every $m\ne1$ and
    $\gamma\in \mathcal{V}_{G_0}^{1}$. Assigning every point in $\Gamma$ to its nearest
    atom yields the Voronoi cells; a shared boundary occurs where the
    competing distances are equal.
    (b) For the highlighted cell $\mathcal{V}_{G_0}^{1}$, the extension
    $\mathcal{V}_{G_0}^{1,j}$ contains the points whose Euclidean distance to
    $\mathcal{V}_{G_0}^{1}$ is less than $r_j=\frac{j\delta}{2k}$, for
    $j=1,\ldots,2k$. Consequently, the extensions are nested and the
    largest one has radius $r_{2k}=\delta$. }
\label{fig:voronoi-extension}
\end{figure}

\noindent
\textbf{Voronoi-based POT (VPOT).} 
To quantify the discrepancy between two mixing measures \(G\) and \(G'\) using extended Voronoi cells, we introduce the Voronoi-based partial optimal transport (VPOT) discrepancy. Within each cell, we compare the restrictions of \(G\) and \(G'\) using POT, which accommodates restricted measures with different total masses. More precisely, the VPOT loss between \(G\) and \(G'\), relative to a reference mixing measure \(G_0\), is defined as
\begin{align}
\label{eqn:dung_defintion_of_lower_VPOT}
    \mathsf{VPOT}_{G_0,\delta}(G,G') = \sum_{i = 1}^{k_0}\inf_{1\leq j\leq 2k} \mathsf{POT}^{r_{i,j}}_{1}\left(G|_{\mathcal{V}^{i,j}_{G_0,\delta}},G'|_{\mathcal{V}^{i,j}_{G_0,\delta}}\right),
\end{align}
where $r_{i,j} = \max\{|\mathcal{S}(G)\cap \mathcal{V}^{i,j}_{G_0,\delta}|+ |\mathcal{S}(G')\cap \mathcal{V}^{i,j}_{G_0,\delta}|-1,1\}$ denotes the total number of atoms from $G$ and $G'$ that locally compete within the Voronoi cell $\mathcal{V}^{i,j}_{G_0,\delta}$, minus one, truncate below at one. When $|\mathcal{S}(G)\cap \mathcal{V}^{i,j}_{G_0,\delta}| = |\mathcal{S}(G')\cap \mathcal{V}^{i,j}_{G_0,\delta}| = 0$, the restrictions $G|_{\mathcal{V}^{i,j}_{G_0,\delta}}$ and $G'|_{\mathcal{V}^{i,j}_{G_0,\delta}}$ become zero measures. In this case, $\mathsf{POT}^{r_{i,j}}_{1}\left(G|_{\mathcal{V}^{i,j}_{G_0,\delta}},G'|_{\mathcal{V}^{i,j}_{G_0,\delta}}\right) = 0$. In addition, when  $|\mathcal{S}(G)\cap \mathcal{V}^{i,j}_{G_0,\delta}| =1 $ and $|\mathcal{S}(G')\cap \mathcal{V}^{i,j}_{G_0,\delta}| = 0$ or  $|\mathcal{S}(G)\cap \mathcal{V}^{i,j}_{G_0,\delta}| = 0$ and $|\mathcal{S}(G')\cap \mathcal{V}^{i,j}_{G_0,\delta}| = 1$, then $r_{i,j} = 1$, leading to $\mathsf{POT}_1^{r_{i,j}}(G|_{\mathcal{V}^{i,j}_{G_0,\delta}},G'|_{\mathcal{V}^{i,j}_{G_0,\delta}}) = |m(G|_{\mathcal{V}^{i,j}_{G_0,\delta}}) - m(G'|_{\mathcal{V}^{i,j}_{G_0,\delta}})|$, which measures the mass discrepancy between two restrictions of measures. \\

\noindent
We next establish several fundamental properties of the proposed Voronoi-based POT discrepancy. For a fixed reference measure $G_0$, $\mathsf{VPOT}_{G_0,\delta}(G_1,G_2)$ is nonnegative and symmetric in $G_1$ and $G_2$. In addition, it separates mixing measures in the sense that $\mathsf{VPOT}_{G_0,\delta}(G_1,G_2)=0$ if and only if $G_1= G_2$. 
Indeed, if $G_1=G_2$, every term in the definition of VPOT vanishes. Conversely, suppose that
$\mathsf{VPOT}_{G_0,\delta}(G_1,G_2)=0$. Since all summands in the definition of VPOT are nonnegative, each infimum must be zero. Moreover, since the infimum is taken over finitely many indices,
for every $i\in[k_0]$ there exists $j_i\in[2k]$ such that
\[
    \mathsf{POT}_{r_{i,j_i}}\!\left(
        \left.G_1\right|_{\mathcal{V}_{G_0,\delta}^{i,j_i}},
        \left.G_2\right|_{\mathcal{V}_{G_0,\delta}^{i,j_i}}
    \right)=0.
\]
Each selected extension contains its corresponding Voronoi cell. As a result, the collection $\{\mathcal{V}_{G_0,\delta}^{i,j_i}:i\in[k_0]\}$ still covers $\Gamma$.
Proposition~\ref{proposition:local_equal_implies_global_equal} then gives
$G_1=G_2$. 
\subsection{Density Estimation Rate}

In the preceding subsection, we introduced the necessary notation and the proposed VPOT loss for comparing mixing measures. We now proceed to study the convergence rate of the maximum likelihood density estimator. Before presenting the density-estimation bound, we state two regularity conditions on the component family under which the estimation error can be uniformly controlled over the entire model class.\\

\noindent
\textbf{Assumption $(A)$} \textit{(Uniform Lipschitz continuity):}
There exists a constant $L>0$ such that, for all
$\gamma,\gamma'\in\Gamma$,
\begin{align*}
    \sup_{x\in\mathbb{R}^{\bar{d}}}
    \left|f(x\mid\gamma)-f(x\mid\gamma')\right|
    \leq L\|\gamma-\gamma'\|_2.
\end{align*}
\textbf{Assumption $(B)$} \textit{(Uniformly non-heavy tail):} There exist two positive constants $c$ and $\zeta$ such that for each $\gamma \in \Gamma$, we have 
    \begin{equation*}
        f(x\mid \gamma) \leq c\cdot\min\left(1,\dfrac{1}{\|x\|^{\bar{d}+\zeta}}\right). 
    \end{equation*}
These assumptions are mild and are satisfied by many standard parametric families, including multivariate Gaussian and Student-$t$ distributions. Indeed, Assumption~$(A)$ follows from a uniform bound on the
derivatives of the density with respect to its parameter. More precisely, if $\Gamma$ is convex and $\sup_{x\in\mathbb{R}^{\bar{d}}}\sup_{\gamma\in\Gamma}
    \left\|\nabla_{\gamma}f(x\mid\gamma)\right\|_2<\infty$,
then the mean value theorem yields $\sup_{x\in\mathbb{R}^{\bar{d}}}
    \left|f(x\mid\gamma)-f(x\mid\gamma')\right|
    \leq L\|\gamma-\gamma'\|_2$, for some constant $L>0$.
 Assumption~$(B)$ is also readily verified. It is sufficient that the density family is uniformly bounded and, for some $\zeta>0$, $\sup_{\gamma\in\Gamma}
    \sup_{\|x\|_2\geq 1}
    \|x\|_2^{\bar{d}+\zeta}f(x\mid\gamma)<\infty$. The uniform boundedness controls the densities on bounded subsets of
$\mathbb{R}^{\bar{d}}$, while the second condition controls their tail behavior. In particular, Assumption~$(B)$ accommodates polynomially decaying
heavy-tailed distributions, including multivariate Student distributions,
as well as distributions with faster exponential or Gaussian decay.
Polynomially decaying and regularly varying distribution families are
discussed extensively in~\cite{resnick2007heavy}.
\begin{proposition}
    \label{prop:MLE_estimation}
     Under Assumptions (A) and (B) and given the MLE $\widehat{G}_n$ defined in equation~\eqref{eq:MLE}, there exists a universal constant $C>0$ such that
    \begin{equation}
    \label{prop:model_convergence}
         \sup_{G\in \mathcal{G}_{\leq k}(\Gamma)} \mathbb{E}_{G}[\|p_{\widehat{G}_n} - p_{G}\|_1] \leq C(\log(n)/n)^{1/2}. 
    \end{equation}
\end{proposition}
\noindent
The proof of Proposition 2 is deferred to Appendix~\ref{proof:density_estimation_rate}. This result establishes the uniform $L_1$-consistency of the MLE density
estimator, with the worst-case expected $L_1$ error over $G\in\mathcal{G}_{\leq k}(\Gamma)$ decreasing at the rate $(\log n/n)^{1/2}$. Therefore, ignoring the logarithmic factor $\sqrt{\log n}$, the estimator exhibits the usual parametric order $n^{-1/2}$. In particular, the bound provides a common worst-case guarantee for density estimation throughout $\mathcal G_{\leq k}(\Gamma)$, since its constant does not depend on the data-generating mixing measure. With the density estimation error controlled at this rate, the remaining task for parameter estimation is to quantify how a discrepancy between mixing measures is reflected in the corresponding mixture densities. In Section \ref{sec:uniform_bound}, we establish this connection through a comparison between the $L_1$ density loss and the proposed VPOT loss. This comparison makes the density estimation bound above directly useful for studying the underlying mixing measure: once the $L_1$ discrepancy is controlled, the corresponding VPOT discrepancy can also be controlled. Because VPOT is constructed to account for the local configuration of mixture components, this connection allows the uniform density guarantee to yield more refined parameter estimation rates that reflect the local configuration of the components.

\section{Uniform Bounds for Finite Mixture Estimation}
\label{sec:uniform_bound}


\subsection{Uniform Upper Bound}
\label{sec:uniform_upper_bound}

In this subsection, we establish uniform upper bounds for parameter estimation in finite mixtures under the proposed VPOT loss. Our analysis relies on comparing discrepancies between the induced mixture distributions with discrepancies between their underlying mixing measures. Since the latter comparison is governed by the local behavior of the component distributions as their parameters approach one another, we require suitable smoothness, identifiability, and continuity conditions on the family $\{f(\cdot\mid\gamma):\gamma\in\Gamma\}$. 
We summarize the required regularity conditions in the following assumption.\\

\noindent
\textbf{Assumption $C(p)$.} We say that the family of density functions $\{f(\cdot\mid \gamma):\gamma \in \Gamma\}$ satisfies \emph{Assumption $C(p)$} if it meets the following conditions: 
    \begin{enumerate}
        \item Let $x \in \mathbb{R}^{\bar{d}} \mapsto F(x\mid \gamma) = \int_{(-\infty,x)}f(t\mid \gamma)\,dt:= \int_{-\infty}^{x_{\bar{d}}}\cdots\int_{-\infty}^{x_1}f(t\mid \gamma)\,dt_1\cdots dt_{\bar{d}}$ be the cumulative distribution function, its derivative $D_{\gamma}^{\boldsymbol{\alpha}}F(\cdot\mid \gamma)$ exists for every multi-index $\boldsymbol{\alpha}$ satisfying
        $0\le |\boldsymbol{\alpha}|\le p$.
        \item ($p$-strong identifiability) The family $\{F(\cdot\mid\gamma):\gamma\in\mathbb{R}^d\}$ is strongly identifiable up to order $p$. That is, for any set of $
        \ell$ distinct points $\gamma_1,\ldots,\gamma_{\ell} \in \Gamma$, the identity
    \begin{equation*}
\left\|\sum_{\boldsymbol{\alpha}=\boldsymbol{0},|\boldsymbol{\alpha}|\leq p }\sum_{1\leq j \leq \ell}c_{\boldsymbol{\alpha},j}D_{\gamma}^{\boldsymbol{\alpha}}F(\cdot\mid\gamma_j)\right\|_{\infty} = 0
    \end{equation*}
    implies that $c_{\boldsymbol{\alpha},j} = 0$ for all $\boldsymbol{\alpha},j$. 
    \item There exists a uniform modulus $\omega:\mathbb{R}^d \to \mathbb{R}$ such that $\lim_{h\to 0} \omega(h) = 0 $ and, for any multi-index $\boldsymbol{\alpha}$ with $|\boldsymbol{\alpha}| = p$, 
    \begin{equation}
\label{eqn:assumption_for_upper_bound_uniform_modulus}
        \sup_{x \in \mathbb{R}^{\bar{d}}}|D_{\gamma}^{\boldsymbol{\alpha}}F(x\mid\gamma) - D_{\gamma}^{\boldsymbol{\alpha}}F(x\mid \gamma')| \leq \omega(\gamma - \gamma'), \quad \gamma, \gamma' \in \Gamma. 
    \end{equation}
    \end{enumerate}
Assumption~$C(p)$ encompasses various regularity conditions required for our convergence analysis: smoothness, strong identifiability, and continuity of higher order derivatives. The first condition ensures that the distribution function is sufficiently smooth with respect to its parameter, so that perturbations of nearby mixture components can be characterized through Taylor expansions up to order $p$. The second condition imposes $p$-th order strong identifiability, requiring the derivatives of the component distribution functions evaluated at distinct parameter values to be linearly independent. This condition prevents nontrivial perturbations of the mixing measure from being completely canceled at the distribution level and is therefore crucial for recovering parameter discrepancies from discrepancies between the induced mixture distributions. The third condition requires the highest-order derivatives to vary uniformly continuously with their parameters, which ensures uniform control of the Taylor remainder as parameters approach one another. 
A sufficient condition is that \(F(\cdot\mid\gamma)\) possesses partial derivatives of order \(p+1\) and that
$ \sup_{x\in\mathbb R^{\bar{d}}} \left|D^{\boldsymbol\alpha}_{\gamma}F(x\mid\gamma)\right| <\infty $
for every multi-index \(\boldsymbol\alpha\) satisfying \(|\boldsymbol\alpha|=p+1\).
Taken together, these assumptions ensure that the behavior of the mixture distribution faithfully reflects perturbations of the underlying mixing measure, uniformly over configurations in which mixture components may become arbitrarily close, and thereby provide the key regularity conditions for deriving the VPOT-based parameter estimation bounds in this section.\\

\noindent
Next, we characterize families of probability distributions that satisfy Assumption $C(p)$ in the following proposition.
\begin{proposition}
\label{dung:prop_strong_identifiability}
Consider the location family of densities $\{f(x\mid\gamma)=f_0(x-\gamma):x,\gamma\in\mathbb{R}^d\}$,
and let $F(x\mid\gamma)$ be their corresponding cumulative distribution function. Suppose that $f_0\in C^{p+1}(\mathbb{R}^d)$, where $p\geq 1$, and $D_{\gamma}^{\boldsymbol{\alpha}} f_0\in L^1(\mathbb{R}^d)$, for every $\boldsymbol{\alpha}\in\mathbb{N}^d$ satisfying $|\boldsymbol{\alpha}|\leq p$. Then, the family $\left\{ F(\cdot\mid\gamma):\gamma\in\mathbb{R}^d\right\}$
is strongly identifiable up to order $p$. 
In addition, if for each set $S$
and a multi-index $\boldsymbol{\beta}$ such that $|S|+|\boldsymbol{\beta}| = p+1$, the integral 
\begin{equation}
\label{eqn:dung_equivalent_form_for_continuous_modulus}
    \sup_{z\in\mathbb{R}^{|S|}}
    \int_{\mathbb{R}^{d-|S|}}
    \left|D^{\boldsymbol{\beta}} f_0(z,y)\right|
    \,dy
    <\infty,
\end{equation}
then it follows that $\sup_{x \in \mathbb{R}^d}|D_{\gamma}^{\boldsymbol{\alpha}}F(x\mid \gamma)| < \infty$, for all $|\boldsymbol{\alpha}|=p+1$. 
\end{proposition}
\noindent
The proof of Proposition \ref{dung:prop_strong_identifiability} is deferred to Appendix~\ref{sec:matrix_utils_separation_lemmma}. This proposition verifies that a broad class of multivariate location
families satisfies the main structural conditions used in our analysis.
Strong identifiability rules out nontrivial cancellations among shifted CDFs
and their parameter derivatives up to order $p$. Furthermore, condition~\eqref{eqn:dung_equivalent_form_for_continuous_modulus} guarantees a uniform bound
on the derivatives of the CDF of order $p+1$, which controls the Taylor
remainder and yields the uniform modulus condition in equation~\eqref{eqn:dung_equivalent_form_for_continuous_modulus}.
Therefore, the proposition provides readily verifiable sufficient conditions
for applying our general estimation theory to multivariate location mixtures.

\begin{example}
\label{example:dung_upper_bound}
\begin{enumerate}
    \item  \textbf{Multivariate Gaussian location family.} Let $\Sigma\in\mathbb{R}^{d\times d}$ be a positive definite matrix and 
$$
    f_0(x)
    =
    \frac{1}{(2\pi)^{d/2}|\Sigma|^{1/2}}
    \exp\left(-\frac{1}{2}x^\top\Sigma^{-1}x\right), \qquad x\in\mathbb{R}^d.
$$
For every multi-index $\boldsymbol{\beta}$, there exists a polynomial $P_{\boldsymbol{\beta}}$ such that $D^{\boldsymbol{\beta}} f_0(x)=P_{\boldsymbol{\beta}}(x)f_0(x)$. Consequently, $h\in C^\infty(\mathbb{R}^d)$ and $D^{\boldsymbol{\beta}} f_0\in L^1(\mathbb{R}^d)$, for every multi-index $\boldsymbol{\beta}$. Moreover, the Gaussian decay implies that,
for every subset $S\subseteq[d]$, 
\[
    \sup_{z\in\mathbb{R}^{|S|}}
    \int_{\mathbb{R}^{d-|S|}}
    \left|D^{\boldsymbol{\beta}} f_0(z,y)\right|
    \,dy
    <\infty.
\]
Thus, condition \eqref{eqn:dung_equivalent_form_for_continuous_modulus} holds. In addition, the conditions that $f_0 \in C^p(\mathbb{R}^d)$ and $D^{\boldsymbol{\alpha}}f_0\in L^1(\mathbb{R}^d)$ are also satisfied.
Therefore, the multivariate Gaussian location family satisfies Assumption $C(p)$, for every finite order $p$.

\item \textbf{Elliptical multivariate Student-$t$ location family.} Let $\nu>0$, $\Sigma\in\mathbb{R}^{d\times d}$ be a positive definite matrix and
\[
    f_0(x)
    =
    \frac{\Gamma\left((\nu+d)/2\right)}
    {\Gamma\left(\nu/2\right)(\nu\pi)^{d/2}|\Sigma|^{1/2}}
    \left(
        1+\frac{x^\top\Sigma^{-1}x}{\nu}
    \right)^{-(\nu+d)/2}, \qquad x\in\mathbb{R}^d.
\]
This density belongs to $C^\infty(\mathbb{R}^d)$. For every multi-index
$\boldsymbol{\beta}$, there exists a finite constant $C_{\boldsymbol{\beta}}$ such that
\[
    \left|D^{\boldsymbol{\beta}} f_0(x)\right|
    \leq
    C_{\boldsymbol{\beta}}(1+\|x\|)^{-(\nu+d+|\boldsymbol{\beta}|)}.
\]
Hence, $D^{\boldsymbol{\beta}} f_0\in L^1(\mathbb{R}^d)$,
for every multi-index $\boldsymbol{\beta}$. If $S\subseteq[d]$ and $s=|S|\geq 1$, then
\[
    \int_{\mathbb{R}^{d-s}}
    \left|D^{\boldsymbol{\beta}} f_0(z,y)\right|
    \,{d}y
    \leq
    C_{\boldsymbol{\beta},S}
    (1+\|z\|)^{-(\nu+s+|\boldsymbol{\beta}|)}
    \leq C_{\boldsymbol{\beta},S}.
\]
Therefore, condition \eqref{eqn:dung_equivalent_form_for_continuous_modulus} holds. In addition, the conditions that $f_0 \in C^p(\mathbb{R}^d)$ and $D^{\boldsymbol{\alpha}}f_0\in L^1(\mathbb{R}^d)$ are also provably satisfied.  Thus, the multivariate Student-$t$ location family satisfies Assumption $C(p)$, for every finite order $p$.
\end{enumerate}
\end{example}
\noindent
Recall that Proposition~\ref{prop:MLE_estimation} provides a uniform parametric-rate bound, up to a logarithmic factor, for estimating the mixture density in the $L_1$ distance. To translate this density-level guarantee into a corresponding convergence rate for the underlying mixing measure, it remains to establish a quantitative relationship between discrepancies at the density and parameter levels. More specifically, we need to show that the $L_1$ distance between two mixture densities uniformly dominates the proposed VPOT loss between their mixing measures. Such a lower bound rules out the possibility that two mixing measures are substantially separated under our parameter loss while inducing nearly indistinguishable mixture densities. This motivates Proposition~\ref{prop:L_1_distance_greater_than_POT}, where we establish local and global inequalities linking $\|f_{G}-f_{G_*}\|_1$ to the corresponding VPOT between two mixing measures $G$ and $G_*$. Combining these inequalities with the density estimation bound in Proposition~\ref{prop:MLE_estimation} then yields the desired upper bounds for parameter estimation.
\begin{proposition}
\label{prop:L_1_distance_greater_than_POT}
Suppose that the family of density functions $\{f(\cdot\mid\gamma):\gamma\in \Gamma\}$ satisfies Assumption $C(2k)$. Let $\delta>0$, $G_0 \in \mathcal{G}_{k_0}(\Gamma)$ be a given mixing measure and $\mathcal{B}_{W_1}(G_0,\varepsilon) = \{G \in \mathcal{G}_{\leq k}(\Gamma): W_1(G,G_0)<\varepsilon\}$. Then, there exists $\varepsilon > 0$ such that  
\begin{equation}
\label{eqn:L_1_distance_greater_than_D_locally}
        \inf_{G,G_* \in \mathcal{B}_{W_1}(G_0,\varepsilon)} \|p_G-p_{G_*}\|_1/  \mathsf{VPOT}_{G_0,\delta}(G,G_*) > 0,
    \end{equation}
and more globally, 
\begin{equation}
    \label{eqn:L_1_distance_greater_than_D_globally}
    \inf_{G,G'\in \mathcal{G}_{\leq k}(\Gamma)}\|p_G-p_{G'}\|_1/ \mathsf{VPOT}_{G,\delta}(G,G') > 0. 
\end{equation}
\end{proposition}
\noindent
The proof of Proposition \ref{prop:L_1_distance_greater_than_POT} is given in Section \ref{sec:proof_upper_bound}. Proposition~\ref{prop:L_1_distance_greater_than_POT} provides the key bridge between density estimation and parameter estimation under the VPOT discrepancy. In particular, part~(i) shows that, within a sufficiently small neighborhood of a fixed mixing measure $G_0$, the $L_1$ distance between two mixture densities uniformly controls their discrepancy $\mathsf{VPOT}_{G_0,\delta}$, while part~(ii) establishes a global relation over the entire class $\mathcal{G}_{\leq k}(\Gamma)$. These inequalities ensure that convergence at the density level can be transferred directly to convergence of the corresponding mixing measures under our proposed loss. Consequently, these inequalities allow us to convert the
 convergence of the estimated mixture density into the convergence of the corresponding mixing measure under our VPOT loss, which we exhibit in the following theorem.

\begin{theorem}
\label{thm:upper_bound}
    Let $\delta>0$ and $G_0 \in \mathcal{G}_{k_0}(\Gamma)$ be a given mixing measure. Under Assumptions $(A)$, $(B)$, and $C(2k)$, there exists $\epsilon > 0$ such that  
    \begin{equation}
    \label{eqn:main_result_local_upper_bound}
        \sup_{\substack{G\in \mathcal{G}_{\leq k}(\Gamma)\\ W_1(G,G_0) < \varepsilon}} \mathbb{E}_{G}[\mathsf{VPOT}_{G_0,\delta}(\widehat{G}_n,G)] \lesssim (\log(n)/n)^{1/2},
    \end{equation}
    and more globally, 
    \begin{equation}
    \label{eqn:main_result_global_upper_bound}
        \sup_{G\in \mathcal{G}_{\leq k}(\Gamma)} \mathbb{E}_{G}[\mathsf{VPOT}_{G,\delta}(\widehat{G}_n,G)] \lesssim (\log(n)/n)^{1/2}. 
    \end{equation}
\end{theorem}
\noindent
The proof of Theorem~\ref{thm:upper_bound} can be found in Section~\ref{sec:proof_upper_bound}. From the above theorem, we observe that under the VPOT, the MLE $\widehat{G}_n$ converges to the ground-truth $G_*$ at the uniform rate of order $n^{-1/2}$, up to a logarithmic factor. However, due to the structure of the VPOT loss, this result further implies heterogeneous local and global convergence rates of fitted parameters compared to those in \cite{heinrich2018}. In the discussion that follows, all convergence rates are stated up to logarithmic factors, which are suppressed for notational simplicity.\\

\noindent
\emph{(i) Local rates.} 
A few remarks regarding this rate are in order. Recall that the loss \(\mathsf{VPOT}_{G_0,\delta}\) is defined using extended Voronoi cells $\mathcal{V}^{i,j}_{G_0,\delta}$. Furthermore, when \(\delta\) is sufficiently small, each extended cell differs from its original Voronoi cell $\mathcal{V}_{G_0}^i$ only by a small neighborhood of the boundary. We therefore use the numbers of atoms in the original Voronoi cells as natural representatives of the corresponding counts in the extended-cell construction. Accordingly, the rates below should be understood as representative local rates. 
\begin{itemize}
    \item  The above local rate is determined by the number of components competing within Voronoi cells with respect to $G_0$ rather than by the number of excess components $d_0=k-k_0$ as in \cite{heinrich2018}. As additional atoms from either
$\widehat G_n$ or $G_*$ concentrate in the same Voronoi cell $\mathcal{V}_{G_0}^i$, the exponent
$r_i$ increases and the local rate becomes slower, reflecting the
higher-order cancellations among nearby mixture components. 
    \item If a Voronoi cell $\mathcal{V}_{G_0}^i$ contains exactly one atom from
$\widehat G_n$ and one atom from $G_*$, then $s_i=t_i=1$ and $r_i=1$,
so the mixture components $\gamma^*_j\in\mathcal{V}_{G_0}^i(G_*)$ enjoy the standard parametric estimation rate $n^{-1/2}$. 
    \item At the other
extreme, the worst possible local configuration occurs when all excess atoms of both mixing measures concentrate around the same
support point of $G_0$, that is,
\[
    s_i=t_i=d_0+1,
    \qquad\text{and hence}\qquad
    r_i=2d_0+1.
\]
The resulting local rate in this case is given by
\[
    n^{-1/\{2(2d_0+1)\}}
    =
    n^{-1/(4d_0+2)},
\]
which recovers the local convergence rate of parameter estimation from \cite{heinrich2018}.
Thus, this worst-case rate arises only from the most singular local
configuration: whenever $s_i+t_i-1<2d_0+1$, the local rates turn out to be strictly faster. 
This illustrates how the
VPOT loss characterizes the heterogeneous convergence behavior of parameter estimation that cannot be captured by the Wasserstein distance used in \cite{heinrich2018}.
    \item We illustrate the above two extreme scenarios in Figure~\ref{fig:local_config}(a). In particular, we set $k_0=5$, which means that there are a total of 5 Voronoi cells generated by 5 atoms of $G_0$. In addition, we also assume that each of $\widehat{G}_n$ and $G_*$ has $k=10$ atoms. It can be seen that each of the four cells $\mathcal{V}_{G_0}^{i}$, for $i\in[2,5]$, has exactly one atom of $\widehat{G}_n$ and another one of $G_*$. Thus, the local convergence rates of estimating true parameters $\gamma^*_j$ in these cells are of parametric order $n^{-1/2}$. Meanwhile, the remaining six atoms of these measures lie in the cell $\mathcal{V}_{G_0}^{1}$. Consequently, true parameters $\gamma^*_j$ in this cell admit significantly slower estimation rates of order $n^{-1/2(6+6-1)}=n^{-1/22}$.
    \item Suppose that $k$ is divisible by $k_0$ and the atoms of $\widehat{G}_n$ and $G_*$ are uniformly distributed to Voronoi cells $\mathcal{V}_{G_0}^i$, that is, $s_i=t_i=k/k_0$, for all $i\in[k_0]$. Then, the local convergence rates of estimating true parameters become homogeneous, standing at the order of $n^{-1/(4k/k_0-2)}$. For example, in Figure~\ref{fig:local_config}(b) where we set $k_0=5$ and $k=10$, each Voronoi cell $\mathcal{V}_{G_0}^i$ has exactly two atoms of $\widehat{G}_n$ and two atoms of $G_*$. Thus, the estimation rates for true parameters $\gamma^*_j$ in these cells have the same order of $n^{-1/2(2+2-1)}=n^{-1/6}$.
\end{itemize}

\begin{figure*}[ht]
\centering
\includegraphics[width=\textwidth]{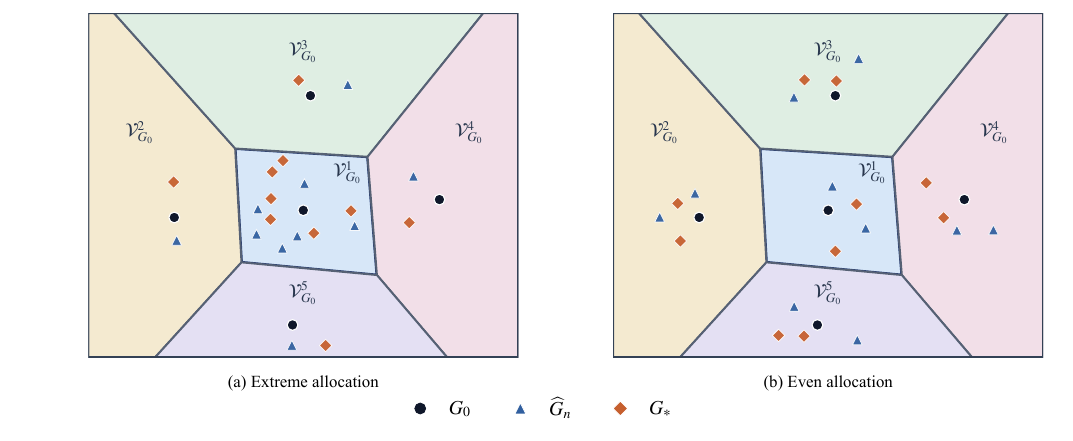}
\caption{Illustration of even and extreme allocations of the atoms of $\widehat{G}_n$ and $G_*$ among the Voronoi cells generated by $G_0$, with $k_0=5$ and $k=10$. Black circles mark the atoms of $G_0$ that generate the Voronoi cells, whereas blue triangles and orange diamonds mark the atoms of $\widehat{G}_n$ and $G_*$, respectively.
\textit{(a) Extreme allocation:} each of the four cells
$\mathcal{V}_{G_0}^{i}$, for $i\in[2,5]$, contains one atom of
$\widehat{G}_n$ and one atom of $G_*$, while the remaining six atoms of each
measure lie in $\mathcal{V}_{G_0}^{1}$. \textit{(b) Even allocation:} each cell contains two atoms of $\widehat{G}_n$ and two atoms of $G_*$.
}
\label{fig:local_config}
\end{figure*}
\noindent
\emph{(ii) Global rates.} The global bound in equation~\eqref{eqn:main_result_global_upper_bound}
admits a similar interpretation, with an important distinction from the
local bound. Rather than
restricting the true mixing measure $G_*$ to lie in a sufficiently small
neighborhood of $G_0$, the global result holds over the entire
space $\mathcal{G}_{\leq k}(\Gamma)$. Accordingly, the Voronoi geometry
underlying the loss is now generated adaptively by $G_*$ itself. More
specifically, for each support point $\gamma_i^*$ of $G_*$ and a
corresponding Voronoi cell $\mathcal{V}_{G_*}^{i}$, let
\[
    s_{i}
    :=
    \big|
    \mathcal{S}(\widehat G_n)\cap\mathcal{V}_{G_*}^{i}
    \big|,
    \qquad
    t_{i}
    :=
    \big|
    \mathcal{S}(G_*)\cap\mathcal{V}_{G_*}^{i}
    \big|, \qquad  r_{i}:=s_{i}+t_{i}-1.
\]
Then, again arguing heuristically, the global
bound~\eqref{eqn:main_result_global_upper_bound} along with the
construction of $\mathsf{VPOT}_{G_*}(\widehat{G}_n,G_*)$ yield the
global estimation rate of order
\[
  n^{-1/(2r_{i})}
  =
  n^{-1/[2(s_{i}+t_{i}-1)]},
\]
for components associated with this neighborhood, provided that the
corresponding transported masses are nonvanishing.
Hence, the global result continues to exhibit heterogeneous convergence
rates: components belonging to less singular neighborhoods are estimated
faster, whereas components involved in larger local clusters converge more
slowly. \\

\noindent
Finally, since both $G_*$ and $\widehat G_n$ contain at most $k$ support
points, we always have $r_{i}\leq 2k-1$.
Consequently, the worst-case global convergence rate of parameter estimation is given by
\[
    n^{-1/[2(2k-1)]}
    =
    n^{-1/(4k-2)},
\]
which occurs in the most singular configuration where $G_*$ has $k_*=k$ atoms and all these $k$ atoms are arbitrarily close or completely overlap. This observation agrees with the optimal convergence rate of parameter estimation in finite Gaussian mixtures derived in previous work \cite{wu2020a,doss_optimal_2023,nguyen2026geometry}. At the other extreme, isolated one-to-one component matching yields the parametric rate of order
$n^{-1/2}$. Thus, the global VPOT bound continuously interpolates between
these two extremes and provides a configuration-adaptive characterization
of the convergence rates of individual mixture parameters over the full
parameter space.

\subsection{Minimax Lower Bound}
\label{sec:minimax_lower_bound}




In this subsection, we establish a minimax lower bound to complement the upper bounds derived in the previous subsection and thereby characterize the optimality of the proposed VPOT-based rates. Our argument is based on constructing a least favorable sequence of local perturbations of a reference mixing measure $G_0$ within a shrinking $W_1$-neighborhood and showing that these alternatives remain statistically difficult to distinguish. To carry out this construction, we require sufficient smoothness and integrability of the component densities along a suitable direction in the parameter space, which allow us to control the corresponding likelihood expansions and the distance between the induced statistical experiments. These regularity requirements are summarized in Assumption~$D(p)$ below. \\

\noindent
\textbf{Assumption $D(p)$.}  The family of densities $\{f(\cdot\mid \gamma),\gamma \in \Gamma\}$ is said to satisfy the \emph{Assumption $D(p)$} if this family satisfies the following condition 
    
    \begin{enumerate}
        \item $D_{\gamma}^{\boldsymbol{\alpha}}f$ exists for every multi-index  $\boldsymbol{\alpha}$ such that $|\boldsymbol{\alpha}|\leq p$.
        \item There exists a unit vector $\upsilon \in \mathbb{R}^d$ such that
    \begin{equation*}
        \mathcal{S}(\gamma,t_1,t_2) := \max_{|\boldsymbol{\alpha}|=p}\int_{\mathbb{R}^{\bar{d}}}\left|\frac{D^{\boldsymbol{\alpha}}_{\gamma}f(x \mid \gamma + t_1\upsilon)}{f(x\mid \gamma + t_2\upsilon)}\right|^{m}f(x\mid \gamma)\,dx 
    \end{equation*}
    is a well-defined continuous function in $\{(\gamma,t_1,t_2) \in \Gamma \times \mathbb{R} \times\mathbb{R}:(\gamma,\gamma + t_1\upsilon,\gamma + t_2\upsilon) \in \Gamma^3\}$ for $|\boldsymbol{\alpha}| \leq p$ and $m \in [1,4]$. In addition, there exists an $\varepsilon > 0$ such that for $|t_1-t_2| < \varepsilon$, $\mathcal{S}(\gamma,t_1,t_2) < \infty$ for all $\gamma$.
    \item  There exists some point $\gamma_0 \in \Gamma^{\circ}$ such that for all multi-index $|\boldsymbol{\alpha}| \leq p-1$, we have 
    \begin{equation*}
        \int |D^{\boldsymbol{\alpha}}_{\gamma}f(x,\gamma_0)|\,dx < \infty. 
    \end{equation*}
        \end{enumerate}
Assumption~$D(p)$ is tailored to the minimax lower-bound argument. Its role is to ensure that one can perturb a component of the mixing measure along a fixed direction in parameter space while keeping the resulting statistical models sufficiently close. The differentiability requirement in the first condition makes it possible to construct perturbations whose lower-order effects cancel, so that the separation between the corresponding mixture distributions only appears at a higher order. The second condition provides the moment bounds and continuity needed to control the likelihood ratios generated by these perturbations; in particular, it guarantees that the associated local experiments remain well behaved as the perturbation size vanishes. The final integrability condition at an interior point $\gamma_0$ ensures that the required derivatives can be integrated and that the perturbation construction can be carried out around a valid parameter value inside $\Gamma$. To illustrate the applicability of Assumption~$D(p)$, we present two examples for the location family of probability distributions $\{f(x\mid \gamma) = f(x-\gamma):\gamma \in \Gamma\}$, where $\Gamma$ is regular\footnote{A subset $S$ of Euclidean space $\mathbb{R}^d$ is called a regular closed set if $S$ equals the closure of its interior \cite{willard1970general}.} compact subset of $\mathbb{R}^d$. 
\begin{example}
    \begin{enumerate}
    \item \textbf{Multivariate Gaussian location family.} Consider the multivariate Gaussian location family as in Part 1 of Example \ref{example:dung_upper_bound}. For every multi-index $\boldsymbol{\alpha}$, there exists a polynomial
    $P_{\boldsymbol{\alpha}}$ such that
    $D_\gamma^{\boldsymbol{\alpha}} f(x\mid\gamma) = P_{\boldsymbol{\alpha}}(x-\gamma)f(x\mid\gamma)$.
    Therefore, for any unit vector $\upsilon\in\mathbb{R}^d$, and $t_1,t_2 \in \mathbb{R}$,
    \[
        \frac{
            D_\gamma^{\boldsymbol{\alpha}} f(x\mid\gamma+t_1\upsilon)
        }{
            f(x\mid\gamma+t_2\upsilon)
        }
        =
        P_{\boldsymbol{\alpha}}(x-\gamma-t_1\upsilon)
        \frac{
            f(x\mid\gamma+t_1\upsilon)
        }{
            f(x\mid\gamma+t_2\upsilon)
        }.
    \]
    Since the two Gaussian densities have the same covariance matrix,
    their likelihood ratio is the exponential of an affine function of
    $x$. Consequently, for every $m\in[1,4]$,
    \[
        \int_{\mathbb{R}^d}
        \left|
            \frac{
                D_\gamma^{\boldsymbol{\alpha}} f(x\mid\gamma+t_1\upsilon)
            }{
                f(x\mid\gamma+t_2\upsilon)
            }
        \right|^m
        f(x\mid\gamma)\,\mathrm{d}x
        <\infty.
    \]
    This integral depends continuously on $(\gamma,t_1,t_2)$. Moreover,
    \[
        \int_{\mathbb{R}^d}
        \left|D_\gamma^{\boldsymbol{\alpha}} f(x\mid\gamma_0)\right|
        \,\mathrm{d}x<\infty
    \]
    for every $\gamma_0\in\mathbb{R}^d$ and every multi-index $\boldsymbol{\alpha}$.
    Hence, the family of location Gaussian distributions satisfies Assumption $D(p)$, for
    every finite $p$.

    \item \textbf{Elliptical multivariate Student-$t$ location family.} Consider the elliptical multivariate Student-t location family as in Part 2 of Example \ref{example:dung_upper_bound}.  
    For every multi-index $\boldsymbol{\alpha}$, there exists a bounded rational function $R_{\boldsymbol{\alpha}}$ such that $D_\gamma^{\boldsymbol{\alpha}} f(x\mid\gamma)
        =
        R_{\boldsymbol{\alpha}}(x-\gamma)f(x\mid\gamma)$. Furthermore, for
    every fixed $a,b\in\mathbb{R}^d$, assume that
        $\sup_{x\in\mathbb{R}^d}
        \frac{f(x\mid a)}{f(x\mid b)}
        <\infty$.
    Then, it follows that for any unit vector $\upsilon\in\mathbb{R}^d$ and $t_1,t_2\in \mathbb{R}$
    \[
        \sup_{x\in\mathbb{R}^d}
        \left|
            \frac{
                D_\gamma^{\boldsymbol{\alpha}} f(x\mid\gamma+t_1\upsilon)
            }{
                f(x\mid\gamma+t_2\upsilon)
            }
        \right|
        <\infty.
    \]
    Therefore, for every $m\in[1,4]$,
    \[
        \int_{\mathbb{R}^d}
        \left|
            \frac{
                D_\gamma^{\boldsymbol{\alpha}} f(x\mid\gamma+t_1\upsilon)
            }{
                f(x\mid\gamma+t_2\upsilon)
            }
        \right|^m
        f(x\mid\gamma)\,\mathrm{d}x
        <\infty.
    \]
    The integral is continuous in $(\gamma,t_1,t_2)$. In addition, the
    derivatives satisfy
    \[
        \left|D_\gamma^{\boldsymbol{\alpha}} f(x\mid\gamma)\right|
        \leq
        C_{\boldsymbol{\alpha}}
        \left(1+\|x-\gamma\|\right)^{-(\nu+d+|\boldsymbol{\alpha}|)},
    \]
    and hence,
    \[
        \int_{\mathbb{R}^d}
        \left|D_\gamma^{\boldsymbol{\alpha}} f(x\mid\gamma_0)\right|
        \,\mathrm{d}x<\infty.
    \]
    Thus, the multivariate Student-$t$ location family satisfies
    Assumption $D(p)$, for every $\nu>0$ and every finite $p$. In
    particular, this includes the multivariate Cauchy distribution associated with $\nu=1$.
\end{enumerate}
\end{example}
Overall, Assumption~$D(p)$ guarantees the existence of statistically indistinguishable local alternatives with a prescribed separation in parameter space, which is the key ingredient for proving that no estimator can converge uniformly faster than the rate stated in Theorem~\ref{thm:lower_bound}.

\begin{theorem}
    \label{thm:lower_bound}
    Let $\varepsilon_n = n^{-1/(4d_0+2) +\kappa}$, for some $0<\kappa < \frac{1}{4d_0+2}$, and $0<\delta <\frac{1}{2} \min_{i\neq j}\|\gamma_{0i}-\gamma_{0j}\|$, where $(\gamma_{0i})_{i=1}^{k_0}$ denote the atoms of a known mixing measure $G_0\in\mathcal{G}_{k_0}(\Gamma)$. Under Assumption $D(2d_0+2)$, the following statement holds for any sequence of estimators $\check{G}_n \in \mathcal{G}_{\leq k}(\Gamma)$, 
    \begin{equation}
      \label{eqn:main_result_local_lower_bound}
        \sup_{\substack{G \in \mathcal{G}_{k}(\Gamma),\\ W_1(G,G_0)<\varepsilon_n}}\mathbb{E}_G\left[\mathsf{VPOT}_{G_0,\delta}(\check{G}_n,G)\right]\gtrsim n^{-1/2}.
    \end{equation}
\end{theorem}
\noindent
The proof of Theorem~\ref{thm:lower_bound} is in Appendix~\ref{app:theorem_2_proof}. Theorem~\ref{thm:lower_bound} complements the upper bounds in Theorem~\ref{thm:upper_bound} by showing that the MLE convergence rate identified there is minimax optimal. More precisely, the upper and lower bounds match at the order of $n^{-1/2}$ under the VPOT loss, establishing the minimax optimality of estimating the mixing measure $G_*$ as a whole with respect to this loss. 
This minimax lower bound, however, does not directly establish the optimality of the convergence rates for individual mixture parameters. The VPOT loss is defined as a sum of local powered POT discrepancies over the Voronoi cells, so a lower bound on the total VPOT loss does not imply a lower bound on each individual POT term. 
Therefore, while the upper bound admits a heuristic interpretation in terms of heterogeneous convergence rates for individual parameters, establishing their minimax optimality would require separate lower bounds for the corresponding local POT discrepancies. Establishing separate minimax lower bounds for the local POT discrepancies is considerably more challenging, because these terms depend on the random local configuration of fitted atoms and on the extended Voronoi construction used in the VPOT loss. A sharp component-wise lower bound would therefore require a more delicate localization argument that isolates each local discrepancy separately. Since this problem lies beyond the scope of our work, we leave it for future development.

\section{Proof of Theorem~\ref{thm:upper_bound}}
\label{sec:proof_upper_bound}
In this section, we present the proof of Theorem~\ref{thm:upper_bound} with Proposition~\ref{prop:L_1_distance_greater_than_POT} established as an intermediate result.
Throughout the proof, whenever a pair of sequences $(G_n)$ and $(G_n')$ are considered, $\mathcal{T}$ denotes the coarse-graining tree associated with the signed measure $G_n - G_n'$, as introduced in Appendix \ref{app:coarse_graining}. We write $J_r$ for the root of $\mathcal{T}$, $J^\uparrow$ for the parent of a non-root node $J$, $\operatorname{Child}(J)$ and $\operatorname{Desc}(J)$ for its children and descendants, respectively, $\varepsilon_J$ for its scale, and $\bar{\pi}_J$ for its total signed weight. \\

\noindent
The proof has three steps. The first two establish the local and global bounds relating VPOT to the $L_1$ distance between mixture densities, and the last applies these bounds to the MLE.

\begin{itemize}
    \item \textbf{Step 1}: We fix a mixing measure $G_0$ and prove by contradiction that $\mathsf{VPOT}_{G_0,\delta}(G,G') \lesssim \|f_G-f_{G'}\|_1$ for all $G$ and $G'$ in a sufficiently small neighborhood of $G_0$. We first reduce the claim to a bound in terms of the sup-norm difference between the mixture CDFs. For a pair of counterexample sequences, we then construct the coarse-graining tree and apply the coarse-graining expansion at the root when $\varepsilon_{J_r}\to 0$ and at the children of the root otherwise. Proposition \ref{prop:POT_based_on_tree} and the extended Voronoi cells bound VPOT from above, whereas Proposition \ref{lemma:multi_dimension_taylor_expansion} and strong identifiability bound the CDF difference from below. The two bounds imply that this ratio is bounded below by a positive constant. This contradicts the assumption that it converges to zero along the counterexample sequence.
    \item \textbf{Step 2}: We next establish the global bound $\mathsf{VPOT}_{G,\delta}(G,G')\lesssim\|f_G-f_{G'}\|_1$ over the entire class $\mathcal{G}_{\leq k}(\Gamma)$, where the first measure $G$ also serves as the reference measure used to construct the Voronoi cells.  If such a bound fails, there exist counterexample sequences $(G_n,G_n')$, and compactness provides a common subsequence along which both converge. Their density difference along this subsequence tends to zero, so identifiability forces their limits to be the same mixing measure. We then repeat the tree argument from the local case, with the reference measure and its Voronoi cells now varying with $n$, and show that the density-to-VPOT ratio is bounded away from zero, contradicting the counterexample sequence.
    \item \textbf{Step 3}: We apply these bounds to the MLE. The global result follows directly by applying the bound from Step~2 to $\widehat{G}_n(k)$ and the true mixing measure. In the local setting, the bound from Step 1 does not immediately apply to the MLE because the estimator is not necessarily close to $G_0$. Nevertheless, we show that the density-to-VPOT bound remains valid when the true mixing measure is close to $G_0$, even if the estimator is not. Consequently, the local and global bounds control the VPOT estimation error of $\widehat{G}_n(k)$ by its $L_1$ density error. The density-estimation result in Proposition \ref{prop:MLE_estimation} then yields the stated bounds for the MLE.
\end{itemize}
Now, we proceed to streamline the proof of Proposition~\ref{prop:L_1_distance_greater_than_POT} and use its result to complete the proof of Theorem~\ref{thm:upper_bound}.
\begin{proof}[Proof of Proposition~\ref{prop:L_1_distance_greater_than_POT}]
We adapt the localization and coarse-graining strategy of Heinrich and Kahn \cite[Theorem 6.3]{heinrich2018} to the multidimensional setting and to the partial optimal transport discrepancy $\mathsf{VPOT}_{G_0,\delta}$. The principal additional ingredient is a decomposition of the transport problem over suitably chosen Voronoi cells. Because these modifications require several nontrivial arguments, we provide the complete proof. \\

\noindent
Suppose that $G$ and $G'$ are finite measures supported on a compact set $\Gamma$ with $0\leq m(G),m(G')\leq 1$. For every $r\leq 2k$, the definition of partial optimal transport implies that
\begin{equation*}
    \mathsf{POT}^{r}_{r}(G,G') \gtrsim \mathsf{POT}^{r}_{1}(G,G'),
\end{equation*}
where the implicit constant is uniform over \(r\leq 2k\). Applying this inequality to the restrictions of \(G\) and \(G'\) on each extended Voronoi cell yields
\begin{align*}
    \overline{\mathsf{VPOT}}_{G_0,\delta}(G,G') &:= \sum_{i=1}^{k_0}\inf_{1\leq j\leq 2k}\mathsf{POT}_{r_{i,j}}^{r_{i,j}}\left(G|_{\mathcal{V}^{i,j}_{G_0,\delta}},G'|_{\mathcal{V}^{i,j}_{G_0,\delta}}\right) \\
    &\gtrsim  \sum_{i = 1}^{k_0}\inf_{1\leq j\leq 2k} \mathsf{POT}^{r_{i,j}}_{1}\left(G|_{\mathcal{V}^{i,j}_{G_0,\delta}},G'|_{\mathcal{V}^{i,j}_{G_0,\delta}}\right) =  \mathsf{VPOT}_{G_0,\delta}(G,G').
\end{align*}


Thus, we can prove stronger results for local regime 
\begin{equation}
\label{eqn_new:L_1_distance_greater_than_D_locally}
        \inf_{G,G_* \in \mathcal{B}_{W_1}(G_0,\varepsilon)} \|p_G-p_{G_*}\|_1/  \overline{\mathsf{VPOT}}_{G_0,\delta}(G,G_*) > 0.
    \end{equation}
and for global regime
\begin{equation}
    \label{eqn_new:L_1_distance_greater_than_D_globally}
    \inf_{G,G'\in \mathcal{G}_{\leq k}(\Gamma)}\|p_G-p_{G'}\|_1/ \overline{\mathsf{VPOT}}_{G,\delta}(G,G') > 0. 
\end{equation}
\textit{(i) Local part in equation~\eqref{eqn_new:L_1_distance_greater_than_D_locally}.} 
Let $f_1$ and $f_2$ be two density functions in $\mathbb{R}^{\bar{d}}$ with corresponding cumulative distribution functions $F_1$ and $F_2$. For every $x \in \mathbb{R}^d$, 
\begin{align*}
    |F_1(x)-F_2(x)| &= \left|\int_{\xi \preceq x}(f_1(\xi )-f_2(\xi ))d\xi \right| \leq \int_{\xi \preceq x}|f_1(\xi )-f_2(\xi )|d\xi\leq \|f_1-f_2\|_1.
\end{align*}
Taking supremum over $x \in \mathbb{R}^{\bar{d}}$ gives $\|F_1-F_2\|_\infty \leq \|f_1-f_2\|_1$. Thus, it suffices to prove that
\begin{equation}
    \label{eqn:L_1_distance_greater_than_D_locally_l_infty}
    \inf_{G,G' \in \mathcal{B}_{W_1}(G_0,\varepsilon)} \|F(\cdot\mid G)-F(\cdot \mid G')\|_{\infty} \gtrsim \overline{\mathsf{VPOT}}_{G_0,\delta}(G,G').
\end{equation}
Arguing by contradiction, suppose that there exist  sequences $(G_n)$ and $(G'_n)$ such that 
\begin{equation*}
    \begin{cases}
        \lim_{n\to \infty} \|F(\cdot\mid G_n)-F(\cdot\mid G'_n)\|_\infty/\overline{\mathsf{VPOT}}_{G_0,\delta}(G_n,G_n') = 0\\
       W_1(G_n,G_0), W_1(G'_n,G_0) \to 0. 
    \end{cases}
\end{equation*}
Without loss of generality, we will assume that $G_n$ and $G_n'$ satisfy the condition in Lemma \ref{lemma:nice_subsequence}. Then, we can construct a tree $\mathcal{T}$ with root $J_r$ and for each node $J \in \mathcal{T}$, we choose a base point $\oupgamma_J$.  \\

\noindent
\textbf{Case 1: } $\varepsilon_{J_r} \to 0$. Then, all the points converge to a single support point of $G_0$, i.e. $k_0 = 1$. In this case, the unique Voronoi cell is all space $\Gamma$ itself, which contains all the points of $G_n$ and $G_n'$.  \\

\noindent
Applying the coarse-graining expansion of $F$ at root $J:=J_r$ up to order $2k$, we have 
\begin{equation}
    \label{eqn:taylor_expansion_to_root}
    F(x\mid J) = \sum_{0 \leq |\boldsymbol{p}|\leq 2k}c(\boldsymbol{p}|J,\oupgamma_J)\varepsilon_J^{|\boldsymbol{p}|}D_{\gamma}^{\boldsymbol{p}}F(x\mid\gamma_J) + R(x\mid J). 
\end{equation}
Using the triangle inequality, identifiability and Proposition \ref{lemma:multi_dimension_taylor_expansion}, we have 
\begin{align}
\label{eqn:case_1_F_distance}
    \|F(\cdot\mid J)\|_{\infty} &\gtrsim \max_{0 \leq |\boldsymbol{p}| \leq 2k}\left|c(\boldsymbol{p}|J,\oupgamma_J)\varepsilon_J^{|\boldsymbol{p}|}\right| - o(\|c_J\|\varepsilon_J^{2k}) \gtrsim \max_{0\leq |\boldsymbol{p}|\leq 2k}\left|c(\boldsymbol{p}|J,\oupgamma_J)\varepsilon_J^{|\boldsymbol{p}|}\right|\nonumber \\
    &\overset{(i)}{\gtrsim} \|c_J\|\varepsilon^{|J|-1}_{J} \overset{(ii)}{\gtrsim}\max_{K \in \mathrm{Desc}(J)}|\ouppi_{K}|\varepsilon^{|J|-1}_{K^{\uparrow}},
\end{align}
where $(i)$ follows the bound \eqref{eqn:max_c_first_dominate_norm}, and $(ii)$ from equation~\eqref{eqn:dung_prop_5_c_J_norm} of Proposition \ref{lemma:multi_dimension_taylor_expansion}.  \\

\noindent
On the other hand, the discrepancy $\mathsf{VPOT}_{G_0,\delta}(G_n,G_n')$ becomes the Wasserstein distance between $G_n$ and $G_n'$. Then, by Proposition \ref{prop:POT_based_on_tree}, we have 
\begin{equation}
\label{eqn:case_1_wasserstein_distance}
    \overline{\mathsf{VPOT}}_{G_0,\delta}(G_n,G_n') \asymp \max_{J \in \textrm{Desc}(J_r)}|\ouppi_J|\varepsilon^{|J|-1}_{J^{\uparrow}}.
\end{equation}
Combining equations~\eqref{eqn:case_1_F_distance} and~\eqref{eqn:case_1_wasserstein_distance}, we obtain that $\|F(x\mid J_r)\|_{\infty}/\overline{\mathsf{VPOT}}_{G_0,\delta}(G_n,G_n')$ cannot go to 0, which is a contradiction. We achieve our result for this case.  \\

\noindent
\textbf{Case 2: } $\varepsilon_{J_r} \not \to 0$. It means that there is more than one cluster limits for $G_n-G'_n$, which happens when $k_0 > 1$ or $k_0=1$ and there is only one convergence with non-zero total mass, while the other points have asymptotically vanishing mass.  \\

\noindent
\textit{Expansion for Kolmogorov's distance: }
Instead of implementing coarse-graining expansion in Proposition \ref{lemma:multi_dimension_taylor_expansion} at the root $J_r$, we apply the coarse-graining expansion for each child $J$ of $J_r$ at order $2k$ as 
\begin{equation}
\begin{aligned}
F(x\mid J_r)
&= \sum_{J\in\mathrm{Child}(J_r)} F(x\mid J)
\\
&= \sum_{J\in\mathrm{Child}(J_r)}
   \sum_{0\leq |\boldsymbol{p}|\leq 2k}
   c(\boldsymbol{p}\mid J,\oupgamma_J)
   \varepsilon_J^{|\boldsymbol{p}|}
   D_{\gamma}^{\boldsymbol{p}}F(x\mid\oupgamma_J)
   + \sum_{J\in\mathrm{Child}(J_r)} R(x\mid J),
\end{aligned}
\end{equation}
Using $2k$-identifiability and estimation of remainder $R(x\mid J)$ in Proposition \ref{lemma:multi_dimension_taylor_expansion} for $J$, we have
\begin{align*}
\|F(\cdot \mid J_r)\|_{\infty}
&\gtrsim
\max_{J\in\mathrm{Child}(J_r)}
\max_{0\leq |\boldsymbol{p}|\leq 2k}
|c(\boldsymbol{p}|J,\oupgamma_J)\varepsilon^{|\boldsymbol{p}|}_{J}|
-
\max_{J\in\mathrm{Child}(J_r)}
o(\|c_J\|\varepsilon_J^{2k})
\\
&\overset{(i)}{\gtrsim}
\max_{J\in\mathrm{Child}(J_r)}
\|c_J\|\varepsilon_J^{|J|-1}
\\
&\overset{(ii)}{\gtrsim}
\max_{J\in\mathrm{Child}(J_r)}
\max_{K\in\mathrm{Desc}(J)}
\left|\ouppi_K\right|
\varepsilon^{|J|-1}_{K^{\uparrow}}
\end{align*}
where $(i)$ is correct due to estimation \eqref{eqn:max_c_first_dominate_norm} node $J_r$, and $(ii)$ is due to estimation \eqref{eqn:dung_prop_5_c_J_norm} for each node $J \in \mathrm{Child}(J_r)$ of Proposition \ref{lemma:multi_dimension_taylor_expansion}. In addition, by plugging in $\boldsymbol{p} = \boldsymbol{0}$ in the above estimation, thanks to the assumption that $\varepsilon_{J_r} \asymp 1$, we have 
\begin{equation*}
    \|F(\cdot\mid J_r)\|_{\infty} \gtrsim \max_{J \in \mathrm{Child}(J_r)}|\ouppi_{J}|\varepsilon^{|J|-1}_{J^\uparrow}.
\end{equation*}
Thus, we have 
\begin{equation}
\label{eqn:f_estimation_based_on_tree}
    \|F(\cdot\mid G_n)-F(\cdot\mid G'_n)\|_\infty \gtrsim \max_{J \in\mathrm{Desc}(J_r)\\}|\ouppi_J|\varepsilon_{J^{\uparrow}}^{|J|-1}. 
\end{equation}
\textit{Upper bound for the POT discrepancy:}
Next, we estimate the $\overline{\mathsf{VPOT}}_{G_0,\delta}(G_1,G_2)$ based on the tree structure. 
\begin{observation}
\label{obs:n_large_contain_local}
    For $n$ sufficiently large, for any index $i$, for any node $J \in \mathrm{Desc}(J_r)$, there exists a Voronoi neighborhood $\mathcal{V}^{i,j_i}_{G_0,\delta}$ such that $J \cap \mathcal{V}^{i,j_i}_{G_0,\delta} = \emptyset$ or $J  \subset \mathcal{V}^{i,j_i}_{G_0,\delta}$. 
\end{observation}
 In fact, fix $i$, as $\varepsilon_J \to 0$, all the points in the nodes $J\in J_r$ converge to a point called $\eta = \eta(J) \in \Gamma$. As there are at most $2k$ point $\eta$, and the boundaries $\{\partial\mathcal{V}^{i,j}_{G_0,\delta},\ 1\leq j \leq 2k\}$ are disjoint (with the topology induced in $\Gamma$ from $\mathbb{R}^d$), there exists at least one neighbor $\mathcal{V}^{i,j_i}_{G_0,\delta}$ such that its boundary does not contain any limit point $\eta$. For each limit point, consider two situations:
 \begin{itemize}
     \item If $\eta \in \mathrm{int}(\mathcal{V}^{i,j_i}_{G_0,\delta})$, then there exists a small ball centered at $\eta$ belonging to $\mathcal{V}^{i,j_i}_{G_0,\delta}$. Thus, when $n$ is sufficiently large, all the points in $J$ lie inside this ball and therefore inside this Voronoi neighborhood.

     \item If $\eta \not \in \overline{\mathcal{V}^{i,j_i}_{G_0,\delta}}$, then there exists a small ball centered at $\eta$ having empty intersection with $\mathcal{V}^{i,j_i}_{G_0,\delta}$. When $n$ is sufficiently large, all the points in $J$ lie inside this ball and therefore do not belong to this Voronoi neighborhood.
 \end{itemize}
Now, we prove that 
\begin{equation}
    \label{eqn:D_estimation_based_on_tree}
    \overline{\mathsf{VPOT}}_{G_0,\delta}(G_n,G'_n) \lesssim \max_{J \in\mathrm{Desc}(J_r)\\}|\ouppi_J|\varepsilon_{J^{\uparrow}}^{q_J-1} ,
\end{equation} 
where $q_J = \min_{1\leq i\leq k}\{|\mathcal{V}^{i,j_i}_{G_0,\delta}|:J\subset \mathcal{V}^{i,j_i}_{G_0,\delta}\}$. This result is based on the observation that 
\begin{align}
    \mathsf{POT}(G_n|_{\mathcal{V}^{i,j_i}_{G_0,\delta}},G_n
    '|_{\mathcal{V}^{i}_{G_0,\delta}}) &\asymp \max_{\substack{J\in \mathrm{Child}(J_r)\\J \subset\mathcal{V}^{i,j_i}_{G_0,\delta}}}\left(|\ouppi_{J}| \lor \max_{K \in \mathrm{Desc}(J)} |\ouppi_{K}|\varepsilon_{K^{\uparrow}}^{|\mathcal{V}^{i,j_i}_{G_0,\delta}|-1}\right)  \lor |\mathsf{m}_{n,i}-\mathsf{m}'_{n,i}|, \nonumber \\
    &\asymp \max_{\substack{J\in \mathrm{Child}(J_r)\\J \subset\mathcal{V}^{i,j_i}_{G_0,\delta}}}\left(|\ouppi_{J}| \lor \max_{K \in \mathrm{Desc}(J)} |\ouppi_{K}|\varepsilon_{K^{\uparrow}}^{|\mathcal{V}^{i,j_i}_{G_0,\delta}|-1}\right),
    \label{eqn:POT_estimation_based_on_tree}
\end{align}
where $\mathsf{m}_{n,i} = m(G_n|_{\mathcal{V}^{i,j_i}_{G_0,\delta}})$ and $\mathsf{m}'_{n,i} = m(G'_n|_{\mathcal{V}^{i,j_i}_{G_0,\delta}})$. In fact, from Observation 1, it is straightforward that the tree with respect to the restriction of $G_n$ and $G_n'$ into Voronoi cell ${\mathcal{V}^{i,j_i}_{G_0,\delta}}$ is a subtree comprising some nodes in $\mathrm{Child}(J_r)$ and their descendants, with identical weight $\ouppi$. Using Proposition \ref{prop:POT_based_on_tree}, we obtain the first part of equation~\eqref{eqn:POT_estimation_based_on_tree}. In addition, to estimate $|\mathsf{m}_{n,i}-\mathsf{m}'_{n,i}|$, we have 
\begin{equation*}
    \mathsf{m}_{n,i} - \mathsf{m}'_{n,i}  =\sum_{\substack{J\in \mathrm{Child}(J_r)\\J \subset\mathcal{V}^{i,j_i}_{G_0,\delta}}} \ouppi_{J} \Rightarrow |\mathsf{m}_{n,i} - \mathsf{m}'_{n,i}| \leq \sum_{\substack{J\in \mathrm{Child}(J_r)\\J \subset\mathcal{V}^{i,j_i}_{G_0,\delta}}}|\ouppi_{J}|.
\end{equation*}
Thus, we achieve the second part of equation~\eqref{eqn:POT_estimation_based_on_tree}. In addition, since $\varepsilon_{J} \lesssim 1$, we achieve estimation \eqref{eqn:D_estimation_based_on_tree}. Combining equations~\eqref{eqn:f_estimation_based_on_tree} and~\eqref{eqn:D_estimation_based_on_tree}, since $q_J \geq |J|-1$, we reach the contradiction. Thus, we achieve the desired inequality in equation~\eqref{eqn_new:L_1_distance_greater_than_D_locally}.  \\

\noindent
\textit{(ii) Global part in equation~\eqref{eqn_new:L_1_distance_greater_than_D_globally}: } Suppose that the inequality in equation~\eqref{eqn_new:L_1_distance_greater_than_D_globally} does not hold, then there exists a sequence of measures $G_n\in \mathcal{G}_{\leq k}(\Gamma)$ and $G_n' \in \mathcal{G}_{\leq k}(\Gamma)$ such that 
\begin{equation*}
   \lim_{n\to \infty} \|p_{G_n}-p_{G'_n}\|_1/ \overline{\mathsf{VPOT}}_{G_n,\delta}(G_n,G_n') = 0.
\end{equation*}
Noting that the parameter space $\Theta$ is compact, by extracting a subsequence, we can suppose that there exist two measures $G,G' \in \mathcal{G}_{\leq k}(\Gamma)$ such that $G_n \xrightarrow{W_1} G$ and $G'_n \xrightarrow{W_1} G'$. In addition, $\overline{\mathsf{VPOT}}_{G,\delta}(G,G')$ is bounded, we have $\|f_{G_n}-f_{G'_n}\|_1 \to 0$, by passing the limit $n\to \infty$, we have $\|f_{G} - f_{G'}\|_1 = 0$, which implies $f_{G} = f_{G'}$. As a result, $G= G'$ under Assumption $C(2k)$. Moreover, as $\mathcal{G}_{\leq k}(\Gamma)$ is a compact set, $G_n \in \mathcal{G}_{\leq k}(\Gamma)$ implies its limit $G'\equiv G \in \mathcal{G}_{\leq k}(\Gamma)$.  \\

\noindent
We similarly consider the global part in two cases as in local part. If $\varepsilon_{J_r} \to 0$, we use the identical argument, especially for the $\overline{\mathsf{VPOT}}_{G_n}(G_n,G_n')$, which becomes the Wasserstein distance. Otherwise, when $\varepsilon_{J_r} \not\to 0$, we consider the following observation: 
\begin{observation}
\label{obs:n_large_contain_global}
For every $i \in [k]$, there exists $j_i \in [2k]$  such that, for all sufficiently large \(n\) and every node \(J\in\operatorname{Desc}(J_r)\), either  $J \cap \mathcal{V}^{i,j_i}_{G_n,\delta} = \varnothing$ or $J  \subset \mathcal{V}^{i,j_i}_{G_n,\delta}$. 
\end{observation}
\noindent
Indeed, we first fix \(i\in[k]\). Since \(\Gamma\) is compact, after passing to a further subsequence if necessary, we may assume that the Voronoi cells \(\mathcal V_{G_n}^{i}\) converge in the Hausdorff distance to a nonempty compact set \(C_i\subseteq\Gamma\). Consequently, $$\sup_{\gamma \in \Gamma}|\mathrm{dist}(\gamma,\mathcal{V}^i_{G_n}) - \mathrm{dist}(\gamma,C_i)| \to 0.$$ 
\noindent
For each node \(J\in\operatorname{Desc}(J_r)\), let \(\eta_J\) denote the common limiting location of the support points indexed by \(J\). In other words
$\max_{\ell\in J}
    \|\oupgamma_{\ell}-\eta_J\|_2
    \to 0$.
There are at most $2k-1$ distinct limiting locations among the nodes
in $\operatorname{Desc}(J_r)$. On the other hand, the $2k$ extension
radii $\rho_j:=\frac{j\delta}{2k}$,
are pairwise distinct. Each limiting location $\eta_J$ rules out at
most one index $j$, namely an index satisfying $\operatorname{dist}(\eta_J,C_i)=\rho_j$. Therefore, there exists $j_i\in[2k]$ such that $\operatorname{dist}(\eta_J,C_i)\neq \rho_{j_i}$
for every $J\in\operatorname{Desc}(J_r)$. \\

\noindent
Since the tree contains only finitely many nodes, we may define
\[
c_i:=
\min_{J\in\operatorname{Desc}(J_r)}
\left|
  \operatorname{dist}(\eta_J,C_i)-\rho_{j_i}
\right|>0.
\]
By the Hausdorff convergence above and the convergence
$J\to\{\eta_J\}$, for all sufficiently large $n$ and every
$\ell\in J$,
\[
\left|
  \operatorname{dist}
  \bigl(\oupgamma_{\ell},\mathcal{V}_{G_n}^{i}\bigr)
  -
  \operatorname{dist}(\eta_J,C_i)
\right|
<\frac{c_i}{2}.
\]
Consequently, all the points $\oupgamma_{\ell}$ with $\ell\in J$ lie
on the same side of the level set
\[
    \left\{
        \gamma\in\Gamma:
        \operatorname{dist}
        \bigl(\gamma,\mathcal{V}_{G_n}^{i}\bigr)
        =\rho_{j_i}
    \right\}.
\]
Recalling that $\mathcal{V}_{G_n,\delta}^{i,j_i}
    =
    \left\{
        \gamma\in\Gamma:
        \operatorname{dist}
        \bigl(\gamma,\mathcal{V}_{G_n}^{i}\bigr)
        <\rho_{j_i}
    \right\}$, we conclude that either
$J\subseteq\mathcal{V}_{G_n,\delta}^{i,j_i}$
or $J\cap\mathcal{V}_{G_n,\delta}^{i,j_i}
    =\varnothing$. This proves the observation.


Using the same argument as in local part, we can prove that when $n$ is sufficiently large, the distance between $G_n$ and $G_n'$ indeed can be estimated as: 
\begin{equation*}
   \overline{\mathsf{VPOT}}_{G_n}(G_n,G_n') \lesssim \max_{J \in\mathrm{Desc}(J_r)\\}|\ouppi_J|\varepsilon_{J^{\uparrow}}^{q_J-1} 
\end{equation*}
where $q_J = \min_{1\leq i\leq k}\{|\mathcal{V}^{i,j_i}_{G_n,\delta}|:J\subset \mathcal{V}^{i,j_i}_{G_n\delta}\}$. Using the coarse-graining expansion of $\|F_{G_n}-F_{G_n'}\|_\infty$, we also reach the contradiction. Thus, this proves the global assertion. 
\end{proof}

\begin{proof}[Proof of Theorem \ref{thm:upper_bound}] 

\textit{(i) Local part: } 
For the local regime in equation \eqref{eqn:main_result_local_upper_bound}, we first prove that for $\overline{\varepsilon} = \varepsilon/2$, where $\varepsilon$ is chosen in equation \eqref{eqn:L_1_distance_greater_than_D_locally}, there exists a constant $C$ such that for each $G_* \in \mathcal{G}_{\leq k}(\Gamma)$ satisfying $W_1(G_*,G_0) < \overline{\varepsilon}$,  
\begin{equation}
    \label{eqn:local_global_situation}
    \|p_G-p_{G_*}\|_1 \geq C\cdot\mathsf{VPOT}_{G_0,\delta}(G,G_*). 
\end{equation}
In fact, if $W_1(G,G_0) < \varepsilon$, then equation \eqref{eqn:L_1_distance_greater_than_D_locally} implies that there exists $C_1$ such that $\|p_G-p_{G_*}\|_1 \geq C_1\cdot\mathsf{VPOT}_{G_0,\delta}(G,G_*)$. For $W_1(G,G_0) \geq \varepsilon$, suppose that there exist two sequences $G_{n} \in \mathcal{G}_{\leq k}(\Theta)$ and $G_{n,*} \in \mathcal{G}_{\leq k}(\Theta)$ such that $W_1(G_{n},G_0) \geq \varepsilon$,  $W_1(G_{n,*},G_0) \leq \overline{\varepsilon}$, and $\|p_{G_n}-p_{G_{n,*}}\|_1/\mathsf{VPOT}_{G_0,\delta}(G_n,G_{n,*}) \to 0$. Noting that $\mathsf{VPOT}_{G_0,\delta}(G_n,G_{n,*})$ is bounded in $\mathcal{G}_{\leq k}(\Theta)$, we have $\|p_{G_n}-p_{G_{n,*}}\|_1 \to 0$. From the compactness of $\mathcal{G}_{\leq k}(\Theta)$, by extracting a subsequence, we can suppose that $G_n \xrightarrow{W_1} \overline{G}$ and $G_{n,*} \xrightarrow{W_1} \overline{G}_*$. Using Bounded Convergence Theorem, we have $\|p_{G_n}-p_{G_{n,*}}\|_1 \to \|p_{\overline{G}}-p_{\overline{G}_{*}}\|_1$, which implies $p_{\overline{G}}=p_{\overline{G}_{*}}$. From identifiability assumption, it means that $\overline{G} = \overline{G}_{*}$. However, as $W_1(G_{n},G_0) \geq \varepsilon$ and $W_1(G_{n,*},G_0) \leq \overline{\varepsilon}$, by continuity of Wasserstein distance, we have $W_1(\overline{G},G_0) \geq \varepsilon$ and $W_1(\overline{G}_{*},G_0) \leq \overline{\varepsilon}$. This cannot happen when $\overline{G} = \overline{G}_{*}$ which is a contradiction. Thus, there exists a constant $C_2$ such that for $W_1(G_*,G_0) < \overline{\varepsilon}$ and $W_1(G,G_0) \geq \varepsilon$, we have $ \|p_G-p_{G_*}\|_1 \geq C_2\cdot\mathsf{VPOT}_{G_0,\delta}(G,G_*)$. Overall, equation \eqref{eqn:local_global_situation} holds for $C = \min\{C_1,C_2\}$. Now, for $G \in \mathcal{G}_{\leq k}(\Gamma)$ such that $W_1(G,G_0) \leq \bar{\varepsilon}$, we have uniformly that 
\begin{equation*}
    \mathbb{E}_{G}[\mathsf{VPOT}_{G_0,\delta}(\widehat{G}_n(k),G)] \overset{(a)}{\lesssim} \mathbb{E}_{G}[\|p_{\widehat{G}_n(k)}-p_{G}\|_1] \overset{(b)}{\lesssim} (\log(n)/n)^{1/2},
\end{equation*}
where the inequality $(a)$ is due to equation~\eqref{eqn:local_global_situation} and the inequality $(b)$ is due to Proposition \ref{prop:MLE_estimation}.  \\

\noindent
\textit{(ii) Global part} For the global regime in equation \eqref{eqn:main_result_global_upper_bound}, we also have uniformly that
\begin{equation*}
    \mathbb{E}_{G}[\mathsf{VPOT}_{G,\delta}(\widehat{G}_n(k),G)] \overset{(c)}{\lesssim} \mathbb{E}_{G}[\|p_{\widehat{G}_n(k)}-p_{G}\|_1] \overset{(d)}{\lesssim} (\log(n)/n)^{1/2},
\end{equation*}
where the inequality $(c)$ is due to equation \eqref{eqn:L_1_distance_greater_than_D_globally} and the inequality $(d)$ is due to Proposition \ref{prop:MLE_estimation}. 
\end{proof}
\section{Conclusion}
\label{sec:conclusion}
In this paper, we aim to characterize the heterogeneous convergence behavior of parameter estimation in finite mixtures. This cannot be fully captured by the commonly used Wasserstein distance, which assigns the same convergence rate to all components. To this end, we introduce the Voronoi-based partial optimal transport (VPOT) loss, which localizes the comparison of two mixing measures to Voronoi neighborhoods and uses partial optimal transport to accommodate the unequal masses of their local restrictions. By raising the first-order POT discrepancy within each neighborhood to a power determined by the number of locally competing atoms, VPOT adapts to the local degree of singularity of the mixture model. Under suitable regularity and strong identifiability conditions, we establish uniform upper bounds under the VPOT loss that admit a heterogeneous component-wise interpretation: up to logarithmic factors, local configurations indexed by $r_i$ correspond to estimation rates of order $n^{-1/(2r_i)}$, where $r_i$ reflects the number of locally competing components. In the most singular local configuration, our result recovers the classical worst-case rate $n^{-1/(4d_0+2)}$ established in previous work, whereas less singular configurations admit strictly faster rates. We further establish a minimax lower bound showing that the convergence rate for estimating the mixing measure as a whole under the VPOT loss is optimal; however, this result does not imply separate minimax lower bounds for each local POT discrepancy or for individual mixture parameters. Moreover, our analysis applies to parameter spaces of arbitrary fixed dimension without requiring a uniform lower bound on the mixing weights or prior knowledge of the true number of components. Taken together, these results provide a configuration-adaptive characterization of the statistical complexity of parameter recovery in finite mixture models and demonstrate that partial optimal transport offers a natural geometry for capturing their heterogeneous local behavior.\\

\appendix
\centering
\textbf{\Large{Supplement to \\
``Characterizing Heterogeneous Rates in Finite Mixture Estimation\vspace{0.2em} via Partial Optimal Transport''}}

\justifying
\setlength{\parindent}{0pt}
\section{Coarse-Graining Tree Structure and Application}
\label{app:coarse_graining}
\subsection{Coarse-Graining Tree and Notation}

The concept of \textit{coarse-graining tree} was first introduced in \cite{heinrich2018}, which is a tool to study the asymptotic closeness of two discrete distributions. The original article deals with one-dimensional setting, and the multi-dimensional counterpart is generalized later in \cite{wei2023minimum}. In this section, we recall the basic concept, intuition, and the key results to be used later in our proof. 

For two sequences of measures $(G_n)_n \subset \mathcal{G}_k(\Gamma)$ and $(G'_n)_n \subset \mathcal{G}_{k'}(\Gamma)$, we write $G_n = \sum_{i=1}^{k} \pi_{i,n}\delta_{\gamma_{i,n}}$ and $G'_n = \sum_{i=1}^{k'} \pi'_{i,n}\delta_{\gamma'_{i,n}}$, and the signed measure $G_n - G'_n = \sum_{i=1}^{k+k'}\overline{\uppi}_{i,n}\delta_{\overline{\upgamma}_{i,n}}$, where 
\begin{equation*}
    \overline{\uppi}_{i,n} = \begin{cases}
        \pi_{i,n},\  1\leq i\leq k,\\
        -\pi_{i-k,n},\  k+1 \leq i\leq k+k',
    \end{cases}
    \text{ and }      \overline{\upgamma}_{i,n} = \begin{cases}
        \gamma_{i,n},\  1\leq i\leq k,\\
        \gamma_{i-k,n},\  k+1 \leq i\leq k+k'
    \end{cases}. 
\end{equation*}

The following lemma serves to categorize the intrinsic closeness among $\gamma_{i,n}$'s up to asymptotic order. As it is a straightforward generalization of \cite[Lemma 7.1]{heinrich2018} and also being reformulated in \cite[Lemma A.3]{wei2023minimum}, we omit the proof.
\begin{lemma}
\label{lemma:nice_subsequence}
    There exists a subsequence of $(G_n-G_n')_n$ (possibly itself) such that we can choose a  finite number $S$ such that 
    \begin{equation*}
        0 \equiv \varepsilon_0(n) < \varepsilon_1(n) < \cdots \varepsilon_S(n) \equiv 1 \text{ satisfying } \varepsilon_{s}(n) = o(\varepsilon_{s+1}(n)),
    \end{equation*}
    and for all $1\leq i,j \leq k+k'$, there exists a unique $s(i,j)\in \{0\}\cup [S]$ such that $\|\overline{\upgamma}_{i,n}-\overline{\upgamma}_{j,n}\| \asymp \varepsilon_{s(i,j)}(n)$. 
\end{lemma}
We first observe that $s:\mathcal [k+k']\times[k+k']\to \{0\} \cup [S]$
defines an ultrametric on $[k+k']$: in addition to
nonnegativity, symmetry, it satisfies
the strong triangle inequality $s(i,j)\leq \max\{s(i,\ell),s(\ell,j)\},\, i,j,\ell\in[k+k'] $. Consequently, the family of open $s$-balls is nested: any two such balls
are either disjoint or one is contained in the other. This nested family encodes the different asymptotic scales at which the
sequences $\overline{\upgamma}_{i,n}$ approach one another along the selected
subsequence. More precisely, the ball
$\overline{B}_i(r):=\{j\in[k+k']:s(i,j)\leq r\}$
collects the indices associated with the cluster containing
$\overline{\upgamma}_{i,n}$ at the scale corresponding to $\varepsilon_r(n)$. If
$r'<r$, then $\overline{B}_i(r')\subseteq \overline{B}_i(r)$,
so decreasing the radius produces a finer cluster around
$\overline{\upgamma}_{i,n}$. Following the hierarchical constructions in
\cite{heinrich2018,wei2023minimum}, we represent this nested system of
clusters by a \emph{coarse-graining tree}.
\begin{definition}[Coarse-graining tree]
Consider the family $\mathscr{B}:=\{\overline{B}_i(r): i\in[k+k'],\ r\in\{0\}\cup[S]\}$, where repeated balls are identified. The \emph{coarse-graining tree} $\mathcal{T}$ is the rooted tree whose nodes are the elements of $\mathscr{B}$ and whose hierarchical ordering is induced by set inclusion. The root of $\mathcal{T}$ is $J_{\mathrm{root}}:=[k+k']$. For every non-root node $J\in\mathcal{T}$, its \emph{parent}, denoted by $J^\uparrow$, is the unique node satisfying $J\subsetneq J^\uparrow$ such that there is no $K\in\mathcal{T}$ for which $J\subsetneq K\subsetneq J^\uparrow$. The sets of children and descendants of a node $J$ are respectively defined by $\operatorname{Child}(J):=\{I\in\mathcal{T}:I^\uparrow=J\}$ and $\operatorname{Desc}(J):=\{I\in\mathcal{T}:I\subsetneq J\}$. Finally, the \emph{diameter} of $J$ is defined as $s(J):=\max_{i,j\in J}s(i,j)$.
\end{definition}
The diameter level $s(J)$ describes the asymptotic scale of the cluster $J$. More precisely, for any $i,j\in J$, the corresponding representatives satisfy $\|\bar{\upgamma}_{i,n}-\bar{\upgamma}_{j,n}\|\lesssim\varepsilon_{s(J)}(n)$. Moreover, if $K$ and $K'$ are two distinct children of $J$, then for every $i\in K$ and $j\in K'$, we have $\|\bar{\upgamma}_{i,n}-\bar{\upgamma}_{j,n}\|\asymp\varepsilon_{s(J)}(n)$. Thus, $s(J)$ identifies the scale at which the representatives in different immediate subclusters of $J$ become separated. Although the scale sequence $\{\varepsilon_s(n)\}_{s=0}^S$ depends on $n$, the combinatorial structure of the coarse-graining tree $\mathcal{T}$ does not depend on $n$ once the relevant subsequence has been fixed. In this way, $\mathcal{T}$ records the hierarchical clustering of the representatives according to the asymptotic orders of their pairwise distances \cite{heinrich2018,wei2023minimum}.

\subsection{Partial Optimal Transport Bounds} 
Now we estimate the POT distance through the tree $\mathcal{T}$. For convenience, we omit the index $n$ in $\pi_i$, $\overline{\uppi}_{i}$, and $\gamma_{i}$ throughout our analysis. For two measures $G_n$ and $G_n'$ whose coarse-graining tree is denoted by $\mathcal{T}$, we write for any subset $J$ of $[1,k+k']$, 
\begin{equation*}
    \pi_{J} = \sum_{i\in J \cap [1,k]}\pi_i, \quad \pi'_{J} = \sum_{i\in J \cap [k+1,k+k']}\pi_i, \quad \ouppi_J = \sum_{j\in J} \ouppi_j, \quad \varepsilon_J = \varepsilon_{s(J)}. 
\end{equation*}

\begin{lemma}[Modification of Lemma B.2] 
\label{lemma:optimal_coupling}
Consider two measures $G_1$ and $G_2$ such that $m(G_1) \leq m(G_2)$. Then, we can find a measure $\Pi$ on $J_r \times J_r$ with marginal measures $\mu_1 = G_1$ and $\mu_2 \preceq G_2$ such that $\Pi(J,J) := \Pi(J\times J) = \pi_{J} \wedge \pi'_J$. 

\end{lemma}

\begin{proof}
    Use the same recurrent argument as in Lemma B.2, \cite{heinrich2018}. 
\end{proof}

\begin{proposition}[Representation of partial OT discrepancy based on tree structure]
\label{prop:POT_based_on_tree}
Suppose that $G_n \in \mathcal{G}_{k}(\Gamma)$ and $G'_n \in \mathcal{G}_{k'}(\Gamma)$. Then, for any $q \geq 1$, we can find an equivalent estimation for the $q$-POT discrepancy between $G_n$ and $G_n'$ 
    \begin{equation}
        (\mathsf{POT}_q(G_n,G_n'))^{q} \asymp \max_{J \in \mathrm{Desc}(J_r)}|\ouppi_J|\varepsilon^q_{J^{\uparrow}} + |m(G_n) - m(G_n')|. 
    \end{equation}
\end{proposition}

\begin{proof}
Write $\mathsf{m} = m(G_n),\ \mathsf{m}' = m(G_n')$. Given a coupling $\Pi = (\Pi_{ab})$ between $G_n$ and $G_n'$ specified in each part (lower bound and upper bound), using the same notation as \cite{heinrich2018}, set
\begin{align*}
    \Pi(I,I') &= \Pi(\{\gamma_i\}_{i \in I \cap [1,k]} \times \{\gamma_i\}_{i \in I' \cap [k+1,k+k']})\\
    w_q(I,I') &= \sum_{(i,i')\in I\times I'} \Pi(\{i\},\{i'\})\|\gamma_i-\gamma_{i'}\|^q. 
\end{align*}

\textbf{Upper bound: } We adapt the recursive argument of Heinrich and Kahn \cite{heinrich2018} to our notation. We show by induction over the tree that, for every node $I \in \mathrm{Desc}(J_r)$, the coupling $\Pi$ constructed in Lemma \ref{lemma:optimal_coupling} satisfies 
\begin{equation}
\label{eqn:dung_wasserstein_upper_bound_induction_hypo}
    w_q(I,I) \lesssim  \max_{J \in \mathrm{Desc}(J_r)}|\ouppi_J|\varepsilon^q_{J^{\uparrow}}. 
\end{equation}
The claim is immediate when $I$ is a leaf. Now let $I$ be an internal node and suppose that the desired bound holds for each $K \in \mathrm{Child}(I)$. Decomposing the transport according to the children of $I$, we obtain 
\begin{equation*}
    w_q(I,I) = \sum_{K \in \mathrm{Child}(I)} \left[w_q(K,K) + \sum_{\substack{K' \in \mathrm{Child}(I)\\K'\neq K}}w_q(K,K')\right].
\end{equation*}

Distinct children of $I$ are separated at scale $\varepsilon_J$, therefore, 
\begin{equation*}
\sum_{\substack{K' \in \mathrm{Child}(I)\\K'\neq K}}w_q(K,K') \lesssim  \sum_{\substack{K' \in \mathrm{Child}(I)\\K'\neq K}}\Pi(K,K')\varepsilon_J^q  \preceq \Pi(K,K^{c})\varepsilon_J^q.
\end{equation*}
By the construction of $\Pi$, mass is transported from \(K\) to another child only when the two measures have unequal masses on $K$. Hence, $\Pi(K,K^{c}) \leq \pi_K - \Pi(K,K)  \leq \pi_K -  \pi_{K} \wedge \pi'_K\leq \ouppi_{K}$. 
Consequently, we have 
\begin{equation*}
    w_q(I,I) \lesssim \sum_{K \in \mathrm{Child}(I)}[w_q(K,K)+|\ouppi_K|\times \varepsilon_I^p]. 
\end{equation*}
Because \(K^\uparrow=I\), we have \(\varepsilon_I=\varepsilon_{K^\uparrow}\). By applying the induction hypothesis to \(w_q(K,K)\) and using the fact that the tree has uniformly bounded size, we obtain the upper bound in equation~\eqref{eqn:dung_wasserstein_upper_bound_induction_hypo}. Apply equation~\eqref{eqn:dung_wasserstein_upper_bound_induction_hypo} to $I = J_r$ and adding the unmatched-mass penalty gives 
\begin{equation*}
    \mathsf{POT}^q_q(G_n,G_n') \lesssim \max_{J \in \mathrm{Desc}(J_r)}|\ouppi_J|\varepsilon^q_{J^{\uparrow}} + |m(G_n) - m(G_n')|. 
\end{equation*}

\textbf{Lower Bound: } Consider an arbitrary admissible partial coupling \(\Pi\) between \(G_n\) and \(G_n'\). We show that
\begin{equation}
\label{eqn:dung_proposition_6_lower_bound}
    w_q(J_r,J_r)+ |\mathsf{m} - \mathsf{m}'|  \gtrsim  \max_{J \in \mathrm{Desc}(J_r)}|\ouppi_I|\varepsilon^q_{I^{\uparrow}} + |\mathsf{m} - \mathsf{m}'|. 
\end{equation}
Without loss of generality, assume that $\mathsf{m}\leq \mathsf{m}'$. Introduce a cemetery point \(\partial\in \mathbb{R}^d \setminus \Gamma\) and extend the space  to 
$\overline{\Gamma} = \Gamma \cup \{\partial\}$ by setting $d(\partial,\Gamma) = M$, where $M > 1 + \mathrm{Diam}(\Gamma)$ is fixed. Define the augmented measure $\overline{G}_n := G_n + (\mathsf{m}' -\mathsf{m})\delta_{\partial}$. Let $\Pi_2$ denote the second marginal of $\Pi$, since $\mathsf{m} \leq \mathsf{m}'$, we achieve the residual measure $R'_n := G_n' - \Pi_2$ is non-negative and has total mass \(m'-m\). We extend \(\Pi\) to a coupling between \(\overline G_n\) and \(G_n'\) by setting $\overline{\Pi} := \Pi + \delta_{\partial} \otimes R_n'$. Attach the cemetery point as a singleton child above the original root, and denote the resulting tree by \(\overline{\mathcal T}\), with root \(\overline J_r\). This construction preserves every node and every scale of the original tree, while adding a node with singleton $\partial$ to the $\mathrm{Child}(J_r)$. 

Let \(\overline w_q\) denote the transport cost associated with \(\overline\Pi\). By Lemma \cite[Lemma 7.3]{heinrich2018}, 
\begin{equation}
\label{eqn:lower_bound_w_q_bar}
    \overline{w}_q(\overline{J}_r,\overline{J}_r) \gtrsim \max_{\overline{I} \in \mathrm{Desc}(\overline{J}_r)}|\ouppi_{\overline{I}}|\varepsilon^q_{\overline{I}^{\uparrow}} \overset{(*)}{\gtrsim} \max_{I \in \mathrm{Desc}(J_r)}|\ouppi_I|\varepsilon^q_{I^{\uparrow}}
\end{equation}
The second inequality holds because the original coarse-graining tree is contained in \(\overline{\mathcal T}\), and the mass discrepancies and scales of its nodes remain unchanged. On the other hand,

\begin{align}
\label{eqn:w_q_bar_and_w_q}
\nonumber  \overline{w}_q(\overline{J}_r,\overline{J}_r) &= w_q(J_r,J_r) + \int_{\Gamma}d(\partial,\gamma)dR'_n(\gamma) \\
    &\leq w_q(J_r,J_r)+ |\mathsf{m}'-\mathsf{m}|(M + \mathrm{Diam}(\Gamma))^{q}
    \lesssim w_q(J_r,J_r) + |\mathsf{m}'-\mathsf{m}|.
\end{align}
Combining equations~\eqref{eqn:lower_bound_w_q_bar} and~\eqref{eqn:w_q_bar_and_w_q} gives 
\begin{equation*}
    w_q(J_r,J_r) + |\mathsf{m}'-\mathsf{m}| \gtrsim \max_{I \in \mathrm{Desc}(J_r)}|\ouppi_I|\varepsilon^q_{I^{\uparrow}} +  |\mathsf{m}'-\mathsf{m}| . 
\end{equation*}
Since this bound holds for every admissible partial coupling \(\Pi\), taking the infimum over \(\Pi\) proves the lower bound. Combining the two estimates establishes the lower bound in equation~\eqref{eqn:dung_proposition_6_lower_bound}. 
\end{proof}

\subsection{Coarse-graining expansion} 
Recall that the cumulative distribution function associated with $\{f(\cdot\mid \gamma),\gamma \in \Gamma\}$ is defined by
$$F(x\mid \gamma) = \int_{(-\infty,x)}f(t\mid \gamma)\,dt:= \int_{-\infty}^{x_{\bar{d}}}\cdots\int_{-\infty}^{x_1}f(t\mid \gamma)\,dt_1\cdots dt_{\bar{d}}.$$
For every node $J \in \mathcal{T}$, set $$F(\cdot\mid J) = \sum_{j\in J} \overline{\uppi}_jF(\cdot \mid \gamma_j).$$ The following proposition adapts the expansion in \cite[Lemma 7.4]{heinrich2018} to the multidimensional setting. It may also be viewed as a specialization of \cite[Lemma A.7]{wei2023minimum} to the family of cumulative distribution functions considered here. We state the result explicitly in our notation for later use. In this part, we assume that $\max\{k,k'\} \leq 2\boldsymbol{k}$.

\begin{proposition}[Order $2\boldsymbol{k}$ coarse-graining expansion of $F$ at $J$]  
\label{lemma:multi_dimension_taylor_expansion}
Let $J \in \mathcal{T}$, choose a representative $\oupgamma_J \in \{\oupgamma_j:j \in J\}$ and let $2\boldsymbol{k}\geq |J|-1$. Suppose that $F(\cdot\mid \gamma)$  is continuously differentiable up to total order $2\boldsymbol{k}$ and that its derivatives of total order $2\boldsymbol{k}$ satisfy the uniform-continuity condition 
\begin{equation}
\label{eqn:dung_prop_5_uniform_continuity}
        \sup_{x \in \mathbb{R}^{\bar{d}}}|D_{\gamma}^{\boldsymbol{\alpha}}F(x\mid\gamma) - D_{\gamma}^{\boldsymbol{\alpha}}F(x\mid \gamma')| \leq \omega(\gamma - \gamma'). 
\end{equation}
Then there exist coefficients $c_{J} = (c(\boldsymbol{p}|J,\oupgamma_J))_{\boldsymbol{p} \succeq \boldsymbol{0},|\boldsymbol{p}| \leq 2\boldsymbol{k}}$ and a remainder $R(x\mid J)$ such that 
\begin{equation}
        \label{eqn:multi_dimension_taylor_expansion} 
        F(x\mid J) = \sum_{0\leq |\boldsymbol{p}|\leq 2\boldsymbol{k}}c(\boldsymbol{p}|J,\oupgamma_J)\varepsilon_J^{|\boldsymbol{p}|}D_{\gamma}^{\boldsymbol{p}}F(x\mid\oupgamma_J) + R(x\mid J). 
\end{equation}
Moreover, $c(\boldsymbol{0}|J,\oupgamma_J) = \ouppi_J$, in addition, 
\begin{equation}
\label{eqn:dung_prop_5_c_J_max}
     \|c_J\| \asymp \max_{0\leq |\boldsymbol{p}|\leq |J|-1} |c(\boldsymbol{p}|J,\oupgamma_J)|;
\end{equation}
\begin{equation}
\label{eqn:dung_prop_5_c_J_norm}
            \|c_J\| \succeq \max_{K \in \mathrm{Desc}(J)}\left[|\ouppi_K|\left(\dfrac{\varepsilon_{K^{\uparrow}}}{\varepsilon_J}\right)^{|J|-1}\right]; 
\end{equation}
and $\sup_{x\in \mathbb{R}^{\bar{d}}}|R(x\mid J)| = o(\|c_J\|\varepsilon^{2\boldsymbol{k}}_{J})$. We call it the \emph{order $2\boldsymbol{k}$  coarse-graining expansion of $F$ at $J$}. 

    
        
\end{proposition}

\begin{proof}
We adapt the inductive argument of Heinrich and Kahn \cite[Lemma 7.4]{heinrich2018} to the multidimensional setting; see also \cite[Lemma A.7]{wei2023minimum} for a related multivariate formulation. The main modification is the use of a multi-index Taylor expansion, for which we provide an explicit uniform Peano remainder. For completeness, we include the details in our notation.

\textbf{Step 1: Recursive construction of the coefficients. }  In this step, we employ the induction argument to give an inductive expression of $c_J$. When $J$ is a leaf of $\mathcal{T}$, then obviously $F(x\mid J) = \ouppi_JF(x\mid \oupgamma_J)$ then we may set $c(\boldsymbol{p}|J,\oupgamma_J) = \ouppi_{J}\boldsymbol{1}_{\{\boldsymbol{p} = \boldsymbol{0}\}}$ and the remainder $R(x\mid J) = 0$. Now, suppose that equation \eqref{eqn:multi_dimension_taylor_expansion} is correct for all children $K$ of node $J$, which means that the following expansion holds for all $K \in \mathrm{Child}(J)$
    \begin{equation}
        \label{eqn:children_node_Taylor_expansion}
        F(x\mid K) = \sum_{0\leq |\boldsymbol{\ell}|\leq  2\boldsymbol{k}}c(\boldsymbol{\ell}|K,\oupgamma_K)\varepsilon_K^{|\boldsymbol{\ell}|}D_{\gamma}^{\boldsymbol{\ell}}F(x\mid \oupgamma_K) + R(x\mid K). 
    \end{equation}

To transfer this expansion from $K$ to its parent $J$, we apply Taylor expansion for the derivative $D^{\boldsymbol{\ell}}F(x\mid \oupgamma_K)$ around $\oupgamma_J$ as
    \begin{align*}
        D_{\gamma}^{\boldsymbol{\ell}}F(x\mid \oupgamma_K) &= \sum_{\substack{\boldsymbol{p}\succeq \boldsymbol{\ell}\\|\boldsymbol{p}| \leq 2\boldsymbol{k}}} \dfrac{(\oupgamma_K-\oupgamma_J)^{\boldsymbol{p}-\boldsymbol{\ell}}}{(\boldsymbol{p}-\boldsymbol{\ell})!} D_{\gamma}^{\boldsymbol{p}}F(x\mid \oupgamma_J) + 
        \varepsilon_J^{2{\boldsymbol{k}}-|\boldsymbol{\ell}|}\overline{R}_{\boldsymbol{\ell},KJ},
    \end{align*}
    where we define the Peano remainder as 
   \begin{equation*}
\begin{aligned}
\overline{R}_{\boldsymbol{\ell},KJ}
&=
(2\boldsymbol{k}-|\boldsymbol{\ell}|)
\sum_{|\boldsymbol{\alpha}|=
      2\boldsymbol{k}-|\boldsymbol{\ell}|}
\frac{
  (\oupgamma_K-\oupgamma_J)^{\boldsymbol{\alpha}}
}{
  \varepsilon_J^{2\boldsymbol{k}-|\boldsymbol{\ell}|}
}
\\
&\quad\times
\int_0^1
(1-t)^{2\boldsymbol{k}-|\boldsymbol{\ell}|-1}
\Bigl[
D_{\gamma}^{\boldsymbol{\ell}+\boldsymbol{\alpha}}
F\left(
  x\mid \oupgamma_J+t(\oupgamma_K-\oupgamma_J)
\right)
\\
&\qquad\qquad
-
D_{\gamma}^{\boldsymbol{\ell}+\boldsymbol{\alpha}}
F\left(x\mid\oupgamma_J\right)
\Bigr]\,dt.
\end{aligned}
\end{equation*}
    By substituting this equation into equation~\eqref{eqn:children_node_Taylor_expansion}, we have 
    \begin{align*}
        F(x\mid K) &= \sum_{0\leq |\boldsymbol{\ell}|\leq 2\boldsymbol{k}}c(\boldsymbol{\ell}|K,\oupgamma_K)\varepsilon_K^{|\boldsymbol{\ell}|}D_{\gamma}^{\boldsymbol{\ell}}F(x\mid \oupgamma_K) + {R(x\mid K)} \\
        &= \sum_{0\leq |\boldsymbol{\ell}|\leq 2\boldsymbol{k}}c(\boldsymbol{\ell}|K,\oupgamma_K)\varepsilon_K^{|\boldsymbol{\ell}|}\left[\sum_{\substack{\boldsymbol{p}\succeq \boldsymbol{\ell}\\|\boldsymbol{p}| \leq 2\boldsymbol{k}}} \dfrac{(\oupgamma_K-\oupgamma_J)^{\boldsymbol{p}-\boldsymbol{\ell}}}{(\boldsymbol{p}-\boldsymbol{\ell})!} D_{\gamma}^{\boldsymbol{p}}(x\mid\oupgamma_J)+ \varepsilon_J^{2\boldsymbol{k}-|\boldsymbol{\ell}|}\overline{R}_{\boldsymbol{\ell},KJ}\right] + R(x\mid K)\\
        &= \sum_{0\leq |\boldsymbol{p}|\leq 2\boldsymbol{k}}{\left[\sum_{\substack{|\boldsymbol{\ell}| \geq 0\\ \boldsymbol{\ell} \preceq \boldsymbol{p}}} c(\boldsymbol{\ell}|K,\oupgamma_K)\left(\dfrac{\varepsilon_K}{\varepsilon_J}\right)^{|\boldsymbol{\ell}|} \dfrac{(\oupgamma_K-\oupgamma_J)^{\boldsymbol{p}-\boldsymbol{\ell}}}{\varepsilon_J^{|\boldsymbol{p}-\boldsymbol{\ell}|}(\boldsymbol{p}-\boldsymbol{\ell})!}\right]}\varepsilon_{J}^{\boldsymbol{p}}D_{\gamma}^{\boldsymbol{p}}F(x\mid\oupgamma_J)\\
        &\hspace{2cm}+\left[R(x\mid K) + \varepsilon_J^{2\boldsymbol{k}}\sum_{0\leq |\boldsymbol{\ell}| \leq 2\boldsymbol{k}}c(\boldsymbol{\ell}|K,\oupgamma_K)\left(\dfrac{\varepsilon_{K}}{\varepsilon_{J}}\right)^{|\boldsymbol{\ell}|} \overline{R}_{\boldsymbol{\ell},KJ}\right]. 
    \end{align*}
    By summing up over all the children $K$ of $J$, we have 
    \begin{equation*}
        F(x\mid J) = \sum_{0\leq |\boldsymbol{p}|\leq 2\boldsymbol{k}}c(\boldsymbol{p}|J,\oupgamma_J)\varepsilon_J^{|\boldsymbol{p}|}D_{\gamma}^{\boldsymbol{p}}F(x\mid \oupgamma_J) + R(x\mid J), 
    \end{equation*}
    where the coefficient $c(\boldsymbol{p}|J,\oupgamma_J)$ and the remainder $R(x\mid J)$ can be defined as 
    \begin{align*}
        c(\boldsymbol{p}|J,\oupgamma_J) &= \sum_{K \in \mathrm{Child}(J)}\sum_{\substack{|\boldsymbol{\ell}|\geq 0\\ \boldsymbol{\ell}\preceq \boldsymbol{p}}}c(\boldsymbol{\ell}|K,\oupgamma_K)\left(\dfrac{\varepsilon_K}{\varepsilon_J}\right)^{|\boldsymbol{\ell}|} \dfrac{(\oupgamma_K-\oupgamma_J)^{\boldsymbol{p}-\boldsymbol{\ell}}}{\varepsilon_J^{|\boldsymbol{p}-\boldsymbol{\ell}|}(\boldsymbol{p}-\boldsymbol{\ell})!}\\
        R(x\mid J) &= \sum_{K \in \mathrm{Child}(J)}\left[R(x\mid K) + \varepsilon_J^{2\boldsymbol{k}}\sum_{0\leq |\boldsymbol{\ell}| \leq 2\boldsymbol{k}}c(\boldsymbol{\ell}|K,\oupgamma_K)\left(\dfrac{\varepsilon_{K}}{\varepsilon_{J}}\right)^{|\boldsymbol{\ell}|} \overline{R}_{\boldsymbol{\ell},KJ}\right]
    \end{align*}
    
    \textbf{Step 2: Control by the low-order coefficients.} At this step, we prove an equivalent estimation of $\|c_J\|$ based on its leading coefficients in equation \eqref{eqn:dung_prop_5_c_J_max}. 
   We first employ induction argument to show that $|c(\boldsymbol{p}|J,\oupgamma_J)|\preccurlyeq 1$. When $J$ is a leaf, then $|c(\boldsymbol{p}|J,\oupgamma_J)| = \ouppi_{J}\boldsymbol{1}_{\{\boldsymbol{p} = \boldsymbol{0}\}}\preccurlyeq 1$. Suppose that $|c(\boldsymbol{p}|K,\oupgamma_K)|\preccurlyeq 1$ for each $0\leq |\boldsymbol{p}|\leq 2\boldsymbol{k}$ and $K \in \mathrm{Child}(J)$. It is straightforward from definition of $|c(\boldsymbol{p}|J,\oupgamma_J)|$ that
    \begin{equation}
    \label{eqn:part_a_c_bounded_when_p_equal_0}
        c(\boldsymbol{0}|J,\oupgamma_J) = \sum_{K \in \mathrm{Child}(J)} c_J(\boldsymbol{0}) = \sum_{K \in \mathrm{Child}(J)} \ouppi_{K} = \ouppi_{J} \preccurlyeq 1. 
    \end{equation}
    In addition, noting that $\varepsilon_K \leq \varepsilon_J$ and $\|\oupgamma_K - \oupgamma_J\| \preccurlyeq \varepsilon_J$, by using induction hypothesis, we have 
    \begin{equation}
    \label{eqn:part_a_c_bounded_by_max}
        |c(\boldsymbol{p}|J,\oupgamma_J)| \preccurlyeq \max_{\substack{K \in \mathrm{Child}(J)\\ 0 \leq |\boldsymbol{\ell}|,\,\boldsymbol{\ell} \preceq \boldsymbol{p}}} \left|c(\boldsymbol{\ell}|K,\oupgamma_K)\left(\dfrac{\varepsilon_K}{\varepsilon_J}\right)^{|\boldsymbol{\ell}|}\right| \preccurlyeq  1. 
    \end{equation}
Next, we prove that 
    \begin{equation}
    \label{eqn:part_b_multi_taylor_expansion}
        \max_{|J|\leq |\boldsymbol{p}| \leq 2\boldsymbol{k}}|c(\boldsymbol{p}|J,\oupgamma_J)| \preccurlyeq \max_{K \in \mathrm{Child}(J)}\max_{\boldsymbol{\ell}\leq |K|-1} \left|c(\boldsymbol{\ell}|K,\oupgamma_K)\left(\dfrac{\varepsilon_K}{\varepsilon_J}\right)^{|\boldsymbol{\ell}|}\right| \asymp \max_{0\leq |\boldsymbol{p}| < |J|}|c(\boldsymbol{p}|J,\oupgamma_J)|. 
    \end{equation}
    For the left-hand side, we prove by induction. When $J$ is a leaf, this equality automatically holds. Suppose that this left-hand side inequality holds for each node $K\in \mathrm{Child}(J)$. Note that from argument in equation \eqref{eqn:part_a_c_bounded_by_max}, for any $|J| \leq |\boldsymbol{p}|\leq 2\boldsymbol{k}$ we have 
 \begin{align*}
|c(\boldsymbol{p}|J,\oupgamma_J)|
&\preceq
\max_{\substack{
K \in \mathrm{Child}(J)\\
0 \leq |\boldsymbol{\ell}|,\,
\boldsymbol{\ell} \preceq \boldsymbol{p}
}}
\left|c(\boldsymbol{\ell}|K,\oupgamma_K)
\left(\dfrac{\varepsilon_K}{\varepsilon_J}\right)^{|\boldsymbol{\ell}|}\right|
\\
&\leq
\max_{\substack{
K \in \mathrm{Child}(J)\\
\boldsymbol{0} \preceq \boldsymbol{\ell},\
|\boldsymbol{\ell}| \leq |K|-1
}}
\left|c(\boldsymbol{\ell}|K,\oupgamma_K)
\left(\dfrac{\varepsilon_K}{\varepsilon_J}\right)^{|\boldsymbol{\ell}|}\right|
+
\max_{\substack{
K \in \mathrm{Child}(J)\\
\boldsymbol{\ell} \preceq \boldsymbol{p},\
|\boldsymbol{\ell}| \geq |K|
}}
\left|c(\boldsymbol{\ell}|K,\oupgamma_K)
\left(\dfrac{\varepsilon_K}{\varepsilon_J}\right)^{|\boldsymbol{\ell}|}\right|
\\
&\leq
\max_{\substack{
K \in \mathrm{Child}(J)\\
\boldsymbol{0} \preceq \boldsymbol{\ell},\
|\boldsymbol{\ell}| \leq |K|-1
}}
\left|c(\boldsymbol{\ell}|K,\oupgamma_K)
\left(\dfrac{\varepsilon_K}{\varepsilon_J}\right)^{|\boldsymbol{\ell}|}\right|
+
\max_{K \in \mathrm{Child}(J)}
\|c_K\|
\left(\dfrac{\varepsilon_K}{\varepsilon_J}\right)^{|K|}
\\
&\preceq
\max_{\substack{
K \in \mathrm{Child}(J)\\
0 \leq |\boldsymbol{\ell}| \leq |K|-1
}}
\left|c(\boldsymbol{\ell}|K,\oupgamma_K)
\left(\dfrac{\varepsilon_K}{\varepsilon_J}\right)^{|\boldsymbol{\ell}|}\right|
+
\max_{K \in \mathrm{Child}(J)}
\max_{0 \leq |\boldsymbol{\ell}| \leq |K|-1}
|c(\boldsymbol{\ell}|K,\oupgamma_K)|
\left(\dfrac{\varepsilon_K}{\varepsilon_J}\right)^{|K|}
\\
&\hspace{8cm}\text{ (by induction hypothesis for node $K$)}
\\
&\preceq
\max_{\substack{
K \in \mathrm{Child}(J)\\
0 \leq |\boldsymbol{\ell}| \leq |K|-1
}}
\left|c(\boldsymbol{\ell}|K,\oupgamma_K)
\left(\dfrac{\varepsilon_K}{\varepsilon_J}\right)^{|\boldsymbol{\ell}|}\right|.
\end{align*}
    For the right-hand side , we decompose $c(\boldsymbol{p}|J,\oupgamma_J)$ into two terms $c^{(1)}(\boldsymbol{p}|J,\oupgamma_J)$ and $c^{(2)}(\boldsymbol{p}|J,\oupgamma_J)$: 
    \begin{align*}
       c^{(1)}(\boldsymbol{p}|J,\oupgamma_J) &= \sum_{K \in \mathrm{Child}(J)}\sum_{\substack{0 \leq |\boldsymbol{\ell}|\leq |K|-1}}c(\boldsymbol{\ell}|K,\oupgamma_K)\left(\dfrac{\varepsilon_K}{\varepsilon_J}\right)^{|\boldsymbol{\ell}|} \dfrac{(\oupgamma_K-\oupgamma_J)^{\boldsymbol{p}-\boldsymbol{\ell}}}{\varepsilon_J^{|\boldsymbol{p}-\boldsymbol{\ell}|}(\boldsymbol{p}-\boldsymbol{\ell})!}\boldsymbol{1}_{\boldsymbol{\ell} \preceq \boldsymbol{p}}\\
        c^{(2)}(\boldsymbol{p}|J,\oupgamma_J) &= \sum_{K \in \mathrm{Child}(J)}\sum_{|\boldsymbol{\ell}| \geq |K|,\boldsymbol{\ell}\preceq \boldsymbol{p}}c(\boldsymbol{\ell}|K,\oupgamma_K)\left(\dfrac{\varepsilon_K}{\varepsilon_J}\right)^{|\boldsymbol{\ell}|} \dfrac{(\oupgamma_K-\oupgamma_J)^{\boldsymbol{p}-\boldsymbol{\ell}}}{\varepsilon_J^{|\boldsymbol{p}-\boldsymbol{\ell}|}(\boldsymbol{p}-\boldsymbol{\ell})!}\boldsymbol{1}_{\boldsymbol{\ell} \preceq \boldsymbol{p}}.
    \end{align*}
    Let $\phi_{K} = (\oupgamma_K-\oupgamma_J)/\varepsilon_J$, then since $\|\phi_K - \phi_{K'}\| = \|\oupgamma_K-\oupgamma_{K'
    }\|/\varepsilon_J \asymp 1$ for $K \neq K'$, there exists an $\epsilon > 0$ such that $\{\phi_{K}\}_{K \in \mathrm{Child}(J)}$ are $\epsilon$-separate in sense of Section \ref{sec:matrix_utils_separation_lemmma}. Using Corollary \ref{coro:epsilon_separation} for $A(\{\phi_K\}_{K \in \mathrm{Child}(J)})$ and $\Lambda = (\lambda_{K,\boldsymbol{\ell}})_{K \in \mathrm{Child}(J), 0 \leq |\boldsymbol{\ell}| < |K|}$, we have 
    \begin{equation*}
        \max_{0 \leq |\boldsymbol{p}| < |J|} | c^{(1)}(\boldsymbol{p}|J,\oupgamma_J)| \asymp \max_{\substack{K \in \mathrm{Child}(J)\\ 0\leq  |\boldsymbol{\ell}| < |K|}} \left|c(\boldsymbol{\ell}|K,\oupgamma_K)\left(\dfrac{\varepsilon_K}{\varepsilon_J}\right)^{|\boldsymbol{\ell}|}\right|. 
    \end{equation*}
In addition, note that by applying induction hypothesis to node $K$
that $\|c_K\| = \max_{0 \leq |\boldsymbol{\ell}|<|K|}
|c(\boldsymbol{\ell}|K,\oupgamma_K)|$,
the same argument also yields
    \begin{equation}
    \label{eqn:max_c_first_dominate_norm}
        \max_{0\leq |\boldsymbol{p}| < |J|} | c^{(1)}(\boldsymbol{p}|J,\oupgamma_J)| \succeq \max_{K \in \mathrm{Child}(J)}\left[\|c_{K}\|\left(\dfrac{\varepsilon_K}{\varepsilon_J}\right)^{|K|-1}\right]
    \end{equation}
    The second component $c^{(2)}(\boldsymbol{p}|J,\oupgamma_J)|$ is asymptotically negligible compared with $c^{(1)}_J(\boldsymbol{p})$. In fact, it is obvious that 
    \begin{equation}
    \label{eqn:max_c_first_dominated_norm}
        \max_{0\leq |\boldsymbol{p}| < |J|} |c^{(2)}(\boldsymbol{p}|J,\oupgamma_J)| \preceq \max_{K \in \mathrm{Child}(J)} \left[\|c_K\|\left(\dfrac{\varepsilon_K}{\varepsilon_J}\right)^{|K|}\right]. 
    \end{equation}
    As a result, we have $c(\boldsymbol{p}|J,\oupgamma_J)$ is asymptotically equivalent relative to $c^{(1)}(\boldsymbol{p}|J,\oupgamma_J)$, i.e. $c(\boldsymbol{p}|J,\oupgamma_J) \asymp c^{(1)}(\boldsymbol{p}|J,\oupgamma_J)$, which completes the proof of the relation in equation~\eqref{eqn:dung_prop_5_c_J_max}. 

 \textbf{Step 3: Lower bound for the coefficient norm.} We prove the lower bound in equation~\eqref{eqn:dung_prop_5_c_J_norm} by induction. We can verify easily when $J$ is a leaf. Suppose that the estimation~\eqref{eqn:dung_prop_5_c_J_norm} holds for all children $K$ of $J$. According to equations~\eqref{eqn:part_a_c_bounded_when_p_equal_0} and~\eqref{eqn:max_c_first_dominate_norm}, we have 
    \begin{equation}
    \label{eqn:c_J_dominate_c_K}
         \|c_J\| \succeq \max_{K \in \mathrm{Child}(J)} \left[\|c_K\|\left(\dfrac{\varepsilon_K}{\varepsilon_J}\right)^{|K|-1}\right] \vee \max_{K \in \mathrm{Child}(J)} |\ouppi_{K}|\\
    \end{equation}
    In addition, using the recurrent hypothesis for each node $K\in \mathrm{Child}(J)$, we have 
    \begin{align*}
        \|c_J\| &\succeq \max_{K \in \mathrm{Child}(J)} \left[\|c_K\|\left(\dfrac{\varepsilon_K}{\varepsilon_J}\right)^{|K|-1}\right] \vee \max_{K \in \mathrm{Child}(J)} |\ouppi_{K}| \\
        &\succeq \max_{K \in \mathrm{Child}(J)}\max_{F\in \mathrm{Desc}(K)} \left[|\ouppi_F|\left(\dfrac{\varepsilon_{F^\uparrow}}{\varepsilon_K}\right)^{|K|-1}\left(\dfrac{\varepsilon_K}{\varepsilon_J}\right)^{|K|-1}\right] \vee \max_{K \in \mathrm{Child}(J)} |\ouppi_{K}|\\
        &\geq \max_{F\in \mathrm{Desc}(J)\setminus\mathrm{Child}(J)} \left[|\ouppi_F|\left(\dfrac{\varepsilon_{F^{\uparrow}}}{\varepsilon_J}\right)^{|J|-1}\right] \vee \max_{K \in \mathrm{Child}(J)} |\ouppi_{K}|\\
        &= \max_{F\in \mathrm{Desc}(J)\setminus\mathrm{Child}(J)} \left[|\ouppi_F|\left(\dfrac{\varepsilon_{F^{\uparrow}}}{\varepsilon_J}\right)^{|J|-1}\right] \vee \max_{K \in \mathrm{Child}(J)} \left[|\ouppi_{K}| \left(\dfrac{\varepsilon_{K^{\uparrow}}}{\varepsilon_J}\right)^{|J|-1}\right] \\
        &= \max_{K \in \mathrm{Desc}(J)}\left[|\ouppi_K|\left(\dfrac{\varepsilon_{K^{\uparrow}}}{\varepsilon_J}\right)^{|J|-1}\right].
    \end{align*}

\textbf{Step 4: Uniform control of the remainder.} Lastly, we prove that $\sup_{x\in \mathbb{R}^{\bar{d}}}R(x\mid J) = o(\|c_J\|\varepsilon^{2\boldsymbol{k}}_{J})$ by induction. When $J$ is a leaf, the remainder $R(x\mid J) = 0$ and the claim follows immediately. Now, suppose that this property holds for all $K \in \mathrm{Child}(J)$. We decompose $R(x,J)$ into two terms $R(x\mid J) = R^{(1)}(x\mid J) + \varepsilon_J^{2\boldsymbol{k}}R^{(2)}(x\mid J)$, where 
    \begin{align*}
        R^{(1)}(x\mid J) &= \sum_{K\in \mathrm{Child}(J)} R(x\mid K)\\
        R^{(2)}(x\mid J) &= \sum_{K\in \mathrm{Child}(J)}\sum_{0\leq |\boldsymbol{\ell}| \leq 2\boldsymbol{k}}c(\boldsymbol{\ell}|K,\oupgamma_K)\left(\dfrac{\varepsilon_{K}}{\varepsilon_{J}}\right)^{|\boldsymbol{\ell}|} \overline{R}_{\boldsymbol{\ell},KJ}.
    \end{align*}
    For the first term, using induction hypothesis, we have 
    \begin{align*}
        \sup_{x\in \mathbb{R}^{\bar{d}}} |R^{(1)}(x\mid J)|\leq \sum_{K\in \mathrm{Child}(J)}\sup_{x\in \mathbb{R}^{\bar{d}}}|R(x\mid K)| \preceq \max_{K\in \mathrm{Child}(J)}[o(\|c_K\|\varepsilon^{2\boldsymbol{k}}_{K})] 
    \end{align*}
    For the second term, we have from uniform continuity condition of derivative in equation~\eqref{eqn:dung_prop_5_uniform_continuity} and estimation~\eqref{eqn:part_b_multi_taylor_expansion} that
    \begin{align*}
        \sup_{x\in \mathbb{R}^{\bar{d}}} |R^{(2)}(x\mid J)| &= \sum_{K\in \mathrm{Child}(J)}\sum_{0\leq |\boldsymbol{\ell}| \leq 2\boldsymbol{k}}c(\boldsymbol{\ell}|K,\oupgamma_K)\left(\dfrac{\varepsilon_{K}}{\varepsilon_{J}}\right)^{|\boldsymbol{\ell}|} \overline{R}_{\boldsymbol{\ell},KJ}  \preceq \|c_J\|o(1) . 
    \end{align*}
    As a result, we have 
    \begin{align*}
        \|R(\cdot,J)\|_{\infty} \preceq \varepsilon_J^{2\boldsymbol{k}} \left\{\max_{K \in \mathrm{Child}(J)}\left[o\left(\|c_K\|\left(\dfrac{\varepsilon_K}{\varepsilon_J}\right)^{2\boldsymbol{k}}\right)\right]+\|c_J\|o(1)\right\}. 
    \end{align*}
    In addition, equation \eqref{eqn:c_J_dominate_c_K} states that $\|c_J\|$ dominates $\|c_K\|(\varepsilon_K/\varepsilon_J)^{2\boldsymbol{k}}$, thus $\sup_{x\in \mathbb{R}^{\bar{d}}}|R(x\mid J)| = o(\|c_J\|\varepsilon^{2\boldsymbol{k}}_{J})$, which proves the claim for $J$.

\end{proof}


\section{Proof of Theorem \ref{thm:lower_bound}}
\label{app:theorem_2_proof}
After an orthogonal change of coordinates, we may assume without loss of generality that $\upsilon=e_1=(1,0,\ldots,0)^\top$ in Assumption $D(2d_0+2)$. This is legitimate because orthogonal transformations preserve Euclidean distances and, consequently, the Wasserstein distance $W_1$.

\begin{lemma}[Local Asymptotic Normality for Theorem \ref{thm:lower_bound}]
\label{lemma:LAN_property_for_upper_bound}
Let $\Gamma$ be a compact subset of $\mathbb R^{d}$, and let
$G_0\in\mathcal G_{k_0}$ have a support point
$\theta_0\in\operatorname{int}(\Theta)$. Under assumption $D(2d_0+2)$, there exists a family $\{G_n(u):n\geq 1,\ u\in\mathbb R\} \subseteq\mathcal G_k$ satisfying the following properties.
\begin{itemize}
    \item [(a)] For each pair of distinct $u,v \in \mathbb{R}$, we have 
    \begin{equation}
        W_1(G_n(u),G_n(u')) \underset{u,v}{\gtrsim} n^{-1/(4d_0+2)} \underset{u}{\gtrsim} W_1(G_n(u),G_n(0)). 
    \end{equation}
    \item [(b)] Let $f_{n,u} = \bigotimes^n_{i=1} f(\cdot,G_n(u))$ denotes the joint density of an $n$-sample drawn from $f(\cdot;G_n(u))$. There exists a sequence of positive real numbers $U_n \to \infty$ such that the sequence of experiments $\mathcal{H}_n = (f_{n,u})_{u \in [-U_n,U_n]}$ is locally asymptotically normal at $u = 0$.  
That is, there exist a sequence of random variables  $\{ Z_n \}$ with $Z_n \xrightarrow{\mathcal D} \mathcal N(0,1)$ and a sequence of positive numbers  $\{ \Gamma_n \}$ bounded away from $0$ and infinity such that,   under $X^{(n)}\sim f_{n,0}$ and for any $u \in \mathbb{R}$,
    \begin{equation*}
        \log
        \frac{f_{n,u}\bigl(X^{(n)}\bigr)}
             {f_{n,0}\bigl(X^{(n)}\bigr)}
        -
        uZ_n\sqrt{\Gamma_n} + \frac{u^2}{2}\Gamma_n
        \xrightarrow{\mathbb P_{n,0}}0.
    \end{equation*}
\end{itemize}
\end{lemma}

\begin{proof}
   Write $G_0 = \sum_{j=1}^{k_0-1}\pi_{0j}\delta_{\gamma_{0j}} + \pi_{00}\delta_{\gamma_{00}}$. Let $\{H(u):u\in \mathbb{R}\}$ be the one-dimensional family of mixing measures constructed in \cite[Theorem 6.1]{heinrich2018}. Define 
    \begin{equation*}
        T_n(t):=\gamma_{00}+\varepsilon_n t e_1, \quad H_n(u) = (T_n)_{\#}H(u), 
    \end{equation*}
    where $(T_n)_{\#}H(u)$ denotes the pushforward of $H(u)$ under $T_n$. We then set 
    \begin{equation*}
        G_n(u) = G_n(u) = \sum_{j=1}^{k_0-1} \pi_{0j}\delta_{\gamma_{0j}} + \pi_{00}H_n(u)
    \end{equation*}
    
    For part (a), the Kantorovich–Rubinstein dual representation gives 
    \begin{equation*}
        W_1(G_n(u),G_0) = \pi_{00}W_1(H_n(u),\delta_{\gamma_{00}}) = \pi_{00}\varepsilon_nW_1(H(u),\delta_0).
    \end{equation*}
    The properties of the one-dimensional construction therefore imply that 
    \begin{equation*}
        W_1(G_n(u),G_n(v)) \underset{u,v}{\gtrsim} \varepsilon_n. 
    \end{equation*}
    Similarly, $$W_1(G_n(u),G_n(v)) = \pi_{00} \varepsilon_nW_1(H(u),H(v))\underset{u,v}{\lesssim} \varepsilon_n.$$ 
    Taking $\varepsilon_n = n^{-1/(4d_0+2)}$ proves part (a). 

    For part (b), consider the one-dimensional submodel 
    \begin{equation*}
        \tilde{f}(x\mid t) = f(x\mid \gamma_{00} + te_1).
    \end{equation*}
    Under Assumption $\mathrm D(p)$, this submodel satisfies the regularity conditions required in \cite[Appendix A]{heinrich2018}. We may therefore apply the one-dimensional LAN argument therein to $\tilde{f}$ and the family $\{H(u):u \in\mathbb{R}\}$. It follows that there exist a sequence $U_n\to\infty$, random variables $Z_n$ satisfying $ Z_n \xrightarrow{\mathcal D} \mathcal N(0,1)$ and positive numbers $\Gamma_n$ such that, under $X^{(n)}\sim f_{n,0}$, for every fixed $u\in\mathbb R$
    \begin{equation*}
         \log
        \frac{f_{n,u}\bigl(X^{(n)}\bigr)}
             {f_{n,0}\bigl(X^{(n)}\bigr)}
        -
        uZ_n\sqrt{\Gamma_n} + \frac{u^2}{2}\Gamma_n
        \xrightarrow{\mathbb P_{n,0}}0.
    \end{equation*}
   Moreover,
   \begin{equation*}
        0< \liminf_{n\to\infty}\Gamma_n \leq \limsup_{n\to\infty}\Gamma_n <\infty.
    \end{equation*}
    This proves part (b) and completes the proof.

\end{proof}

\begin{proof}[Proof of Theorem \ref{thm:lower_bound}] 
We divide the proof into two steps.

\textbf{Step 1: } We show that for $\varepsilon_n = n^{-1/(4d_0+2)+ \kappa}$ for some $\kappa > 0$, and any sequence of estimator $\check{G}_n$: 
\begin{equation}
\label{eqn:minimax_rate_for_W_1}
    \underset{\substack{G \in \mathcal{G}_{k}(\Gamma)\\W_1(G,G_0)\leq \varepsilon_n}}{\mathbb{E}_G}[W_1(G,\check{G}_n)] \gtrsim n^{-1/(4d_0+2)}. 
\end{equation}
Once equation~\eqref{eqn:minimax_rate_for_W_1} is established, Jensen’s inequality gives
\begin{equation}
    \label{eqn:minimax_rate_for_W_1^q}
    \underset{\substack{G \in \mathcal{G}_{k}(\Gamma)\\W_1(G,G_0)\leq \varepsilon_n}}{\mathbb{E}_G}[W^{2d_0+1}_1(G,\check{G}_n)] \gtrsim n^{-1/2}. 
\end{equation}

We now prove the lower bound in equation~\eqref{eqn:minimax_rate_for_W_1} by adapting the Le Cam two-point argument \cite{lecam1986asymptotic} in \cite[Section 6.1]{heinrich2018}. Let $G_n(0)$ and $G_n(1)$ be the mixing measures constructed in Lemma \ref{lemma:LAN_property_for_upper_bound}. By Lemma \ref{lemma:LAN_property_for_upper_bound}(a), both measures belong to the ball $\{G: W_1(G,G_0) \leq \varepsilon_n\}$ for all sufficiently large $n$, and there exists $a > 0$ such that 
\begin{equation*}
    W_1(G_n(0),G_n(1)) \geq 2an^{-1/(4d_0+2)}.
\end{equation*}

Let $\mathbb{P}_{n,u}$ denote the probability measure associated with the $n$-sample density $f_{n,u}$, for $u\in\{0,1\}$. The contiguity argument in \cite[Section 6.1]{heinrich2018}, together with Lemma \ref{lemma:LAN_property_for_upper_bound}(b), implies that for every measurable event $B$,
\begin{equation}
\label{eqn:dung_theorem_2_step_1_equivalent_bw_prob_measure}
    \mathbb{P}_{n,0}(B) \geq \frac{3}{4} \Rightarrow \mathbb{P}_{n,1}(B) \gtrsim e^{-\Gamma_n/2}. 
\end{equation}
Define $A = \{W_1(G_n(1),\check{G}_n) \geq an^{-1/(4d_0+2)}\}$, by the triangle inequality and equation~\eqref{eqn:dung_theorem_2_step_1_equivalent_bw_prob_measure} on $A^c$ we have 
\begin{equation*}
    W_1(G_n(0),\check{G}_n) \geq an^{-1/(4d_0+2)}. 
\end{equation*}
We now distinguish two cases. If $\mathbb{P}_{n,0}(A^c) \geq \frac{1}{4}$, then taking $G = G_n(0)$, we obtain 
\begin{equation*}
    \mathbb{E}_{G}[W_1(G,\check{G}_n)] \geq  \mathbb{E}_{G}[W_1(G,\check{G}_n)\boldsymbol{1}_{A^c}] \geq \frac{a}{4}n^{-1/(4d_0+2)}. 
\end{equation*}
Otherwise, $P_{n,0}(A^c)<1/4$, and hence $P_{n,0}(A)>3/4$. Applying equation~\eqref{eqn:dung_theorem_2_step_1_equivalent_bw_prob_measure} with $B=A$, we obtain $\mathbb{P}_{n,1}(A) \gtrsim e^{-\Gamma_n/2}$. Therefore, for $G = G_n(1)$, 
\begin{equation*}
    \mathbb{E}_G[W_1(G,\check{G}_n)] \geq \mathbb{E}_G[W_1(G,\check{G}_n)\boldsymbol{1}_A] \gtrsim ae^{-\Gamma_n/2}n^{-1/(4d_0+2)}. 
\end{equation*}
Since $\limsup_{n\to\infty}\Gamma_n<\infty$, the factor $e^{-\Gamma_n/2}$ is bounded away from zero. Thus, in either case,
\begin{equation*}
    \sup_{G \in \{G_n(0),G_n(1)\}} \mathbb{E}_G[W_1(G,\check{G}_n)] \gtrsim n^{-1/(4d_0+2)},
\end{equation*}
which proves the lower bound in equation~\eqref{eqn:minimax_rate_for_W_1}.



\textbf{Step 2: } Using Lemma \ref{lemma:partial_OT_dominate_wasserstein}, we have $\mathsf{VPOT}_{G_0,\delta}(G,G') \gtrsim  W^{2d_0+1}_1(G,G')$. Combining this estimation with equation~\eqref{eqn:minimax_rate_for_W_1^q} and H\"{o}lder inequality gives the conclusion of the proof. 
\end{proof}

\begin{lemma}[Lower bound for the Voronoi-based POT discrepancy]
 \label{lemma:partial_OT_dominate_wasserstein} 
For measure $G_0$, let $\pi_{\min} =  \min_{1\leq i \leq k_0} \pi_{0i}$. Then, under assumption of Theorem \ref{thm:lower_bound}, there exists a constant \(C=C(\Gamma,G_0,q,\kappa)>0\) such that, for every \(G,G'\in \mathcal{G}_{\leq k}(\Gamma)\) satisfying $W_1(G,G_0) \leq \pi_{\min}\cdot\delta/2$, the Voronoi-based POT discrepancy satisfies
$\mathsf{VPOT}_{G_0,\delta}(G,G') \geq C\cdot W^{2d_0+1}_1(G,G')$. Equivalently, 
            \begin{equation}
        \label{eqn:dominance_of_partial_OT_over_wasserstein}
        \sum_{i = 1}^{k_0}\inf_{1\leq j\leq 2k} \mathsf{POT}^{r_{i,j}}_{1}\left(G|_{\mathcal{V}^{i,j}_{G_0,\delta}},G'|_{\mathcal{V}^{i,j}_{G_0,\delta}}\right)\underset{\Theta,G_0,q,\kappa}{\gtrsim} W^{2d_0+1}_1(G,G').
    \end{equation}

\end{lemma}

\begin{proof}
Arguing by contradiction, suppose that equation~\eqref{eqn:dominance_of_partial_OT_over_wasserstein} does not hold. Then, there exist sequences of measures $(G_n)_{n},(G'_n)_{n}$ in $\mathcal{G}_{\leq k}(\Gamma)$ such that $W_1(G_n,G_0)\leq \pi_{\min}\cdot \delta/2$ and 
\begin{equation}
\label{eqn:lemma_4_dung_contradiction}
    \sum_{i = 1}^{k_0}\inf_{1\leq j\leq 2k} \mathsf{POT}^{r_{i,j}}_{1}\left(G_n|_{\mathcal{V}^{i,j}_{G_0,\delta}},G'_n|_{\mathcal{V}^{i,j}_{G_0,\delta}}\right)/W^{2d_0+1}_1(G_n,G'_n) \to 0.
\end{equation}

After passing to a subsequence, we may assume that $G_n \to G$ and $G'_n\to G'$ and that their convergence behavior admits a coarse-graining tree $\mathcal{T}$ with root $J_r$. By Proposition \ref{prop:POT_based_on_tree}, 
\begin{equation}
    \label{eqn:partial_OT_dominance_W_1_estimate}
    W^{2d_0+1}_1(G_n,G_n') \asymp \max_{J \in \mathrm{Desc}(J_r)}|\ouppi_{J}|^{2d_0+1}\varepsilon^{2d_0+1}_{J^\uparrow}.
\end{equation}

Next, we estimate the sum $\sum_{i = 1}^{k_0}\inf_{1\leq j\leq 2k} \mathsf{POT}^{r_{i,j}}_{1}\left(G_n|_{\mathcal{V}^{i,j}_{G_0,\delta}},G'_n|_{\mathcal{V}^{i,j}_{G_0,\delta}}\right)$. For any node $J \in \mathrm{Child}(J_r)$, set $$B_J = |\ouppi_{J}| \lor \max_{K \in \mathrm{Desc}(J)} |\ouppi_{K}|\varepsilon_{K^{\uparrow}}.$$ 
For further presentation, by splitting the weight from one atom to create two atoms, we denote $G_0 = \sum_{i=1}^{k_0}\pi_{0i}\delta_{\gamma_{0i}}$, $G_n = \sum_{i=1}^{k}\pi_{ni}\delta_{\gamma_{ni}}$, and $G'_n = \sum_{i=1}^{k}\pi'_{ni}\delta_{\gamma'_{ni}}$, where $\pi_{0i},\pi_{ni},\pi_{ni}$ are positive reals. Let $\pi_{\min} = \min_{1\leq i\leq k_0}\pi_{0i} > 0$, we consider two situations according to the value of $W_1(G,G')$.

\textbf{Case 1: } $W_1(G_0,G')<\pi_{\min}\cdot\delta/2$. Suppose that this inequality holds for all sufficiently large \(n\). Both \(G_n\) and \(G_n'\) then satisfy the conclusion of Lemma \ref{lemma:dung_separation}. In particular, each ball $B(\theta_{0i},\delta/2)$ contains at least one support point of $G_n$ and at least one support point of $G'_n$. By the definition of the Voronoi cell extension,
\begin{equation*}
    r_{i,j} \leq 2(k-k_0)+1 = 2d_0+1, \quad 1\leq i\leq k_0, \quad 1\leq j \leq 2k. 
\end{equation*}
We divide this case into two subcases.
 
\textit{Case 1.1: } $\varepsilon_{J_r} \to 0$.

In this case, all support points of $G_n$ and $G'_n$ converge to the same point $\gamma \in \Gamma$. Lemma \ref{lemma:dung_separation} and the separation condition on the support points of $G_0$ therefore imply that $k_0 = 1$. Consequently, the unique Voronoi cell is the entire parameter space, and hence 
\begin{equation*}
    \mathsf{VPOT}_{G_0,\delta}(G,G') = W_1^{|\mathcal{S}(G)| + |\mathcal{S}(G')|-1}(G_n,G_n').
\end{equation*}
In addition, as $G,G' \in \mathcal{G}_{\leq}(\Gamma)$, $|\mathcal{S}(G)| + |\mathcal{S}(G')|-1 \leq 2k-1$. As $W_1(G_n,G_n')\leq \mathrm{Diam}(\Gamma)$, we have 
\begin{equation*}
     \mathsf{VPOT}_{G_0,\delta}(G,G') \gtrsim W_1^{2k-1}(G_n,G_n') = W_1^{2d_0-1}(G_n,G_n'),
\end{equation*}
which contradicts equation~\eqref{eqn:lemma_4_dung_contradiction}. 

\textit{Case 1.2: } $\varepsilon_{J_r} \not\to 0$. 
After passing to a further subsequence, we may assume that $\varepsilon_{J_r} \asymp 1$. 
Thus, the children of \(J_r\) represent distinct limiting clusters. . Let $\rho_i = \max_{1\leq j \leq 2k}r_{i,j}$, and for each node $J \in \mathrm{Child}(J_r)$  denote the limiting location of the points indexed by \(J\). We prove that for each $1\leq j \leq 2k$
\begin{equation}
\label{eqn:dung_lemma_4_prelim_estimation_for_sup_POT}\mathsf{POT}^{r_{i,j}}_{1}\left(G_n|_{\mathcal{V}^{i,j}_{G_0}},G'_n|_{\mathcal{V}^{i,j}_{G_0}}\right) \gtrsim \max_{\substack{J\in \mathrm{Child}(J_r)\\  \theta_J\in\mathcal{V}^{i}_{G_0}}}B_J^{\rho_i} \gtrsim \max_{\substack{J\in \mathrm{Child}(J_r)\\  \theta_J\in\mathcal{V}^{i}_{G_0}}}B_J^{2d_0+1}. 
\end{equation}
Let $\mathcal{T}_{ij}$ be the coarse-graining tree associated with the restricted measures $G_n|_{\mathcal{V}^{i,j}_{G_0,\delta}}$ and $G'_n|_{\mathcal{V}^{i,j}_{G_0,\delta}}$. For each node $J$ such that $\theta_J\in\mathcal{V}^{i}_{G_0}$, for $n$ sufficiently large, all the points inside the node $J$ belongs to neighborhood $\mathcal{V}^{i,j}_{G_0,\delta}$ of $\mathcal{V}^{i}_{G_0}$, which means that $J$ is a node of $\mathcal{T}_{ij}$. Thus, Proposition \ref{prop:POT_based_on_tree} yields
\begin{equation*}
    \mathsf{POT}^{r_{i,j}}_{1}\left(G_n|_{\mathcal{V}^{i,j}_{G_0,\delta}},G'_n|_{\mathcal{V}^{i,j}_{G_0,\delta}}\right) \gtrsim B_J^{r_{i,j}} \gtrsim B_J^{\rho_i} \gtrsim B_J^{2d_0+1},
\end{equation*}
where the last two inequalities follow from the fact that $|\uppi_K|\leq 1$ and $\varepsilon_K \leq \mathrm{Diam}(\Gamma)$ for each node $K$. Thus, equation~\eqref{eqn:dung_lemma_4_prelim_estimation_for_sup_POT} holds, and by summing up this result for all $i$, we achieve that 
\begin{equation*}
     \sum_{i = 1}^{k_0}\inf_{1\leq j\leq 2k} \mathsf{POT}^{r_{i,j}}_{1}\left(G_n|_{\mathcal{V}^{i,j}_{G_0,\delta}},G'_n|_{\mathcal{V}^{i,j}_{G_0,\delta}}\right) \gtrsim \max_{J \in \mathrm{Desc}(J_r)}B_J^{2d_0+1} = \max_{J \in \mathrm{Desc}(J_r)}|\ouppi_{J}|^{2d_0+1}\varepsilon^{2d_0+1}_{J^\uparrow},
\end{equation*}
which contradicts equations~\eqref{eqn:lemma_4_dung_contradiction} and~\eqref{eqn:partial_OT_dominance_W_1_estimate}. Thus, equation~\eqref{eqn:dominance_of_partial_OT_over_wasserstein} holds.

\textbf{Case 2: } $W_1(G_0,G') \geq \pi_{\min}\cdot\delta/2$.

The triangle inequality and the assumed bound on $W_1(G_n,G_0)$ give 
\begin{equation*}
    W_1(G_n,G_n') \geq W_1(G_0,G_n') -  W_1(G_0,G_n) \geq  \frac{\pi_{\min}\cdot\delta}{4}. 
\end{equation*}
On the other hand, the scales \(\varepsilon_{J^\uparrow}\) associated with proper descendants of the root tend to zero. Since \(|\bar\pi_J|\leq1\), asymptotic approximation~\eqref{eqn:partial_OT_dominance_W_1_estimate} therefore implies that  
\begin{equation*}
    W^{2d_0+1}_1(G_n,G_n') \asymp \max_{J \in \mathrm{Child}(J_r)}|\ouppi_{J}|^{2d_0+1}.
\end{equation*}
As a result, $\max_{J \in \mathrm{Child}(J_r)}|\ouppi_{J}| \asymp 1$, thus $\max_{J \in \mathrm{Desc}(J_r)}B_J \asymp 1$. Using an argument similar to that in Case 1.2, we also obtain the estimate in equation~\eqref{eqn:dung_lemma_4_prelim_estimation_for_sup_POT}, which yields 
\begin{equation*}
     \sum_{i = 1}^{k_0}\inf_{1\leq j\leq 2k} \mathsf{POT}^{r_{i,j}}_{1}\left(G|_{\mathcal{V}^{i,j}_{G_0,\delta}},G'|_{\mathcal{V}^{i,j}_{G_0,\delta}}\right) \asymp 1.
\end{equation*}
Because $\Gamma$ is compact, $W_1(G_n,G_n') \leq \mathrm{Diam}(\Gamma)$. Hence, 
\begin{equation*}
     \left(\sum_{i = 1}^{k_0}\inf_{1\leq j\leq 2k} \mathsf{POT}^{r_{i,j}}_{1}\left(G_n|_{\mathcal{V}^{i,j}_{G_0,\delta}},G'_n|_{\mathcal{V}^{i,j}_{G_0,\delta}}\right) \right)/W_1^{2d_0+1}(G_n,G_n') \overset{\Gamma,G_0}{\gtrsim} 1. 
\end{equation*}
Thus, equation~\eqref{eqn:lemma_4_dung_contradiction} cannot hold in this case.

Both cases lead to a contradiction, so equation~\eqref{eqn:dominance_of_partial_OT_over_wasserstein} holds. 
\end{proof}


\section{Proof of density estimation rate}
\label{proof:density_estimation_rate}

This appendix proves Proposition \ref{prop:MLE_estimation}. We use the standard empirical-process approach based on local bracketing entropy. The overall strategy follows \cite{vandeGeer-00}, which combines covering arguments with classical results for density estimation. We therefore recall only the notation needed for our mixture model and focus subsequently on the model-specific entropy bounds. Although the general argument is standard, we provide the details because the relevant entropy bounds depend on the particular parameterization of our mixture model. 

\subsection{Model classes and entropy notation}
Let $\Pi_k := \{c_1,\ldots,c_k:c_i \geq 0\ , c_1+\ldots +c_k = 1\}$ be the \((k-1)\)-dimensional probability simplex, and let $\Theta_k = \Pi_k \times \Gamma^k$ or the parameter space and recall the corresponding class of mixing measures by 
\begin{equation*}
    \mathcal{G}_{\leq k}(\Gamma) = \{G:G = \sum_{i=1}^k\pi_i \delta_{\gamma_i},(\pi_1\cdots\pi_k) \in \Pi_k, \gamma_i \in \Gamma\}
\end{equation*}
Each \(G\in\mathcal{G}_{\leq k}(\Gamma)\) induces the mixture density $p_G(x) = \sum_{i=1}^k\pi_if(x\mid \gamma_i)$, where $f(x\mid \gamma_i)$ belongs to the family of working distribution $\{f(\cdot \mid \gamma),\gamma \in \Gamma\}$. The induced density class is $\mathscr{P}_k(\Gamma) = \{p_{G}:G \in  \mathcal{G}_{\leq k}(\Gamma)\}$. Fix the true mixing measure \(G_*\in  \mathcal{G}_k(\Gamma)\), and let \(p_{G_*}\) be its density. For \(G\in\  \mathcal{G}_k(\Gamma)\), define $\overline{p}_{G} = (p_{G} + p_{G_0})/2$. The classes used in the local entropy argument are 
\begin{align*}
    \overline{\mathscr{P}}_k(\Gamma) =\{\overline{p}_{G}: p_G \in \mathscr{P}(\Gamma)\}, &\quad \overline{\mathscr{P}}_k^{1/2}(\Gamma) =\{\overline{p}^{1/2}_{G}: p_G \in \mathscr{P}_k(\Gamma)\},\\ 
    \overline{\mathscr{P}}^{1/2}_k(\Gamma,\varepsilon) =\{\overline{p}^{1/2}_{G}: &p_G \in \mathscr{P}_k(\Gamma),h(\overline{p}_{G},p_{G_0})\leq \varepsilon\}.
\end{align*}

We next specify our entropy conventions. For a metric space \((\mathscr{F},d)\), an $\varepsilon$-net of \((\mathscr{F},d)\) is a collection of balls of radius $\varepsilon$ whose union is exactly $\mathscr{F}$. Let $N(\varepsilon, \mathscr{P}, d)$ be 
\textit{covering number}, defined as the minimal cardinality of such an $\varepsilon$-net in \((\mathscr{F},d)\), and $H(\varepsilon, \mathscr{P}, d) = \log N(\varepsilon, \mathscr{P}, d)$ is the corresponding \textit{entropy number}.

When $\mathscr{F}$ is a family of density, unless otherwise specified, we choose metric $d$ to be the $\mathcal{L}^2(m)$ norm, where $m$ denotes the Lebesgue measure, and we abbreviate $N(\varepsilon,\mathscr{F},m)$ and $H(\varepsilon,\mathscr{F},m)$ for $N(\varepsilon,\mathscr{F},\|\cdot\|_{\mathcal{L}^2(m)})$ and $H(\varepsilon,\mathscr{F},\|\cdot\|_{\mathcal{L}^2(m)})$. Let $N_B(\varepsilon,\mathscr{F},d)$ be the \textit{bracketing number}, defined as the smallest integer $n$ such that there exists $n$ couples of function $\{f^{U}_i,f^V_i\}_{1\leq i\leq n}$ such that $d(f^{U}_i,f^V_i) < \varepsilon$ and for each $f \in \mathscr{F}$, $f^{U}_i\leq f\leq f^V_i$ for some $i \in [n]$. Likewise, we define the \textit{bracketing entropy} $H_B(\varepsilon,\mathscr{F},d) := \log N_B(\varepsilon,\mathscr{F},d)$. When the metric $d$ is induced by the $\mathcal{L}^2(m)$ norm, $H_B(\varepsilon,\mathscr{F},m)$ and $N_B(\varepsilon,\mathscr{F},m)$ are similarly used to denote bracketing number and bracketing entropy.

\subsection{Bracket Entropy Bounds}
We first derive covering number and bracketing entropy bounds for the density class \(\mathscr P(\Theta)\).

\begin{lemma}
    \label{lemma:bracket_entropy_estimation}
   Let $\Gamma$ be the parameter space defined above. For every $0 < \varepsilon<1/2$, the following estimations hold
    \begin{enumerate}
        \item [(1)] $\log N(\varepsilon,\mathscr{P}_k(\Gamma),\|\cdot\|_{\infty}) \lesssim \log(1/\varepsilon)$.
        \item [(2)] $H_B(\varepsilon,\mathscr{P}_k(\Gamma),h) \lesssim \log(1/\varepsilon)$.
    \end{enumerate}
\end{lemma}

\begin{proof}
    \textbf{Part (1)}. Recall that $\Theta_k = \Pi_k \times \Gamma^k$. Let $\mathscr{E}_{\varepsilon}(\Theta_k)$ be $\varepsilon$-net of $\Theta$ with respect to Euclidean metric. Thus, for every set such that for each $\theta \in \Theta_k$, there exists $\overline{\theta} \in \mathscr{E}_{\varepsilon}(\Theta_k)$ such that $\|\theta - \overline{\theta}\|_2 \leq \varepsilon$. Since $\Theta_k$ is a bounded finite-dimensional parameter space, a standard volumetric argument (see, for example, \cite[Lemma 6]{ho2022gaussian}) $\log |\mathscr{E}_{\varepsilon}(\Theta_k)| \lesssim \log(1/\varepsilon)$. 

\begingroup
\setlength{\emergencystretch}{2em}
Define the corresponding finite family of densities by
$\mathscr{E}_{\varepsilon}(\mathscr{P}_k(\Gamma))
:=\allowbreak
\{p_G:G = \allowbreak
\sum_{i=1}^k \lambda_i\delta_{\gamma_i},\ \allowbreak
(\lambda,\allowbreak\gamma_1,\allowbreak\ldots,\allowbreak\gamma_k)
\in \allowbreak
\mathscr{E}_{\varepsilon}(\Theta_k)\}$.
For any \(p_G\in\mathscr P_k(\Gamma)\), choose \(G'\)
whose parameter vector belongs to
\(\mathscr E_\varepsilon(\Theta_k)\)
and is within Euclidean distance \(\varepsilon\) of the
parameter vector of \(G\).
By the uniform Lipschitz property of the component densities
with respect to their parameters,
$\|p_G-p_{G'}\|_\infty
\lesssim \allowbreak \|G-G'\|_2
\leq \allowbreak \varepsilon$.
After adjusting the multiplicative constant in the covering radius,
\(\mathscr E_\varepsilon(\mathscr P_k(\Gamma))\)
is therefore an \(\varepsilon\)-net of \(\mathscr P_k(\Gamma)\).
Consequently,
$\log N(\varepsilon,\allowbreak
\mathscr{P}_k(\Gamma),\allowbreak
\|\cdot\|_{\infty})
\lesssim \allowbreak \log(1/\varepsilon)$,
which proves part (1).
\par
\endgroup

    \textbf{Part (2)}. Let \(\eta>0\), whose value will be chosen below. By part (1), there exists an \(\eta\)-net  $p_1,\ldots,p_N$ of $\mathscr{P}_k(\Gamma)$ in the uniform norm, where $\log N \lesssim \log(1/\eta)$. By the uniform polynomial-tail assumption, there exist constants \(c\geq1\) and \(\alpha>0\) such that every component density is bounded by the envelope
    \begin{equation*}
        H(x) = \begin{cases}
            c, \text{ when } \|x\| \leq 1,\\
            \dfrac{c}{\|x\|^{\bar{d}+\alpha}}, \text{ when }\|x\|> 1. 
        \end{cases}
    \end{equation*}
    In particular, $f(x\mid \gamma) \leq H(x)$ for every $x\in \mathbb{R}^{\bar{d}}$ and $\gamma \in \Gamma$, which means that the same bound holds for every mixture density in \(\mathscr P_k(\Gamma)\). For each $i \in [N]$ we construct the bracket $[p^{L}_i, p^{U}_i]$ as 
    \begin{equation*}
            p^{L}_i(x) = \max\{f_i(x)-\eta,0\},\quad
            p^{U}_i(x) = \min\{f_i(x)+\eta,H(x)\}.
    \end{equation*}
    If \(p\in\mathscr P_k(\Gamma)\) and \(\|p-p_i\|_\infty\leq\eta\), then $p^{L}_i(\cdot) \leq p(\cdot\mid\gamma) \leq p^{U}_i(\cdot)$. Therefore, the brackets $\{[p_i^L,p_i^U]\}_{i \in [n]}$ cover $\mathscr{P}_k(\Gamma)$. Moreover, $0 \leq p^{U}_i(x) - p^{L}_i(x) \leq \min\{2\eta,H(x)\}$. Set $B = (c/2\eta)^{1/(\bar{d}+\alpha)}$, the non-heavy tail assumption implies
    \begin{equation}
    \label{eqn:difference_between_bracket}
        \int_{\mathbb{R}^{\bar{d}}} (p^{U}_i(x) - p^{L}_i(x))dx \leq \int_{\|x\|\leq B}2\eta dx + \int_{\|x\|\geq B} \dfrac{c}{\|x\|^{\bar{d}+\alpha}} dx,
    \end{equation}
    Let \(v_{\bar{d}}\) denote the volume of the \(\bar{d}\)-dimensional unit ball and \(\sigma_{\bar{d}-1}\) the surface area of the unit sphere in \(\mathbb R^{\bar{d}}\). The first term in equation~\eqref{eqn:difference_between_bracket} equals $2\eta B^{\bar{d}}.V(\mathbb{B}^{\bar{d}})$, where  $V(\mathbb{B}^{\bar{d}})$ is the volume of the $\bar{d}$-dimensional unit ball. Using polar coordinates, the second term satisfies 
\begin{equation*}
        \int_{\|x\|\geq B} \dfrac{c}{\|x\|^{\bar{d}+\alpha}} dx = \int_{t\geq B}dt\int_{S(0,t)}\dfrac{c}{\|x\|^{\bar{d}+\alpha}}dx = c\cdot A(\mathbb{S}^{\bar{d}-1})\int_{t\geq B}t^{-1-\alpha}dt = c\alpha^{-1}\cdot A(\mathbb{S}^{\bar{d}-1})B^{-\alpha},
    \end{equation*}
    Since \(B=(c/(2\eta))^{1/(\bar{d}+\alpha)}\), both terms are of order \(\eta^{\alpha/(\bar{d}+\alpha)}\). Hence,
    \begin{equation*}
        \int_{\mathbb{R}^{\bar{d}}} (p^{U}_i(x) - p^{L}_i(x))dx \lesssim \eta^{\alpha/(\bar{d}+\alpha)}. 
    \end{equation*}
    It follows that $H_B(\eta^{\alpha/(\bar{d}+\alpha)},\mathscr{P}_k(\Gamma),\|\cdot\|_1) \lesssim \log(N) \lesssim \log (1/\eta)$. Finally, because $h^2 \leq \|\cdot\|_1$, every \(L_1\)-bracket of width \(\delta\) is a Hellinger bracket of width at most \(\sqrt{\delta}\). Thus, 
    \begin{equation*}
        H_B(\sqrt{\delta},\mathscr{P}_k(\Gamma),h) \leq H_B(\delta,\mathscr{P}_k(\Gamma),\|\cdot\|_1) \lesssim \log(1/\delta). 
     \end{equation*}
    Replacing \(\sqrt{\delta}\) by \(\varepsilon\) proves $H_B(\varepsilon,\mathscr{P}_k(\Gamma),\|\cdot\|_1) \lesssim \log(1/\varepsilon)$ and completes the proof. 
\end{proof}

\subsection{Empirical Process Analysis of the MLE}
Using standard arguments from empirical process theory (see, e.g., \cite{vandeGeer-00}), the local bracketing entropy is connected to the complexity of a class of distribution via the \textit{bracketing entropy integral}
\begin{equation*}
    \mathcal{J}_B\left(\varepsilon,  \overline{\mathscr{P}}_k^{1/2}(\Gamma,\varepsilon),m\right) = \left(\int_{\varepsilon^2/2^{13}}^{\varepsilon} H_B^{1/2}(u, \overline{\mathscr{P}}_k^{1/2}(\Gamma,\varepsilon),m)du\right)\vee \varepsilon,
\end{equation*}
where $u\vee \varepsilon = \max\{u,\varepsilon\}$. Then, for our setting, we have an estimation for $\mathcal{J}_B$, which is a key component for the proof of MLE consistency. 

\begin{lemma}
\label{lemma:dung_density_estimation_bracketing_integrl_bound}
There exist universal constants \(J>0\) and \(N\in\mathbb N\) such that, for every \(n\geq N\) and every $\varepsilon \geq (\log(n)/n)^{1/2}$, we have 
    \begin{equation}
\label{eqn:bracket_entropy_estimation}
    \mathcal{J}_B\left(\varepsilon,  \overline{\mathscr{P}}_k^{1/2}(\Gamma,\varepsilon),m\right)\leq J\sqrt{n}\varepsilon^2.
    \end{equation}
\end{lemma}
\begin{proof}
We derive the result from the entropy bound in Lemma \ref{lemma:bracket_entropy_estimation}. Since
$$ \overline{\mathscr P}_k^{1/2}(\Gamma,\varepsilon) \subseteq \overline{\mathscr P}_k^{1/2}(\Gamma), $$
monotonicity of the bracketing number gives $H_B(u, \overline{\mathscr{P}}_k^{1/2}(\Gamma,\varepsilon),m) {\leq} H_B(u, \overline{\mathscr{P}}_k^{1/2}(\Gamma),m)$. Recall that $\bar{p}_G = \frac{p_G + p_{G_0}}{2}$, under our convention for the Hellinger distance,  $\|\sqrt{\overline{p}_{G_1}} - \sqrt{\overline{p}_{G_2}}\|$. Moreover, the map \(p\mapsto(p+p_{G_0})/2\) contracts the squared Hellinger distance $h^2\left(\frac{p_{G_0}+p_1}{2},\frac{p_{G_0}+p_2}{2} \right) \leq \frac{1}{2}h^2(p_1,p_2)$, see \cite[Lemma 4.2]{vandeGeer-00}. Consequently, we have 
\begin{equation}
\label{eqn:dung_lemma_5_bounded_for_bracketing_entropy}
       H_B(u, \overline{\mathscr{P}}_k^{1/2}(\Gamma),m)  = H_B(\dfrac{u}{\sqrt{2}}, \overline{\mathscr{P}}_k(\Gamma),h)\leq H_B(u, {\mathscr{P}}_k(\Gamma),h) {\lesssim} \log(1/u),
\end{equation}
    where the final inequality follows from Lemma \ref{lemma:bracket_entropy_estimation}. 
Substituting the entropy bound in equation~\eqref{eqn:dung_lemma_5_bounded_for_bracketing_entropy} into the definition of the local bracketing entropy integral yields
\begin{equation*}
    \mathcal{J}_B\left(\varepsilon,  \overline{\mathscr{P}}_k^{1/2}(\Gamma,\varepsilon),m\right) \lesssim \left(\int_{\varepsilon^2/2^{13}}^{\varepsilon} \sqrt{\log(1/u)}\,du\right)\vee \varepsilon \lesssim \varepsilon\left(\log\dfrac{2^{13}}{\varepsilon^2}\right)^{1/2}.
\end{equation*}
It remains to compare the last expression with \(\sqrt n\,\varepsilon^2\). Since $\varepsilon\geq (\log(n)/n)^{1/2}$, we have for all sufficiently large $n$, $\log(2^{13}/\varepsilon^2) \lesssim n\varepsilon^2$. Therefore 
\begin{equation*}
    \varepsilon\left(\log\dfrac{2^{13}}{\varepsilon^2}\right)^{1/2}\lesssim \sqrt{n}\varepsilon^2.
\end{equation*}
Increasing the constant if necessary also controls the \(\varepsilon\) term in the definition of \(\mathcal J_B\). This proves the upper bound in equation~\eqref{eqn:bracket_entropy_estimation}.
\end{proof}

To control the stochastic term in the likelihood analysis, let \(P_0\) denote the probability measure with density \(p_{G_0}\), and let $P_n:=\frac1n\sum_{i=1}^n\delta_{X_i}$
be the empirical measure associated with an i.i.d. sample \(X_1,\ldots,X_n\sim P_0\). For every \(G\in\mathscr P_k(\Gamma)\), define
$\bar p_G:=\frac{p_G+p_{G_0}}{2}$
and introduce the empirical process
$$ \mu_n(G) := \sqrt n \int \frac12 \mathbf 1_{\{f_{G_0}>0\}} \log\left(\frac{\bar p_G}{p_{G_0}}\right) \,d(P_n-P_0).$$
The following maximal inequality is a direct specialization of \cite[Theorem 5.11 and estimation (7.7)]{vandeGeer-00} to the class of localized mixture densities considered here. 

\begin{lemma}
\label{lemma:tail_of_empirical_process}
Let $R > 0$ and $k \geq 1$ and suppose that $C_1 < \infty$. There exists a universal constant $C_0 > 0$ such that the following assertion holds for every $C \geq C_0$: let $n \in \mathbb{N}$ and $t > 0$ satisfy 
\begin{align}
t &\leq (8\sqrt{n}R) \wedge \left( {C_1 \sqrt{n} R^2} \right), \label{eqn:cond1} \\
t &\geq C^2 (C_1 + 1) \left( R \,\vee\, \int_{t/(2^6 \sqrt{n})}^{R} 
H_B^{1/2}\!\left( \frac{u}{\sqrt{2}}, \, \overline{\mathscr{P}}^{1/2}_k(\Gamma, R), \mu \right) du \right), \label{eqn:cond2}
\end{align}
Then, 
\begin{align}
\label{eqn:concentration_of_empirical_process}
\mathbb{P}_{G_0} \left(
\sup_{ h(\bar{p}_{G},\, f_{G_0}) \leq R}
|\mu_n(G)| \geq t
\right)
\leq C \exp\left(
- \frac{t^2}{C^2 (C_1 + 1) R^2}
\right).
\end{align}
Here, $\widehat{P}_n$ is the empirical distribution based on the sample $X_1,\ldots,X_n$, and $P$ is the measure induced by the density $f_{G_0}$.
\end{lemma}
We now apply this maximal inequality, together with the entropy bound established in Lemma \ref{lemma:dung_density_estimation_bracketing_integrl_bound}, to prove the upper bound in equation~\eqref{prop:model_convergence}. 

\begin{proof}[Proof of Proposition \ref{prop:MLE_estimation}]
\textbf{Step 1: Exponential tail bound.}
We first show that there exist universal constants \(c,C>0\) such that, for every \(\delta\geq\delta_n\), 
\begin{equation}
\label{eqn:tail_estimation}
\sup_{G_0\in\mathscr{P}_k(\Gamma)}\mathbb{P}_{G_0}\left(h(p_{\widehat{G}_n},p_{G_0}) > \delta\right)\leq c\exp\left(-\dfrac{n\delta^2}{c}\right). 
\end{equation}
Here, \(\delta_n\) is chosen so that
$\sqrt n\,\delta_n^2 \geq c_0\Psi(\delta_n)$
for a sufficiently large constant \(c_0\), where \(\Psi\) denotes the relevant local entropy integral. In view of Lemma 5, one may take $\delta_n=M\sqrt{{\log n}/{n}}$
with \(M>0\) sufficiently large. The basic likelihood inequality, together with \cite[Lemmas 4.1 and 4.2]{vandeGeer-00}, gives
$$ \frac1{16} h^2(p_{\widehat G_n},p_{G_0}) \leq h^2(\bar p_{\widehat G_n},p_{G_0}) \leq \frac1{\sqrt n}\mu_n(\widehat G_n).$$
Consequently,
\begin{align} 
\label{eqn:dung_proposition_2_bound_of_hellinger}
\mathbb P_{G_0} \left\{ h(p_{\widehat G_n},p_{G_0})>\delta \right\} \leq \mathbb P_{G_0} \left\{ \sup_{\substack{G\in\mathscr{P}_k(\Gamma)\\ h(\bar p_G,p_{G_0})>\delta/4}} \left[ \mu_n(G)-\sqrt n\,h^2(\bar p_G,p_{G_0}) \right] \geq0 \right\}. 
\end{align}
To control the right-hand side, define
$r_s:=\frac{2^s\delta}{4}$, and $R_s:=2r_s=\frac{2^{s+1}\delta}{4}$, and consider
$$ \mathcal A_s := \left\{ G\in\mathscr{P}_k(\Gamma),r_s<h(\bar p_G,p_{G_0})\leq R_s \right\}. $$
Let \(S\) be the smallest integer such that \(R_S\geq1\). Since the Hellinger distance between probability densities is bounded, the sets \(\mathcal A_0,\ldots,\mathcal A_S\) cover the region $\{G\in\mathscr{P}_k(\Gamma): h(\bar p_G,p_{G_0})>\delta/4\}$ appearing in equation~\eqref{eqn:dung_proposition_2_bound_of_hellinger}. Therefore,
\begin{equation}
\label{eqn:dung_proposition_2_another_bound_of_hellinger}
\begin{aligned}
\mathbb P_{G_0}
\left\{
  h(p_{\widehat G_n},p_{G_0})>\delta
\right\}
&\leq
\sum_{s=0}^S
\mathbb P_{G_0}
\left\{
  \sup_{G\in\mathcal A_s}
  |\mu_n(G)| \geq \sqrt n\,r_s^2
\right\}
\\
&\leq
\sum_{s=0}^S
\mathbb P_{G_0}
\left\{
  \sup_{\substack{
    G\in\mathscr{P}_k(\Gamma)\\
    h(\bar p_G,p_{G_0})\leq R_s
  }}
  |\mu_n(G)| \geq \sqrt n\,r_s^2
\right\}.
\end{aligned}
\end{equation}
Now we apply Lemma \ref{lemma:tail_of_empirical_process} with
$R=R_s$, $t=\sqrt n\,r_s^2 =\frac14\sqrt n\,R_s^2$, and  $C_1=15$. The upper restriction on \(t\) in equation~\eqref{eqn:cond1} follows immediately. Moreover, Lemma \ref{lemma:dung_density_estimation_bracketing_integrl_bound} and the choice of the multiplicative constant in \(\delta_n\) ensure that the entropy condition~\eqref{eqn:cond2} in Lemma \ref{lemma:tail_of_empirical_process} holds uniformly over \(s=0,\ldots,S\). Hence,
\begin{align*} \mathbb P_{G_0} \left\{ \sup_{\substack{G\in\mathscr{P}_k(\Gamma), h(\bar p_G,p_{G_0})\leq R_s}} |\mu_n(G)| \geq\sqrt n\,r_s^2 \right\} &\leq C\exp\left( -c\frac{nr_s^4}{R_s^2} \right)= C\exp(-cnr_s^2). 
\end{align*}
Substituting this bound into equation~\eqref{eqn:dung_proposition_2_another_bound_of_hellinger} yields
\begin{align*} \mathbb P_{G_0} \left\{ h(p_{\widehat G_n},p_{G_0})>\delta \right\} &\leq C\sum_{s=0}^S \exp(-cn4^s\delta^2) \leq C'\exp(-c'n\delta^2), \end{align*}
where the final inequality follows by summing the resulting geometric-type series. This proves exponential tail bound in equation~\eqref{eqn:tail_estimation}.

\textbf{Step 2: Expected estimation error.} 

Using the tail-integral representation and equation \eqref{eqn:tail_estimation}, we obtain
\begin{equation}
\label{eqn:dung_prop2_expectation_bound_of_hellinger}
\begin{aligned}
\mathbb E_{G_0}
\left[h(p_{\widehat G_n},p_{G_0})\right]
&=
\int_0^\infty
\mathbb P_{G_0}
\left\{h(p_{\widehat G_n},p_{G_0})>\delta\right\}
\,d\delta
\\
&\leq
\delta_n
+
C\int_{\delta_n}^\infty e^{-cn\delta^2}\,d\delta
\lesssim
\delta_n+\frac{e^{-cn\delta_n^2}}{\sqrt n}
\lesssim
\sqrt{\frac{\log n}{n}}.
\end{aligned}
\end{equation}
Finally, under the convention $h^2(f,g) = \frac12\int (\sqrt f-\sqrt g)^2\,dx$, the Cauchy–Schwarz inequality gives
\begin{align*}
\|f-g\|_1
&=
\int |\sqrt f-\sqrt g|(\sqrt f+\sqrt g)\,dx
\\
&\leq
\left(\int(\sqrt f-\sqrt g)^2\,dx\right)^{1/2}
\left(\int(\sqrt f+\sqrt g)^2\,dx\right)^{1/2}
\leq 2\sqrt2\,h(f,g).
\end{align*}
Thus, it follows from equation~\eqref{eqn:dung_prop2_expectation_bound_of_hellinger} that
$$ \sup_{G_0\in\mathscr{P}_k(\Gamma)} \mathbb E_{G_0} \left[ \|p_{\widehat G_n}-p_{G_0}\|_1 \right] \lesssim \sqrt{\frac{\log n}{n}}. $$
This completes the proof.

\end{proof}

\section{Auxiliary Results}
\label{sec:auxiliary_results}
\subsection{Separation Lemmas}
\label{sec:matrix_utils_separation_lemmma}
In this section, we present auxiliary separation results for multivariate
polynomials. These results are adapted from \cite[Lemma A.8]{wei2023minimum} and generalize the arguments in \cite[Appendix D]{heinrich2018} to the multivariate setting.

\begin{lemma}
\label{lemma:rank_matrix}
    Let $j$, $h$, $h_1,\ldots,h_j$ be positive integers satisfying $\sum_{i=1}^j h_i = h$. Consider $\theta_1,\ldots,\theta_j \in \mathbb{R}^{d}$ be pairwise distinct. Let
    \begin{equation*}
        \mathcal{I} = \{(i,\boldsymbol{\ell}) \in \mathbb{N}\times \mathbb{N}^d: \ 1\leq i\leq j,\ \boldsymbol{\ell} \succeq \boldsymbol{0}, |\boldsymbol{\ell}| < h_i\}. 
    \end{equation*}
    For each $(i,\boldsymbol{\ell}) \in \mathcal{I}$, we define the column vector of dimension $\binom{h+d-1}{d}$ indexed by  $\boldsymbol{p} \in \mathbb{N}^d$ with $|\boldsymbol{p}|< h$, by
    \begin{equation*}
        a_{i,\boldsymbol{\ell}}[\boldsymbol{p}] = \dfrac{\theta_i^{\boldsymbol{p}-\boldsymbol{\ell}}}{(\boldsymbol{p}-\boldsymbol{\ell})!}\boldsymbol{1}_{\boldsymbol{p} \succeq \boldsymbol{\ell}}, 
    \end{equation*}
    where inequalities between multi-indices are understood componentwise.

   Stack these vectors as columns to form the matrix 
    \begin{equation*}
        A:= A(\theta_1,\ldots,\theta_j) = [\underbrace{a_{1,\boldsymbol{\ell}}}_{{\binom{h+d-1}{d}} \text{ columns}} |\cdots|\underbrace{a_{j,\boldsymbol{\ell}}}_{{\binom{h+d-1}{d}} \text{ columns}}] \in \mathbb{R}^{\gamma_1 \times \gamma_2}
    \end{equation*}
    where 
\begin{equation*}
    \gamma_1 = {\binom{h+d-1}{d}}, \ \gamma_2 = \sum_{i=1}^j {\binom{h+d-1}{d}}.
\end{equation*}
Then, $\mathrm{rank}(A(\theta_1,\ldots,\theta_j)) = \gamma_2$. 
\end{lemma}

\begin{proof}
    Consider a $\gamma_2$-dimension vector $\Lambda = (\lambda_{i,\boldsymbol{\ell}})_{(i,\boldsymbol{\ell})\in \mathcal{I}}$ satisfying $A\Lambda = 0$, we need to prove that $\Lambda = 0$. We calculate 
    \begin{equation*}
        (A\Lambda)_{\boldsymbol{p}} = \sum_{(i,\boldsymbol{\ell}) \in \mathcal{I}} \lambda_{i,\boldsymbol{\ell}} a_{i,\boldsymbol{\ell}}[\boldsymbol{p}] = \sum_{(i,\boldsymbol{\ell}) \in \mathcal{I}} \lambda_{i,\boldsymbol{\ell}}\dfrac{\theta_i^{\boldsymbol{p}-\boldsymbol{\ell}}}{(\boldsymbol{p}-\boldsymbol{\ell})!}\boldsymbol{1}_{\boldsymbol{\ell} \preceq\boldsymbol{p}} = 0.
    \end{equation*}
    Thus, for a $d$-variate polynomial $P(x) = \sum^{|\boldsymbol{p}|<h}_{\boldsymbol{0}\preceq\boldsymbol{p}}c_{\boldsymbol{p}}x^{\boldsymbol{p}}/\boldsymbol{p}!$, writing $c = (c_{\boldsymbol{p}})_{\boldsymbol{0}\preceq\boldsymbol{p},|\boldsymbol{p}|<h}$, we have 
    \begin{align*}
    c \cdot A\Lambda &= \sum_{\boldsymbol{0}\preceq \boldsymbol{p},|\boldsymbol{p}|\leq h} c_{\boldsymbol{p}}\sum_{(i,\boldsymbol{\ell}) \in \mathcal{I}} \lambda_{i,\boldsymbol{\ell}}\dfrac{\theta_i^{\boldsymbol{p}-\boldsymbol{\ell}}}{(\boldsymbol{p}-\boldsymbol{\ell})!}\boldsymbol{1}_{\boldsymbol{\ell} \preceq\boldsymbol{p}} = \sum_{(i,\boldsymbol{\ell}) \in \mathcal{I}} \lambda_{i,\boldsymbol{\ell}} \sum_{\boldsymbol{0}\preceq \boldsymbol{p},|\boldsymbol{p}|\leq h} c_{\boldsymbol{p}}\dfrac{\theta_i^{\boldsymbol{p}-\boldsymbol{\ell}}}{(\boldsymbol{p}-\boldsymbol{\ell})!}\boldsymbol{1}_{\boldsymbol{\ell} \preceq\boldsymbol{p}}\\
    &= \sum_{(i,\boldsymbol{\ell}) \in \mathcal{I}} \lambda_{i,\boldsymbol{\ell}} P^{(\boldsymbol{\ell})}(\theta_i). 
    \end{align*}
    The remaining part is to prove that each coefficient $\lambda_{i,\boldsymbol{\ell}}$ is equal to 0. This can be done by  finding suitable polynomial $P$. 
 Consider any $i \in [1,j]$, noting that $\theta_1,\ldots,\theta_j$ are different, for each $i'\neq i$, there exists a coordinate $q_{i'}$ such that $\theta^{q_{i'}}_i \neq \theta^{q_{i'}}_{i'}$. Choosing the polynomial 
    \begin{equation*}
        P(x) = (x-\theta_i)^{\boldsymbol{p}}\prod_{\substack{i'=1\\i'\neq i}}(x_{q_i'}-\theta^{q_{i'}}_{i'})^{h_{i'}}
    \end{equation*}
    for $\boldsymbol{p}\succeq\boldsymbol{0}$ and $|\boldsymbol{p}| < h_i$. Then, it is obvious to verify that
    \begin{equation*}
        P^{(\boldsymbol{\ell})}(\theta_{i'}) = 0,\ \forall \boldsymbol{\ell}\succeq \boldsymbol{0}, \ |\boldsymbol{\ell}| < h_{i'}.
    \end{equation*}
    In addition, by successively substituting $\boldsymbol{p}$ from $|\boldsymbol{p}| = h_i-1$ down to $|\boldsymbol{p}|=0$ and using Lemma \ref{lemma:derivative_lemma}, we get $\lambda_{i,\boldsymbol{\ell}} = 0$. This completes our proof. 
\end{proof}

\begin{lemma}
\label{lemma:derivative_lemma}
Let $P(x) = (x-a)^{\boldsymbol{p}}Q(x)$, where $P$ and $Q$ are two $d$-variable polynomials. For $\boldsymbol{\ell}$, suppose that there exists a coordinate $i$ such that $\ell_i < p_i$, then $P^{(\boldsymbol{\ell})}(a) = 0$. 
\end{lemma}
\begin{proof}
    From the formulation of $P$, we have $P(x) = (x_i-a_i)^{p_i}\tilde{Q}(x)$. By applying the derivative to the variable $x_i$ for $\ell_i$ times, we get the zero value. From this, we get  $P^{(\boldsymbol{\ell})}(a) = 0$. 
\end{proof}

\begin{corollary}
\label{coro:epsilon_separation}
    Consider $A(\theta_1,\ldots,\theta_j)$ as defined in Lemma \ref{lemma:rank_matrix}. For $\epsilon > 0$, we define the set of $\epsilon$-separated vectors in $\Theta^j$ by 
    \begin{equation*}
        \Theta^{j}_\epsilon = \{(\theta_i)_{1\leq i\leq j}:\forall i\neq i',\ \|\theta_i-
        \theta_j\|\geq \epsilon\}. 
    \end{equation*}
    Then, for any vector $\Lambda \in \mathbb{R}^{\gamma_2}$ and any vector $(\theta_i)_{1\leq i\leq j} \in \Theta^{j}_{\epsilon}$, 
    \begin{equation*}
        \|A(\theta_1,\ldots,\theta_j)\Lambda\| \asymp_{\varepsilon,d} \|\Lambda\|.
    \end{equation*}
\end{corollary}
\begin{proof}
    Consider continuous function $\Phi: \Theta^{j}_{\epsilon} \times \mathbb{S}^{\gamma_2-1} \to \mathbb{R}$,   $((\theta_1,\ldots,\theta_j),\Lambda) \mapsto  \|A(\theta_1,\ldots,\theta_j)\Lambda\|$ defined in a compact set. From its continuity, $\Phi$ attains the minimum and maximum value in this set at $((\theta_{1*},\ldots,\theta_{j*}),\Lambda_*)$ and $((\theta_{1}^*,\ldots,\theta_{j}^*),\Lambda^*)$ respectively. From Lemma \ref{lemma:rank_matrix}, we see that the $\Phi((\theta_{1*},\ldots,\theta_{j*}),\Lambda_*) > 0$. As a result, we see that $\Phi((\theta_{1},\ldots,\theta_{j}),\Lambda) \asymp 1 = \|\Lambda\|$ for $\Lambda \in \mathbb{S}^{\gamma_2-1}$. Thus, we have 
    \begin{equation*}
        \Phi((\theta_{1},\ldots,\theta_{j}),\Lambda) \asymp \|\Lambda\|, \ \forall \Lambda \in \mathbb{R}^{\gamma_2}. 
    \end{equation*}
\end{proof}

\begin{lemma}(Adapted from \cite[Lemma 11]{nguyen2026geometry})
\label{lemma:dung_separation}
Suppose that $G_0 = \sum_{\pi_i}^{k_0}\pi_{i0}\delta_{\gamma_{i0}}$ and $G = \sum_{\pi_i}^{k}\pi_{i}\delta_{\gamma_{i}}$ be two discrete measures ($k\geq k_0$), such that 
\begin{enumerate}
    \item (Non-vanishing mass) $\pi_{\min} = \min_{1\leq i\leq k_0}\pi_{0i} > 0$. 
    \item (Separation condition) For each $1\leq i\neq j \leq k_0$, we have $\|\gamma_{0i} - \gamma_{0j}\| \geq 2\delta$. 
\end{enumerate}
Then, if $W_1(G,G_0) < \pi_{\min}\delta/2$, for each $j \in [k_0]$, there exists $c(j) \in [k]$ such that $\|\gamma_{0j} - \gamma_{j}\| < \delta/2$. Moreover, the map $j \mapsto c(j)$ is injective. 
\end{lemma}

\begin{proof}
Let $\mathbf Q$ be an arbitrary coupling of $G_0$ and $G$. Fix
$j\in[k_0]$ and suppose, by contradiction, that
$\|\gamma_i-\gamma_{0j}\|\geq \frac{\delta}{2}$   for every $i\in[k]$. Since the first marginal of $\mathbf Q$ is $G_0$, the total mass
transported from $\gamma_{0j}$ is $\pi_{0j}$. Separation condition implies that every portion of this mass must be transported over a distance of at least $\delta/2$.
Consequently,
\begin{align*}
    \int_{\Gamma\times\Gamma}
    \|\gamma-\gamma'\|\,d\mathbf Q(\gamma,\gamma')
    \geq
    \int_{\{\gamma_{0j}\}\times\Theta}
    \|\gamma-\gamma'\|\,d\mathbf Q(\gamma,\gamma') 
    \geq
    \frac{\delta}{2}\,
    \mathbf Q\bigl(\{\gamma_{0j}\}\times\Theta\bigr) =
    \frac{\pi_{0j}\delta}{2}
    \geq
    \frac{\pi_{\min}\delta}{2}.
\end{align*}
Because this inequality holds for every coupling $\mathbf Q$ of
$G_0$ and $G$, taking the infimum over all such couplings gives $$W_1(G,G_0)\geq\frac{\pi_{\min}\delta}{2},$$
contradicting the assumption
$W_1(G,G_0)<\pi_{\min}\delta/2$. Therefore, for every $j\in[k_0]$, there exists an index
$c(j)\in[k]$ such that $\|\gamma_{c(j)}-\gamma_{0j}\|<\frac{\delta}{2}$. In addition, suppose that there exist two distinct indices $j$ and $j'$ such that $\ell = c(j) = c(j')$, then $\|\gamma_{\ell} - \gamma_{0j}\| < \frac{\delta}{2}$ and $\|\gamma_{\ell} - \gamma_{0j'}\| < \frac{\delta}{2}$, thus by triangle inequality, 
\begin{equation*}
    \|\gamma_{0j} - \gamma_{0j'}\| \leq \|\gamma_{\ell} - \gamma_{0j}\| + \|\gamma_{\ell} - \gamma_{0j'}\| \leq \delta,
\end{equation*}
which is a contradiction to the fact that $\min_{i\neq j }\|\gamma_{0i} - \gamma_{0j}\| \geq 2\delta$. Thus the map $j \mapsto c(j)$ is injective. 
\end{proof}

\subsection{Proof of Proposition \ref{dung:prop_strong_identifiability}}
\label{dung:proof_prop_strong_identifiability}
Before delving into detail of the proof of Proposition \ref{dung:prop_strong_identifiability}, we state the following Lemma. Indeed, this result is a generalization for the main component of the demonstration of \cite[Lemma 3]{Chen1995}. While the idea is straightforward, we present here for completeness

\begin{lemma}
\label{dung:lemma_independent_polynomial_exponential}
    Let $\theta_1 > \ldots > \theta_m$ be distinct reals. Suppose that there exist complex numbers $c_{u,v}$ ($0\leq u \leq p$, $1\leq v\leq m$) such that 
    \begin{equation}
        \label{dung:lemma_auxilliary_independent}
        \sum_{v=1}^m\left[c_{0,v} + \cdots c_{p,v}(it)^{p}\right]\exp(it\theta_v) = 0
    \end{equation}
    for all $t \in \mathbb{R}$. Then, it is necessary that all $c_{u,v} = 0$ for $0\leq u \leq p$ and $1\leq v\leq m$. 
\end{lemma}
\begin{proof}[Proof of Lemma \ref{dung:lemma_independent_polynomial_exponential}]
Multiplying both sides of equation~\eqref{dung:lemma_auxilliary_independent} with $e^{-t^2/2}$ and applying inverse Fourier transformation, we get 
\begin{equation}
\label{eqn:dung_lemma_appendix_fourier_transform}
    \sum_{v=1}^m\left[\sum_{u = 0}^p (-1)^{u}c_{u,v}H_u(t-\theta_v)\right]\exp\left(-\frac{(t-\theta_v)^2}{2}\right) = 0,
\end{equation}
where $H_u$ denotes the $u$-th order Hermite polynomial. Multiplying both sides of equation~\eqref{eqn:dung_lemma_appendix_fourier_transform} with $\exp\left(\frac{(t-\theta_1)^2}{2}\right)$ and letting $t \to \infty$, noting that for any polynomial $P$ and index $v>1$, 
\begin{equation*}
    P(t)\cdot\exp\left(-\frac{(t-\theta_v)^2-(t-\theta_1)^2}{2}\right) \to 0 \text{ for } t \to \infty, 
\end{equation*}
we achieve that 
\begin{equation*}
    \sum_{u = 0}^p (-1)^{u}c_{u,1}H_u(t-\theta_1) \to 0
\end{equation*}
for $t \to \infty$. This only happens when $c_{u,1} = 0$ for all $0\leq u \leq p$. Using a similar argument, we obtain that $c_{u,v} = 0$ for $0\leq u \leq p$ and $1\leq v\leq m$. 
\end{proof}
\begin{proof}[Return to Proof of Proposition \ref{dung:prop_strong_identifiability}]
\textbf{Step 1: } First, we prove the sufficient condition for identifiability of location family of distribution.

Suppose that there exist constants $c_{j,\boldsymbol{\alpha}}$ such that
\begin{equation}
\label{eqn:dung_condition_of_prop_3}
    \sum_{j=1}^m
    \sum_{|\boldsymbol{\alpha}|\leq p}
    c_{j,\boldsymbol{\alpha}}
    D_\gamma^{\boldsymbol{\alpha}} F(x\mid\gamma_j)
    =0,
\end{equation}
Taking derivative $\frac{\partial^d}{\partial x_1\ldots \partial x_d}$ on both sides of equation~\eqref{eqn:dung_condition_of_prop_3}, we have 
\begin{equation}
\label{eqn:dung_proof_of_prop3_identifiability_condition}
    \sum_{j=1}^m
    \sum_{|\boldsymbol{\alpha}|\leq p}
    c_{j,\boldsymbol{\alpha}}
    D_\theta^{\boldsymbol{\alpha}} f(x\mid\theta_j)
    =0
\end{equation}
Let $\varphi(t)
    :=
    \int_{\mathbb{R}^d}
    \exp(\mathrm{i}t^\top x)f(x)\,dx$, $t\in\mathbb{R}^d$
be the characteristic function associated with $f$. Then, the characteristic
function corresponding to the translated density $f(\cdot-\theta)$ is $\exp(\mathrm{i}t^\top\theta)\varphi(t)$.
Furthermore, for every multi-index $\boldsymbol{\alpha}\in\mathbb{N}^d$,
\[
    D_\theta^{\boldsymbol{\alpha}}
    \left[
        \exp({i}t^\top\theta)\varphi(t)
    \right]
    =
    ({i}t)^{\boldsymbol{\alpha}}
    \exp({i}t^\top\theta)\varphi(t),
\]
where $({i}t)^{\boldsymbol{\alpha}} :=
    \prod_{\ell=1}^d
    ({i}t_\ell)^{\alpha_\ell}$. Taking the Fourier transform of equation~\eqref{eqn:dung_proof_of_prop3_identifiability_condition}, we have 
\[
    \varphi(t)
    \sum_{j=1}^m
    \sum_{|\boldsymbol{\alpha}|\leq p}
    c_{j,\boldsymbol{\alpha}}
    ({i}t)^{\boldsymbol{\alpha}}
    \exp({i}t^\top\theta_j)
    =0.
\]
Since $\varphi(0)=1$ and $\varphi$ is continuous, there exists an open
neighborhood $\mathcal{N}$ of the origin such that
$\varphi(t)\neq 0$ for every $t\in\mathcal{N}$. Therefore,
\begin{equation}
\label{eqn:dung_prop_identifiabilty_another_form}
    \sum_{j=1}^m
    P_j(t)\exp({i}t^\top\theta_j)
    =0,
    \qquad t\in\mathcal{N}
\end{equation}
where $P_j(t) := \sum_{|\boldsymbol{\alpha}|\leq p} c_{j,\boldsymbol{\alpha}}({i}t)^{\boldsymbol{\alpha}}$
is a polynomial of degree at most $p$. As polynomial function is also analytic, this polynomial vanishes in all $\mathbb{R}^d$. Choose $v\in\mathbb{R}^d$ such that $v^\top\theta_1,\ldots,v^\top\theta_m$
are pairwise distinct. The set of such vectors is open and dense in
$\mathbb{R}^d$, since its complement is contained in the finite union
of hyperplanes $\bigcup_{j\neq k}
    \left\{
        v\in\mathbb{R}^d:
        v^\top(\theta_j-\theta_k)=0
    \right\}$. Setting $t=sv$ for $s$ in a sufficiently small neighborhood of zero, we obtain
\[
    \sum_{j=1}^m
    P_j(sv)
    \exp\bigl({i}s v^\top\theta_j\bigr)
    =0.
\]
For each $j$, write $P_j(sv) = \sum_{r=0}^p s^r Q_{j,r}(v)$,
where $Q_{j,r}(v) := {i}^r \sum_{|\boldsymbol{\alpha}|=r} c_{j,\boldsymbol{\alpha}}v^{\boldsymbol{\alpha}}$.
It follows that
\[
    \sum_{j=1}^m
    \sum_{r=0}^p
    Q_{j,r}(v)s^r
    \exp\bigl({i}s v^\top\theta_j\bigr)
    =0.
\]
Since $v^\top\theta_1,\ldots,v^\top\theta_m$ are pairwise distinct,
the functions $\left\{
        s^r\exp\bigl({i}s v^\top\theta_j\bigr):
        j\in[m],\ 0\leq r\leq p
    \right\}$
are linearly independent by Lemma \ref{dung:lemma_independent_polynomial_exponential}. Consequently, $Q_{j,r}(v)=0$ 
for every $j\in[m]$ and every $0\leq r\leq p$. The preceding argument applies to every $v$ in an open dense subset of
$\mathbb{R}^d$, and since each $Q_{j,r}$ is a polynomial, it follows that $Q_{j,r}(v)=0$
for every $v\in\mathbb{R}^d$. Hence,
\[
    \sum_{|\boldsymbol{\alpha}|=r}
    c_{j,\boldsymbol{\alpha}}v^{\boldsymbol{\alpha}}=0
    \qquad
    \text{for every }v\in\mathbb{R}^d.
\] 
Since a polynomial that vanishes identically has all of its
coefficients equal to zero, we conclude that $c_{j,\boldsymbol{\alpha}}=0$ for every $j\in[m]$ and every $\boldsymbol{\alpha}\in\mathbb{N}^d$ satisfying
$|\boldsymbol{\alpha}|\leq p$. 

\textbf{Step 2: } For uniformly modulus condition, from the assumption about location-scale distribution, $ D_{\theta}^{\boldsymbol{\alpha}} F(x\mid\theta) = (-1)^{|\boldsymbol{\alpha}|}D_{x}^{\boldsymbol{\alpha}} F(x-\theta\mid 0)$, and thus for each multi-index $\boldsymbol{\alpha}$, 
\[
    \sup_{x\in\mathbb{R}^d}
    \left|D^{\boldsymbol{\alpha}} F(x\mid\theta)\right|
    =
    \sup_{x\in\mathbb{R}^d}
    \left|D^{\boldsymbol{\alpha}} F(x\mid 0)\right|.
\]
Therefore, the supremum does not depend on $\theta$. Now, for any multi-index $\boldsymbol{\alpha}$ such that $|\boldsymbol{\alpha}| = p+1$, let $S(\boldsymbol{\alpha}):=\{i\in [d]:\alpha_i > 0\}$ be the position such that $\alpha_i >0$, and let $\boldsymbol{\beta}:= \boldsymbol{\alpha} - \boldsymbol{1}_{S(\boldsymbol{\alpha})}$. 
Differentiating the multivariate CDF under the integral sign gives
\begin{equation*}
    D^{\boldsymbol{\alpha}} F(x\mid 0)
    =
    \int_{(-\infty,x_{S(\boldsymbol{\alpha})^c}]}
    D^{\boldsymbol{\beta}} f\bigl(x_{S(\boldsymbol{\alpha})},y_{S(\boldsymbol{\alpha})^c}\bigr)
    \,\mathrm{d}y_{S(\boldsymbol{\alpha})^c}.
\end{equation*}
Therefore, $\sup_{x \in \mathbb{R}^d}|D^{\boldsymbol{\alpha}}F(x\mid \gamma)| < \infty$ for each $|\boldsymbol{\alpha}| = p+1$ follows if equation~\eqref{eqn:dung_equivalent_form_for_continuous_modulus} holds. This completes the proof.
\end{proof}

\bibliographystyle{abbrv}
\bibliography{references}

\end{document}